\documentclass[11pt,dvipsnames]{amsart}

\usepackage[utf8]{inputenc}
\usepackage{lmodern}
\usepackage[T1]{fontenc}
\usepackage[english]{babel}
\usepackage{fullpage}

\usepackage{latexsym}
\usepackage{amsmath,amssymb,amsthm}
\usepackage{mathrsfs}
\usepackage{stmaryrd}

\usepackage{enumerate}

\usepackage{rotating}

\usepackage{todonotes}

\usepackage{tasks}

\usepackage{tikz-cd}
\usetikzlibrary{shapes}
\usepackage{tikz}
\usetikzlibrary{decorations,shapes,decorations.markings,decorations.pathreplacing}
\tikzset{
		OpenCirc/.style={
			label={[inner sep=0,minimum size=10mm, circle, draw, densely dotted]center:{}}
		},
		closed/.style = {decoration = {markings, mark = at position 0.5 with { \node[transform shape, xscale = .8, yscale=.4] {/}; } }, postaction = {decorate} },
		UpperBar/.style={decoration = {markings, mark = at position 0.25 with{ \arrow[line width=1pt]{|}
				} }, postaction={decorate} },
		LowerBar/.style={decoration = {markings, mark = at position 0.75 with { \arrow[line width=1pt]{|} } }, postaction={decorate} },
		kpm/.style={decoration = {markings, mark = at position 0.75 with {\node[draw=none, below =-1cm ]{$(k^+(u),k^-(u))$}} }, postaction={decorate} },
		10/.style={decoration = {markings, mark = at position 0.75 with { $(1,0)$ } }, postaction={decorate} },
		27/.style={decoration = {markings, mark = at position 0.75 with { $(2,7)$ } }, postaction={decorate} },
		Bullet/.style={circle, fill, draw, inner sep=0, minimum size=4pt}
}
\tikzcdset{scale cd/.style={every label/.append style={scale=#1},
		cells={nodes={scale=#1}}}}
	
\makeatletter
\DeclareRobustCommand{\rvdots}{%
	\vbox{
		\baselineskip4\p@\lineskiplimit\z@
		\kern-\p@
		\hbox{.}\hbox{.}\hbox{.}
}}
\makeatother

\usepackage{wasysym}

\usepackage{color}
\usepackage{hyperref}
\hypersetup{
	colorlinks=true,
	linkcolor=blue,
	citecolor=blue,
	urlcolor=blue
}
\usepackage{enumitem}

\usepackage{xcolor}

\usepackage{soul}
\setstcolor{red}

\newtheorem{theoremAlph}{Theorem}

\newtheorem*{theorem*}{Theorem}

\newtheorem{theorem}{Theorem}[section]
\newtheorem{lemma}[theorem]{Lemma}	
\newtheorem{proposition}[theorem]{Proposition}
\newtheorem{corollary}[theorem]{Corollary}
\theoremstyle{definition}
\newtheorem{definition}[theorem]{Definition} 
\newtheorem{remark}[theorem]{Remark}	

\newtheorem{question}[theorem]{Question}

\theoremstyle{definition} 
\newtheorem*{ack}{Acknowledgements}

\numberwithin{equation}{section}

\newcommand{\C}{\mathbb{C}} 
\newcommand{\R}{\mathbb{R}} 
\newcommand{\Q}{\mathbb{Q}} 
\newcommand{\Z}{\mathbb{Z}} 
\newcommand{\N}{\mathbb{N}} 
\newcommand{\Quat}{\mathbb{H}}

\newcommand{\SO}{\mathrm{SO}}

\newcommand{\II}{\mathrm{I\!I}}

\DeclareMathOperator{\im}{im}
\DeclareMathOperator{\intr}{Int}
\newcommand{\bq}{/ \hspace{-.12cm} /}
\newcommand{\bigslant}[2]{{\raisebox{.2em}{$#1$}\left/\raisebox{-.2em}{$#2$}\right.}}
\newcommand{\vbigslant}[2]{{\raisebox{0em}{$#1$}\left/\raisebox{-.2em}{$#2$}\right.}}

\makeatletter
\DeclareRobustCommand*\uell{\mathpalette\@uell\relax}
\newcommand*\@uell[2]{
	\setbox0=\hbox{$#1\ell$}
	\setbox1=\hbox{\rotatebox{10}{$#1\ell$}}
	\dimen0=\wd0 \advance\dimen0 by -\wd1 \divide\dimen0 by 2
	\mathord{\lower 0.1ex \hbox{\kern\dimen0\unhbox1\kern\dimen0}}
}

\begin{document}
	\title[\rmfamily Metrics of Positive Ricci Curvature on Simply-Connected Manifolds of Dimension $6k$]{\rmfamily Metrics of Positive Ricci Curvature on Simply-Connected Manifolds of Dimension $6k$}
	\date{\today}
	\author{Philipp Reiser$^*$}
	\address{Department of Mathematics\\University of Fribourg\\Switzerland.}
	\email{\href{mailto:philipp.reiser@unifr.ch}{philipp.reiser@unifr.ch}}
	\thanks{$^{*}$The author acknowledges funding by the Deutsche Forschungsgemeinschaft (DFG, German Research Foundation) -- 281869850 (RTG 2229).}
	
	\begin{abstract}
		A consequence of the surgery theorem of Gromov and Lawson is that every closed, simply-connected 6-manifold admits a Riemannian metric of positive scalar curvature. For metrics of positive Ricci curvature it is widely open whether a similar result holds; there are no obstructions known for those manifolds to admit a metric of positive Ricci curvature, while the number of examples known is limited. In this article we introduce a new description of certain $6k$-dimensional manifolds via labeled bipartite graphs and use an earlier result of the author to construct metrics of positive Ricci curvature on these manifolds. In this way we obtain many new examples, both spin and non-spin, of $6k$-dimensional manifolds with a metric of positive Ricci curvature.
	\end{abstract}

	\maketitle
	
	\section{Introduction and Main Results}

		Determining which manifolds admit a complete Riemannian metric of positive Ricci curvature remains a long standing open problem. In this article we consider this problem for closed and simply-connected manifolds. Under this assumption, the only known obstructions for the existence of a Riemannian metric of positive Ricci curvature already vanish for the (potentially) weaker condition of positive scalar curvature. In particular, it is an open question whether any closed, simply-connected manifold with a Riemannian metric of positive scalar curvature also admits a Riemannian metric of positive Ricci curvature.
		
		When considering low dimensions, the first non-trivial case is dimension 4 as any closed, simply-connected manifold in dimension 2 and 3 is a standard sphere. Here it was shown by Sha and Yang \cite{SY93} that any closed, simply-connected 4-manifold that admits a metric with positive scalar curvature is homeomorphic to a manifold with positive Ricci curvature, and in particular completely classified the intersection forms that can be realized by such a manifold. Further, in dimension 5, Sha and Yang \cite{SY91} proved that any closed, simply-connected 5-manifold with torsion-free homology admits a metric of positive Ricci curvature by constructing such metrics on connected sums of sphere bundles over spheres. In this article we consider dimension 6 and construct new examples of closed, simply-connected 6-manifolds with torsion-free homology that admit a metric of positive Ricci curvature. Moreover, we extend our construction to all dimensions $6k$, $k\in\N$.
	
	As a result of the surgery theorem of Gromov and Lawson \cite{GL80a}, \emph{any} closed, simply-connected 6-manifold admits a metric of positive scalar curvature. For positive Ricci curvature, while there is no obstruction known for the existence of metrics with this curvature condition, there are only relatively few known examples, see \cite[Section 5.1]{Re23}. In particular, the only examples where we have arbitrarily large Betti numbers are connected sums of sphere bundles.
	
	The main tool to construct metrics of positive Ricci curvature will be the theorem below. It constructs so-called \emph{core metrics}, a notion introduced by Burdick \cite{Bu19}. The main property of a core metric is that it has positive Ricci curvature and that the connected sum of manifolds with core metrics admits a metric of positive Ricci curvature, see Subsection \ref{SS:CORE_METRICS} below.
	\begin{theorem}[{\cite[Theorem B]{Re23}}]
		\label{T:PLUMBING}
		Let $W$ be the manifold obtained by plumbing linear disc bundles $\overline{E}_i\to B_i$, $1\leq i\leq k$, with compact base manifolds according to a simply-connected graph. If all $B_i$ admit a core metric with $\dim(B_1)\geq3$ and the fiber dimension of $\overline{E}_1$ is at least 4, then $\partial W$ admits a core metric.
	\end{theorem}

	To state our result, given a closed, simply-connected and oriented $6k$-dimensional manifold $M$ with torsion-free homology, we have a symmetric trilinear form $\mu_M\colon H^{2k}(M)\times H^{2k}(M)\times H^{2k}(M)\to\Z$ defined by
	\[\mu_M(x,y,z)=\langle x\smile y\smile z,[M]\rangle.\]
	Further invariants we will consider are the $k$-th power of the second Stiefel-Whitney class
	\[w_2(M)^k\in H^{2k}(M;\Z/2)\cong H^{2k}(M)\otimes\Z/2\]
	and the $k$-th Pontryagin class
	\[p_k(M)\in H^{4k}(M)\cong \mathrm{Hom}(H^{2k}(M),\Z).\]
	In particular, these invariants are all defined on the cohomology group $H^{2k}(M)$. For $k=1$, by the classification of Jupp \cite{Ju73}, see also Theorem \ref{T:JUPP}, these invariants already determine the diffeomorphism type of $M$ up to connected sums with copies of $S^3\times S^3$. Given a finitely generated free abelian group $H$, a symmetric trilinear form $\mu$ on $H$, an element $w\in H\otimes \Z/2$ and a linear form $p$ on $H$, we call the system $(H,\mu,w,p)$ \emph{admissible in dimension $6k$}, if it can be realized as the invariants of a closed, simply-connected $6k$-dimensional manifold with torsion-free homology.
	
	Let $G=(U,V,E,(\alpha,k^+,k^-))$ be a bipartite graph, where $U$ and $V$ are the sets of vertices and $E\subseteq U\times V$ is the set of edges, with a labeling $(\alpha,k^+,k^-)\colon U\to \Z\times \N_0^2$ for vertices in $U$. We call such a graph an \emph{algebraic plumbing graph}. We draw vertices $u\in U$ as follows:
		
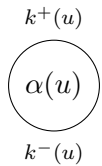
\[
		\begin{tikzpicture}
			\node[circle,draw,label=below:${\scriptstyle k^{-}(u)}$,label=above:${\scriptstyle k^{+}(u)}$] (A) at (0,0) {$\alpha(u)$};
		\end{tikzpicture}
		\]
		If one of $k^+(u)$ and $k^-(u)$ vanishes, then we will omit it. Vertices in $V$ will simply be drawn as dots. An example for such a graph is given as follows:
		
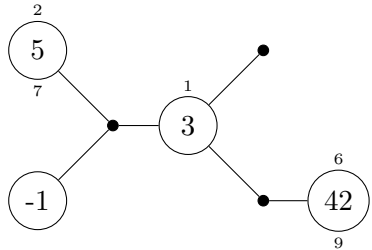
\[
		\begin{tikzpicture}
			\begin{scope}[every node/.style={circle,draw,minimum height=2em}]
				\node[Bullet] (V1) at (0,0) {};
				\node[Bullet] (V2) at (2,1) {};
				\node[Bullet] (V3) at (2,-1) {};
				\node[label={[label distance=-0.25cm]below:${\scriptscriptstyle 7}$},label={[label distance=-0.25cm]above:${\scriptscriptstyle 2}$}] (U1) at (-1,1) {5};
				\node (U2) at (-1,-1) {-1};
				\node[label={[label distance=-0.25cm]above:${\scriptscriptstyle 1}$}] (U3) at (1,0) {3} ;
				\node[label={[label distance=-0.25cm]below:${\scriptscriptstyle 9}$},label={[label distance=-0.25cm]above:${\scriptscriptstyle 6}$}] (U4) at (3,-1) {42} ;
			\end{scope}
			\path[-](V1) edge (U1);
			\path[-](V1) edge (U2);
			\path[-](V1) edge (U3);
			\path[-](V2) edge (U3);
			\path[-](V3) edge (U3);
			\path[-](V3) edge (U4);
		\end{tikzpicture}
		\]
		By assigning a suitable disc bundle to each vertex of an algebraic plumbing graph $G$ we will define for each $k\in \N$ a geometric plumbing graph, denoted $\overline{G}^k$, see Definition \ref{D:GEOM_PL_GRAPH} (and note that for $k>2$ there can be multiple such graphs, see Remark \ref{R:preimage_choice}), which in turn defines a $6k$-dimensional manifold $M_{\overline{G}^k}$, see Definition \ref{D:M_G}. Important invariants, such as the cohomology group $H^{2k}(M_{\overline{G}^k})$, the trilinear form $\mu_{M_{\overline{G}^k}}$, and characteristic classes can be computed directly from the data provided by the algebraic plumbing graph if it is simply-connected. For example, if no vertex in $V$ is a leaf, then $H^{2k}(M_{\overline{G}^k})$ has rank $|U|-|V|$ and $M_{\overline{G}^k}$ is spin if and only if $k^-=k^+\equiv0$. The fact that the invariants can be obtained from the graph data if $G$ is simply-connected motivates defining invariants $(H_G,\mu_G^k,w_G,p_G^k)$ in a similar way for any algebraic plumbing graph $G$, see Definition~\ref{D:INV}. We set $\mu_G=\mu_G^1$ and $p_G=p_G^1$.
		
		\begin{theoremAlph}
			\label{T:ALG_PL_GR_CORE1}
			Let $G$ be an algebraic plumbing graph.
			\begin{enumerate}
				\item If $k=1$, then the system of invariants $(H_G,\mu_G,w_G,p_G)$ is admissible in dimension $6$.
				\item If every connected component of $G$ is simply-connected, then the system of invariants $(H_G,\mu_G^k,w_G,p_G^k)$ is admissible in dimension $6k$ and realized by the manifold $M_{\overline{G}^k}$. Further, $M_{\overline{G}^k}$ admits a core metric.
				\item If $k=1$ and every connected component of $G$ is simply-connected, then any closed, simply-connected 6-manifold with torsion-free homology, whose invariants are equivalent to\linebreak $(H_G,\mu_G,w_G,p_G)$, admits a core metric.
			\end{enumerate}
		\end{theoremAlph}

	Since different algebraic plumbing graphs can have equivalent systems of invariants, it is not clear a priori, how large the class of manifolds is that we obtain in this way. To analyze this further, we introduce a reduced form in Subsection \ref{SS:RED_GRAPHS} and conjecture that systems of invariants obtained from different reduced forms are indeed not equivalent, see Question \ref{Q:RED_EQUIV=ISOM?}. The difficulty lies in the problem that in general it is hard to determine whether two given trilinear forms are equivalent or not. We prove the conjecture for graphs $G$ with $\mathrm{rank}(H_G)\leq2$, see Propositions \ref{P:GRAPH_RED1} and \ref{P:GRAPH_RED2}, except for the case where $\mathrm{rank}(H_G)=2$ and $w_G=0$ where we obtain a partial result by using invariant theory of $\mathrm{SL}(2,\C)$. This result is sufficient to show that infinitely many of the graphs in the latter case define new examples of 6-manifolds with a metric of positive Ricci curvature and of 6-manifolds with core metrics, see Remark \ref{R:NEW_EX}. An interesting subfamily of these graphs is given by certain graphs for which the corresponding $6$-manifolds split as a connected sum where one of the summands is a homotopy $\C P^3$, see Proposition \ref{P:HOM_CP3}.
	
	For larger Betti numbers, using Theorem \ref{T:ALG_PL_GR_CORE1}, we have the following result.
	\begin{theoremAlph}
		\label{T:NEW_EX_B2_LARGE}
		For every $k\in\N$ and for every odd $l\in\N$ sufficiently large there exists an infinite family $M_j^{6k}$ of pairwise non-diffeomorphic closed $6k$-dimensional manifolds with torsion-free homology with the following properties:
		\begin{itemize}
			\item $M_j$ is $(2k-1)$-connected with $b_{2k}(M_j)=l$,
			\item $M_j$ does not split non-trivially as a connected sum,
			\item $M_j$ is not diffeomorphic to the total space of a linear sphere bundle, a homogeneous space, a biquotient, a cohomogeneity one manifold or a Fano variety,
			\item $M_j$ admits a core metric.
		\end{itemize}
		Further, if $k=1$ or $k$ is even, then we can replace the condition that $M_j$ is $(2k-1)$-connected by $M_j$ being simply-connected and non-spin.
	\end{theoremAlph}
	It follows that the manifolds $M_j^6$ are new examples of manifolds with a metric of positive Ricci curvature and, to the best of our knowledge, this also holds for the manifolds $M_j^{6k}$.
	
	This article is organized as follows. In Section \ref{S:PREL} we recall core metrics and establish basic results on the topology of certain linear sphere bundle. We proceed in Section \ref{S:PLUMBINGS} by considering plumbing graphs and the topology of the resulting manifolds. Finally, in Section \ref{S:6-MFDS} we consider applications in dimension $6k$. Here we introduce the concept of algebraic plumbing graphs and give the proof of Theorem \ref{T:ALG_PL_GR_CORE1} in Subsection \ref{SS:ALG_PLUMBING}. In Subsection \ref{SS:RED_GRAPHS} we consider reduced forms for algebraic plumbing graphs and give applications in Subsection \ref{SS:ALG_PLUMBING_APPLICATIONS}. In Appendix \ref{A:GRAPHS} we recall basic facts about the adjacency and incidence matrix of a directed graph needed in Section \ref{S:PLUMBINGS}.
	
	\begin{ack}
		This article is based on the author's Ph.D.\ thesis written at the Karlsruhe Institute of Technology (KIT). The author would like to thank his supervisors Wilderich Tuschmann and Fernando Galaz-García for their support. The author would also like to thank them and Anand Dessai for helpful comments on an earlier version of this article and Manuel Krannich for helpful conversations on the homotopy groups of the orthogonal group. Finally, the author would like to thank the anonymous referee for their insightful comments.
	\end{ack}
	
	\section{Preliminaries}
	\label{S:PREL}

	All manifolds and maps between manifolds will be assumed to be smooth. For a manifold $M$ we will write $M^n$ to indicate that $M$ has dimension $n$. By $D^n$ we denote the closed $n$-dimensional disc and by $M\setminus D^n$ the manifold $M$ with \emph{some} embedded disc removed, where we require the embedding to be orientation preserving if $M$ is oriented. If $M$ is connected, then the embedding is unique up to isotopy, see \cite[Theorem 5.5]{Pa59}, so the expression $M\setminus D^n$ is well-defined up to diffeomorphism. If $M$ is oriented, then $-M$ denotes $M$ with the reversed orientation. If $M$ is oriented and has non-empty boundary we use the convention that the induced orientation on $\partial M$ is obtained by inserting the inward normal $\nu$ first, i.e.\ a basis of tangent vectors $(e_1,\dots,e_{n-1})$ tangent to $\partial M$ is oriented in $\partial M$ if $(\nu,e_1,\dots,e_{n-1})$ is oriented in $M$. If not stated otherwise, we use homology and cohomology with coefficients in $\Z$. We will use the convention that $0\not\in\N$ and set $\N_0=\N\cup\{0\}$.
	
	For a closed, oriented and connected manifold $M^n$ we have $H_n(M;\Z)\cong\Z$ and $H_{n-1}(M;\Z)\cong H^1(M;\Z)$ is torsion-free, hence, by the universal coefficient theorem, for a commutative ring $R$,
	\[H_n(M;R)\cong H_n(M;\Z)\otimes R\cong R. \]
	It follows that we obtain a fundamental class $[M;R]\in H_n(M;R)$ from the orientation class $[M;\Z]\in H_n(M;\Z)$, which induces an orientation of $M$ in the homological sense with coefficients in $R$. In particular, Poincaré duality with coefficients in $R$ can be applied for $M$.
	
		\subsection{Core Metrics}
	\label{SS:CORE_METRICS}
	
	In this section we introduce core metrics. We refer to the work of Burdick \cite{Bu19a,Bu19,Bu20,Bu20a} and the author \cite{Re23} for further details.
	
	\begin{definition}[{Burdick \cite{Bu19}, based on work by Perelman \cite{Pe97} }]
		Let $M^n$ be a manifold. A Riemannian metric $g$ on $M$ is called a \emph{core metric} if it has positive Ricci curvature and if there is an embedding $\varphi\colon D^n\hookrightarrow\textup{Int}(M)$ such that the induced metric $g|_{\varphi(S^{n-1})}$ is the round metric of radius 1 and such that the second fundamental form $\II_{\varphi(S^{n-1})}$ is positive semi-definite with respect to the inward pointing normal vector of $S^{n-1}\subseteq D^n$.
	\end{definition}
	
	Our definition differs from Burdick’s original definition, as he requires the second fundamental	form to be positive definite. However, the two definitions are equivalent: a metric with positive semi-definite second fundamental form can always be deformed into a metric with positive definite fundamental form while keeping the Ricci curvature positive, e.g.\ by a deformation as in \cite[Proposition 1.2.11]{Bu19a}.
	
	There exist different sign conventions for the second fundamental form, here we use the sign convention so that, after possibly rescaling the metric, the round metric on $S^n$ is a core metric, where the embedded disc can be any geodesic ball that contains a hemisphere.
	
	The main interest for core metrics comes from the following fact:
	\begin{theorem}[{\cite[Theorem B]{Bu19}}]
		Let $M^n_i$, $1\leq i\leq k$ be manifolds that admit core metrics. If $n\geq 4$, then $\#_i M_i$ admits a metric of positive Ricci curvature.
	\end{theorem}
	So far, the following closed manifolds are known to admit a core metric:
	\begin{itemize}
		\item $S^n$, if $n\geq2$,
		\item the complex projective space $\C P^n$, the quaternionic projective space $\Quat P^n$ and the Cayley plane $\mathbb{O} P^2$ (see \cite{Pe97} and \cite{Bu19}),
		\item $M_1^n\# M_2^n$, if $n\geq 4$ and $M_1$, $M_2$ admit core metrics (see \cite{Bu20a}),
		\item Total spaces of linear $S^p$-bundles $E^n\to B^q$ if $B$ admits a core metric and $n\geq6$, $p,q\geq2$ (see \cite{Bu20} and \cite{Re23}),
		\item The manifolds obtained in Theorem \ref{T:PLUMBING}.
	\end{itemize}
	
	\subsection{Linear Sphere Bundles}
	\label{SS:LSB}
	
	Let $E\xrightarrow{\pi} B^q$ be a fiber bundle with fiber $F^p$ and structure group $G$. We suppose that $\pi$ is oriented, that is, the manifold $F$ is oriented and the action of $G$ on $F$ is orientation-preserving. If the base $B$ is oriented, then the product orientation on every local trivialization $U\times F$, $U\subseteq B$ open, induces an orientation on the total space $E$.
	
	If $F$ has non-empty boundary and $B$ has no boundary, then $\partial E\xrightarrow{\pi|_{\partial E}}B$ is a fiber bundle with fiber $\partial F$ and structure group $G$. We have two orientations on $\partial E$: The induced orientation from $E$ as a boundary and the induced orientation from $B$ and $\partial F$ when viewed as a fiber bundle. It can be seen in local trivializations that these two orientations coincide if and only if $q$ is even.
	
	We will be interested in the case where $F=D^p$ and $G=\mathrm{SO}(p)$ acts linearly. Then $\partial F= S^{p-1}$. These bundles are called \emph{linear disc bundles} and \emph{linear sphere bundles}, respectively. In the following, if $E\xrightarrow{\pi}B$ is a linear sphere bundle, we denote by $\overline{E}\xrightarrow{\pi}B$ its corresponding disc bundle. In particular, $\partial\overline{E}=E$.
	
	Let $E\xrightarrow{\pi}B^q$ be an oriented linear sphere bundle with fiber $S^{p-1}$ and connected base $B$. To describe the cohomology ring of $E$ we fix a commutative ring $R$ (which we assume to be unital). To simplify notation, when no coefficients for (co-)homology are indicated, we assume in this section that the coefficient ring is given by $R$. We denote by $\rho_R\colon H^*(-;\Z)\to H^*(-)$ the map induced by the ring homomorphism $\Z\to R$, $z\mapsto z\cdot 1_R$.
	
	Let $e(\pi)\in H^p(B;\Z)$ be the Euler class of the bundle $\pi$ (see e.g.\ \cite[Chapter 9]{MS74} for its definition) and set $e_R(\pi)=\rho_R(e(\pi))\in H^p(B)$. Then we have a long exact sequence, called the \emph{Gysin sequence},
	\begin{equation}
		\label{EQ:GYSIN}
		\dots\xrightarrow{\cdot\smile e_R(\pi)} H^i(B)\xrightarrow{\pi^*} H^i(E)\xrightarrow{\psi} H^{i-p+1}(B)\xrightarrow{\cdot\smile e_R(\pi)}H^{i+1}(B)\xrightarrow{\pi^*}\dots,
	\end{equation}
	see e.g.\ \cite[Corollary 12.2]{MS74} for $R=\Z$ and note that the proof carries over in the same way for arbitrary $R$ (see also \cite[Lemma B.3.1]{Re22}). The map $\psi$ is defined by
	\begin{equation}
		\label{EQ:PSI}
		\psi=\Phi^{-1}\circ \delta,
	\end{equation}
	 where $\Phi=\cdot\smile u_R(\pi)\colon H^{*-p}(B)\to H^*(\overline{E},E)$ is the Thom isomorphism and $u_R(\pi)\in H^p(\overline{E},E)$ is the Thom class, and $\delta\colon H^{*-1}(E)\to H^{*}(\overline{E},E)$ is the connecting homomorphism in the long exact sequence in cohomology for the pair $(\overline{E},E)$.
	
	Now assume that the Euler class $e_R(\pi)$ vanishes, so we obtain from the Gysin sequence short exact sequences
	\begin{equation}
		\label{EQ:GYSIN_SES}
		0\longrightarrow H^i(B)\xrightarrow{\pi^*} H^i(E)\xrightarrow{\psi} H^{i-p+1}(B)\longrightarrow0.
	\end{equation}
	As in \cite[Section 5.3]{Re23}, for any $a\in H^{p-1}(E)$ that maps to $1\in R\cong H^0(B)$ under $\psi$, we obtain by \cite[Lemma 1]{Ma58} a splitting $\theta_{a}\colon H^*(B)\to H^{*+p-1}(E)$ of \eqref{EQ:GYSIN_SES} defined by
	\[\theta_{a}(x)=(-1)^{ip}a\smile \pi^*(x) \]
	for $x\in H^i(B)$. Hence, we have
	\begin{equation}
		\label{EQ:BDL_COHOM_SPLIT}
		H^i(E)=\pi^*(H^i(B))\oplus \theta_{a}(H^{i-p+1}(B)).
	\end{equation}

	It follows from \eqref{EQ:PSI} that $u_R=\Phi(1)=\delta(a)$. Hence, since the Thom class satisfies
	\begin{equation}
		\label{EQ:u_FIBR_INCL}
		\overline{\iota}^* u_R=[D^p,S^{p-1}]^*
	\end{equation}
	for any fiber inclusion $\overline{\iota}\colon D^p\hookrightarrow\overline{E}$, it follows that
	\begin{equation}
		\label{EQ:a_FIBR_INCL}
		\iota^*a=[S^{p-1}]^*
	\end{equation}
	for any fiber inclusion $\iota\colon S^{p-1}\hookrightarrow E$.
	\begin{remark}
		The Thom class $u_R$ is uniquely determined by the property \eqref{EQ:u_FIBR_INCL}, see e.g.\ \cite[Theorem 10.4]{MS74}. For linear sphere bundles we see that, while $a$ satisfies the corresponding property \eqref{EQ:a_FIBR_INCL}, it is not uniquely determined by it as the choice of $a$ is not unique provided the map $\pi^*\colon H^{p-1}(B)\to H^{p-1}(E)$ is non-trivial.
	\end{remark}
	
	To determine the cohomology ring structure of $E$, if $p$ is odd, two characteristic classes are important: The Stiefel-Whitney class $w_{p-1}(\pi)\in H^{p-1}(B;\Z/2)$ and the Pontryagin class $p_i(\pi)\in H^{4i}(B;\Z)$ for $i=\frac{2p-2}{4}$. As a consequence of the Wu formula for the bundle $\pi$, see e.g.\ \cite[Theorem C]{Th60}, we have
	\begin{equation}
		\label{EQ:WU}
		\mathcal{P}_2(w_{p-1}(\pi))\equiv p_i(\pi)\mod 4,
	\end{equation}
	where $\mathcal{P}_2\colon H^{j}(B;\Z/2)\to H^{2j}(B;\Z/4)$ is the Pontryagin square operation.
	\begin{proposition}
		\label{P:SPHERE_BDL_COHOM}
		Let $E\xrightarrow{\pi}B$ be an oriented linear sphere bundle with fiber $S^{p-1}$, $p$ odd, whose Euler class $e_R(\pi)\in H^p(B;R)$ vanishes and assume that $B$ is connected. Let $W\in H^{p-1}(B;\Z)$ so that $\rho_{\Z/2}W=w_{p-1}(\pi)$ and suppose that $H^{2p-2}(B;\Z)$ has no element of order 2. Define (using \eqref{EQ:WU}) $P=\frac{1}{4}(p_i(\pi)-W^2)$ for $i=\frac{2p-2}{4}$. Then there is an element $a\in H^{p-1}(E;R)$, so that
		\[H^*(E;R)\cong  \vbigslant{H^*(B;R)\otimes_R R[a]}{\langle 1\otimes a^2-\rho_R W\otimes a-\rho_R P\otimes 1 \rangle}.  \]
		The isomorphism is given by $x\otimes a^i\mapsto \pi^*(x)\smile a^i$.
	\end{proposition}
	\begin{proof}
		By \eqref{EQ:BDL_COHOM_SPLIT} every $y\in H^i(E)$ can be written as $y=\pi^*(x_1)+\pi^*(x_2)\smile a$ for unique elements ${x_1\in H^i(B)}$, $x_2\in H^{i-p+1}(B)$. In particular, there are $\alpha\in H^{2p-2}(B)$ and $\beta\in H^{p-1}(B)$ so that
		\[a\smile a=\pi^*(\alpha)+\pi^*(\beta)\smile a. \]
		Thus, $H^*(E)$ is isomorphic to
		\[\bigslant{H^*(B)\otimes_R R[a]}{\langle 1\otimes a^2-\beta\otimes a-\alpha\otimes 1 \rangle}. \]
		For the values of $\alpha$ and $\beta$ first consider the case $R=\Z$. Then, by \cite[Theorem III]{Ma58}, we can choose $a$ so that $\beta =W$. Further, by \cite[Theorem IV]{Ma58}, we have
		\[p_i(\pi)=4\alpha+\beta^2. \]
		Since $H^{2p-2}(B)$ has no element of order $2$, this equation determines $\alpha$ uniquely, and we can write
		\[\alpha=\frac{1}{4}(p_i(\pi)-\beta^2)=\frac{1}{4}(p_i(\pi)-W^2)=P. \]
		
		For an arbitrary commutative ring $R$ first note that $\psi\circ\rho_R=\rho_R\circ\psi$. This follows by \eqref{EQ:PSI} and from the fact that the Thom classes $u_\Z$ and $u_R$ with coefficients in $\Z$ and $R$, respectively, satisfy $\rho_R u_\Z=u_R$, cf.\ \cite[Remark on p.\ 111]{MS74}.
		
		Thus, for any $a\in H^{p-1}(E)$ with $\psi(a)=1$, we have $\psi(\rho_R a)=\rho_R\psi(a)=1$ and we denote $\rho_R a\in H^{p-1}(E;R)$ again by $a$. Hence, $\alpha=\rho_R P$ and $\beta=\rho_R W$.
	\end{proof}

	Recall that a \emph{stable characteristic class} is a sequence of elements $c\in H^i(\mathrm{BO}(p);R)$ for all $p\in\N_0$ (and we denote all these elements by $c$) satisfying $(Bj_p)^*c=c$ for all $p$, where $Bj_p\colon \mathrm{BO}(p)\to\mathrm{BO}(p+1)$ denotes the map obtained from the inclusion $\mathrm{O}(p)\hookrightarrow\mathrm{O}(p+1)$. For a vector bundle $\xi\colon E\to B$ of rank $p$ with classifying map $f\colon B\to\mathrm{BO}(p)$ the class $c(\xi)\in H^i(B;R)$ is then defined by $c(\xi)=f^*c$. The property $(Bj_p)^*c=c$ implies $c(\xi\oplus\underline{\R}_B)=c(\xi)$ for any vector bundle $\xi\colon E\to B$, where $\underline{\R}_B$ denotes the trivial line bundle over $B$. If one considers vector bundles with an additional structure, such as an orientation or a spin structure, the orthogonal group $\mathrm{O}(p)$ is replaced by the corresponding structure group, such as $\mathrm{SO}(p)$ or $\mathrm{Spin}(p)$. Examples of stable characteristic classes are the Pontryagin classes $p_i$ (with $R=\Z$) and the Stiefel-Whitney classes $w_i$ (with $R=\Z/2$).
	
	The following is well-known. We include the proof for convenience.
	\begin{proposition}
		\label{P:SPHERE_BDL_CHAR_CLASSES}
		Let $E\xrightarrow{\pi}B$ be a linear sphere bundle and let $c\in H^i(\mathrm{BO}(p);R)$ be a stable characteristic class. Denote the vector bundle corresponding to $\pi$ by $\xi$. Then
		\[c(TE)=\pi^*c(TB\oplus\xi).\]
	\end{proposition}
	\begin{proof}
		Let $\iota\colon E\hookrightarrow\overline{E}$ be the inclusion. Then $\iota^*T\overline{E}\cong TE\oplus \underline{\R}_E$, the trivial factor corresponds to the normal bundle of $E=\partial\overline{E}$. Further, $T\overline{E}\cong\pi^*TB\oplus\pi^*\xi$, which can be verified by turning the bundle into a Riemannian submersion, since then the horizontal distribution is isomorphic to $\pi^*TB$ and the vertical distribution is isomorphic to $\pi^*\xi$.
		
		It follows that
		\[c(E)=c(TE\oplus\underline{\R}_E)=c(\iota^*T\overline{E})=\iota^*c(\pi^*(TB\oplus\xi))=\iota^*\pi^*c(TB\oplus \xi)=\pi^*c(TB\oplus\xi). \]
	\end{proof}
	For $q=4k, p=2k+1$, Propositions \ref{P:SPHERE_BDL_COHOM} and \ref{P:SPHERE_BDL_CHAR_CLASSES} yield the following corollary.
	\begin{corollary}
		\label{C:S2k_BDLS_S4k}
		Let $E\xrightarrow{\pi}B^{4k}$ be an oriented linear $S^{2k}$-bundle with $e(\pi)=0$, where $B$ is closed, connected and oriented. Then for any $W\in H^{2k}(B;\Z)$ with $\rho_{\Z/2}W=w_{2k}(\pi)$ there exists $a\in H^{2k}(E;\Z)$ so that
		\[H^{2k}(E;\Z)=\pi^*(H^{2k}(B;\Z))\oplus \Z a \]
		and for $x_1,x_2,x_3\in H^{2k}(B;\Z)$ we have
		\begin{alignat*}{2}
			\mu_E(\pi^*x_1,\pi^*x_2,\pi^*x_3) &= 0,\\
			\mu_E(\pi^*x_1,\pi^*x_2, a) &= \langle x_1\smile x_2,[B]\rangle,\\
			\mu_E(\pi^*x_1, a, a) &= \langle x_1\smile W,[B]\rangle,\\
			\mu_E(a, a, a) &= \frac{1}{4}\langle 3W^2+p_k(\pi),[B]\rangle.
		\end{alignat*}
		Further, for $x\in H^{2k}(B;\Z)$ we have
		\begin{alignat*}{2}
			p_k(E) (\pi^*x)&=0,\\
			p_k(E) (a)&=\langle p_k(TB\oplus \xi),[B]\rangle.
		\end{alignat*}
	\end{corollary}

	For the proof of Theorem \ref{T:NEW_EX_B2_LARGE} we will need the following lemma to show that certain spaces cannot be the total space of a linear sphere bundle.
	
	\begin{lemma}
		\label{L:SPHERE_BDL_OBSTR}
		Let $E\xrightarrow{\pi}B^q$ be a linear $S^{p-1}$-bundle with $B$ closed, oriented and connected.
		\begin{enumerate}
			\item For the Euler characteristic we have $\chi(E)=\chi(B)\chi(S^{p-1})$. In particular, $\chi(E)$ vanishes if $p$ is even and $\chi(E)$ is even if $p$ is odd.
			\item If all cohomology groups of $E$ in odd degrees vanish, then $b_{j}(E)$ is even for $j=\frac{p+q-1}{2}$.
			\item If $p+q-1=6k$ and $p_k(E)\neq0$ and
			\begin{itemize}
				\item $2k+1\leq p\leq 6k$, or
				\item $1<p\leq 2k$ and $E$ is $(2k-1)$-connected,
			\end{itemize}
			then $\mu_E$ is trivial on $\ker(p_k(E))\times \ker(p_k(E))\times \ker(p_k(E))$.
		\end{enumerate}
	\end{lemma}
	\begin{proof}
		The first claim is well-known and holds more generally for arbitrary fiber bundles, see \cite{Se53}. For the second claim, first note that, since $E$ only has cohomology in even degrees, its Euler characteristic is positive. Hence, by item (1), we have that $p$ is odd, so $\chi(E)$ is even. Then it follows by Poincaré duality, that
		\[\chi(E)=2\sum_{i=0}^{j-1}b_i(E)+b_{j}(E), \]
		and hence $b_{j}(E)$ is even.
		
		Now suppose that $p+q-1=6k$ and $p_k(E)\neq0$. We will consider coefficients in $\Z$. If $p\geq 2k+1$, i.e.\ $q\leq 4k$, it follows from Proposition \ref{P:SPHERE_BDL_CHAR_CLASSES} that $p_k(E)\in H^{4k}(E)$ can only be non-trivial if $p=2k+1$, since otherwise $H^{4k}(B)$ is trivial.
		Then, by Corollary \ref{C:S2k_BDLS_S4k}, the trilinear form is trivial on $\ker(p_k(E))$ provided $e(\pi)$ vanishes. If $e(\pi)$ is non-trivial, the map $H^{2k}(B)\xrightarrow{\pi^*}H^{2k}(E)$ is an isomorphism by the Gysin sequence, so $\mu_E$ vanishes on all of $H^{2k}(E)$.
		
		Finally, we assume that $4k<q<6k$ and that $E$ is $(2k-1)$-connected. Then for $0<i+p<2k$, by the Gysin sequence, we have an isomorphism
		\[H^i(B)\xrightarrow{\cdot\smile e(\pi)}H^{i+p}(B). \]
		Hence, we have for $0<i<2k$
		\[H^i(B)\cong\begin{cases}
			\Z,\quad &p\mid i,\\
			0,\quad &\text{else.}
		\end{cases} \] 
		Again from the Gysin sequence we obtain the following exact sequence.
		\[H^{2k}(B)\xrightarrow{\pi^*}H^{2k}(E)\xrightarrow{\psi}H^{2k-p+1}(B). \]
		We have $H^{2k-p+1}(B)\cong\Z$ or $0$. Let $y\in H^{4k}(B)$ so that $\pi^*y=p_k(E)$ (which exists by Proposition \ref{P:SPHERE_BDL_CHAR_CLASSES}). Then for every $x\in H^{2k}(B)$ we have
		\[\pi^*(x)\smile p_k(E)=\pi^*(x\smile y)=0. \]
		Hence, if $p_k(E)\neq0$, we have $H^{2k-p+1}(B)\cong\Z$ and $\psi$ is non-trivial, otherwise $p_k(E)$ would be trivial on $H^{2k}(E)$, which would imply $p_k(E)=0$ by Poincaré duality. Hence, if $a\in H^{2k}(E)$ is a preimage of a generator of $\im\psi$, every $y\in H^{2k}(E)$ can be written as $y=\pi^*x+\lambda a$, where $x\in H^{2k}(B)$. Since $p_k(E)$ is trivial on $\pi^*H^{2k}(B)$, it follows that $\ker(p_k(E))=\pi^*H^{2k}(B)$, on which $\mu_E$ is trivial.
	\end{proof}
	
	We now construct linear $S^{2k}$-bundles over a closed, oriented $4k$-dimensional base $B^{4k}$. For that, recall that a \emph{collapse map} $B\to S^{4k}\cong D^{4k}/\partial D^{4k}$ for an embedding $D^{4k}\hookrightarrow B$ is defined as the identity on ${D^{4k}}^\circ$ and maps all other points to $\partial D^{4k}$. Since any two orientation-preserving embeddings $D^{4k}\hookrightarrow B$ are isotopic \cite[Theorem 5.5]{Pa59}, any two orientation-preserving collapse maps are homotopic.

	Further, we set
	\[\nu_k=\begin{cases}
		2,\quad & k=1,4,\\
		8,\quad & k=2,\\
		1,\quad & k=3\text{ or }k>4.
	\end{cases}
	\]
	\begin{definition}\label{D:pi_alpha,B}
		Let $B^{4k}$ be a closed, oriented $4k$-dimensional manifold and let $\alpha\in\Z$. We define the linear $S^{2k}$-bundle $\pi_{\alpha,B}\colon E\to B$ as the bundle classified by the composition of an orientation-preserving collapse map $B\to S^{4k}$ and a preimage of $\nu_k\alpha\in \Z\cong \pi_{4k}(\mathrm{BSO})$ in $\pi_{4k}(\mathrm{BSO}(2k+1))$.
	\end{definition}
	The existence of such a preimage is guaranteed by \cite[Theorem 1.1 and Proposition 2.1]{DM89}. For $k=1,2$ this preimage is unique since $\pi_4(\mathrm{BSO}(3))\cong\pi_8(\mathrm{BSO}(5))\cong\Z$, and for $k>2$ it is non-unique in general. For example, we have $\pi_{12}(\mathrm{BSO}(7))\cong \Z\oplus\Z/2$ and $\pi_{16}(\mathrm{BSO}(9))\cong \Z\oplus(\Z/2)^3$ (see e.g.\ \cite[App.\ A, Table 6.VII]{En87}), showing that for $k=3,4$ there exist $2$, resp.\ $8$, preimages. We note that the choice of preimage is not relevant for the results of this article.
	
	\begin{lemma}
		\label{L:S2k_BUNDLES_S4k}
		Set $\lambda_k=\nu_k\frac{3-(-1)^k}{2}(2k-1)!$. Then the bundle $\pi_{\alpha,B}$ satisfies the following:
		\begin{align*}
			w_{i}(\pi_{\alpha,B})&=\begin{cases}
				1,\quad &i=0,\\0,\quad &\text{else},
			\end{cases}\\
			p_i(\pi_{\alpha,B})&=\begin{cases}
				1,\quad &i=0,\\\alpha \lambda_k[B]^*,\quad &i=k,\\0,\quad &\text{else,}
			\end{cases}\\
			e(\pi_{\alpha,B})&=0.
		\end{align*}
	\end{lemma}
	\begin{proof}
		We first consider the bundle $\pi_{\alpha,S^{4k}}$ over $S^{4k}$. By \cite[Theorem 26.5]{BH59}, the $k$-th Pontryagin class $p_k$ of a stable vector bundle corresponding to a generator of the group $\pi_{4k}(\mathrm{BSO})\cong \Z$ is given by $\frac{3-(-1)^k}{2}(2k-1)![S^{4k}]^*$. Further, since $S^{4k}$ has non-vanishing cohomology groups only in degrees $0$ and $4k$ and since $\pi_{\alpha,S^{4k}}$ has rank $(2k+1)<4k$, all of the characteristic classes $w_i(\pi_{\alpha,S^{4k}})$ ($i\neq 0$), $p_i(\pi_{\alpha,S^{4k}})$ ($i\neq 0,k$) and $e(\pi_{\alpha,S^{4k}})$ vanish.
		
		For arbitrary base $B^{4k}$ the bundle $\pi_{\alpha,B}$ is the pull-back of $\pi_{\alpha,S^{4k}}$ along an orientation-preserving collapse map $h\colon B\to S^{4k}$, since an orientation-preserving collapse map of $S^{4k}$ is homotopic to the identity. The claim now follows from the naturality properties of $w_i$, $p_i$ and $e$ and the fact that $h$ has degree $1$, so $h^*[S^{4k}]^*=[B]^*$.
	\end{proof}
	
	\begin{remark}
		\label{R:LAMBDA_K}		
		Note that for $k\geq3$ we have
		\[{2k+1\choose k}=\frac{(2k+1)!}{k!(k+1)!}=(2k-1)!\frac{2k(2k+1)}{k!(k+1)!}\]
		and 
		\[\frac{2k(2k+1)}{k!(k+1)!}\leq\frac{2k(2k+2)}{k!(k+1)!}=\frac{4}{(k-1)!k!}< 1, \]
		showing that ${2k+1\choose k}< (2k-1)!\leq\lambda_k$. Since $\lambda_1=4$ and $\lambda_2=48$, the inequality ${2k+1\choose k}< \lambda_k$ holds for all $k\in\N$. This fact will be useful in the proof of Theorem \ref{T:NEW_EX_B2_LARGE}.
	\end{remark}
	\begin{remark}\label{R:S2_B}
		When $k=1$ and $B=S^4$, the linear $S^2$-bundles over $B$ constructed in Lemma \ref{L:S2k_BUNDLES_S4k} cover in fact all possible linear linear $S^2$-bundles over $S^4$, see e.g.\ \cite[Corollary 5.2]{Re23}. If $B=\pm\C P^2$, then, by \cite[Corollary 5.2]{Re23}, isomorphism classes of linear $S^2$-bundles over $B$ are in bijection with $\Z\times \{0,1\}$, where the bijection assigns to $(\alpha,\beta)$ the linear $S^2$-bundle $\pi$ with $p_1(\pi)=(4\alpha\pm\beta)[B]^*$ and $w_2(\pi)=\beta b$, where $b\in H^2(\C P^2;\Z/2)$ denotes a generator. Hence, the bundles constructed in Lemma \ref{L:S2k_BUNDLES_S4k} are precisely those corresponding to $(\alpha,0)\in \Z\times\{0,1\}$.
	\end{remark}
	
	We will now restrict to the following base manifolds:
	\begin{definition}
		Let $k,m\in \N$ and for $\overline{\gamma}=(\gamma_1,\dots,\gamma_m)\in\{\pm1\}^m$ we define the manifold $B_{\overline{\gamma}^{4k}}$ by
		\[B^{4k}_{\overline{\gamma}}=\gamma_1\C P^{2k}\#\dots\# \gamma_m\C P^{2k}.\]
	\end{definition}
	
	Let $B=B_{\overline{\gamma}}^{4k}$ and let $b_i\in H^2(\gamma_i\C P^{2k})$ be a generator of $H^*(\gamma_i\C P^{2k})$ so that $\gamma_i b_i^k=[\gamma_i\C P^{2k}]^*$. Then we have a ring isomorphism
	\[H^*(B)\cong\vbigslant{\Z[b_1,\dots,b_m]}{\langle b_i b_j,\gamma_ib_i^{2k}-\gamma_j b_j^{2k},b_i^{2k+1}\mid i\neq j\rangle }  \]
	given by mapping each $b_i\in H^2(\gamma_i \C P^{2k})$ (considered as an element of $H^2(B)$ via the collapse $B\to \gamma_i \C P^{2k}$) to the generator $b_i$ on the right-hand side. This also determines $H^*(B;R)$ via $H^*(B;R)\cong H^*(B)\otimes R$ (as $\C P^{2k}$, and thus $B$, has a cell decomposition with only cells in even dimensions).
	
	If $E\xrightarrow{\pi}B$ is a linear sphere bundle we will denote the images $\pi^* b_i$ in $H^*(E)$ again by $b_i$.
	
	Further,
	\[w_j(B)=\begin{cases}{2k+1\choose j/2}\sum_i \rho_{\Z/2}b_i^{j/2},\quad &j\text{ even,}\\0,\quad&\text{else,}\end{cases} \]
	and
	\[p_j(B)={2k+1\choose j}\sum_i b_i^{2j},\]
	see \cite[Corollary 11.15 and Example 15.6]{MS74}.
		
	The following corollary is now an immediate consequence of Propositions \ref{P:SPHERE_BDL_COHOM} and \ref{P:SPHERE_BDL_CHAR_CLASSES} and Lemma~\ref{L:S2k_BUNDLES_S4k}. To simplify notation, we will again write $b_i$ for $\rho_R b_i$.
	\begin{corollary}
		\label{C:2k_SPHERE_BDL_COHOM}
		Let $E\xrightarrow{\pi}B_{\overline{\gamma}}^{4k}$ be an oriented linear $S^{2k}$-bundle corresponding to $\alpha\in\Z$ in Lemma~\ref{L:S2k_BUNDLES_S4k}. Then
		\[H^*(E;R)\cong \vbigslant{R[a,b_1,\dots,b_m]}{\langle b_i b_j,\gamma_i b_i^{2k}-\gamma_j b_j^{2k},b_i^{2k+1},a^2-\frac{1}{4}\alpha\lambda_k \gamma_1b_1^{2k}\mid i\neq j \rangle }, \]
		where $a$ has degree $2k$ and $b_1,\dots,b_m$ have degree $2$. Further, 
		\begin{align*}
			w_j(E)&=\begin{cases}{2k+1\choose j/2}\sum_i b_i^{j/2},\quad &j\text{ even,}\\0,\quad&\text{else,}\end{cases}\\
			p_j(E)&=\begin{cases}
				\left({2k+1\choose k}\sum_i\gamma_i +\alpha \lambda_k\right)\gamma_1 b_1^{2k},\quad &j=k,\\
				{2k+1\choose j}\sum_{i}b_i^{2j},\quad &\text{else.}
			\end{cases}
		\end{align*}
		In particular, we have
		\[H^{2k}(E;\Z)=\bigoplus_i\Z b_i^k\oplus\Z a \]
		and
		\begin{alignat*}{2}
			\mu_E (b_i^k\smile b_j^k\smile b_m^k) &= 0,\\
			\mu_E (b_i^k\smile b_j^k\smile a) &= \delta_{ij}\gamma_i,\\
			\mu_E (b_i^k\smile a\smile a) &= 0,\\
			\mu_E (a\smile a\smile a) &=\frac{\lambda_k}{4}\alpha,\\
			w_2(E)^k&=\sum_i b_i^k, \\
			p_k(E) (b_i)&=0,\\
			p_k(E) (a)&={2k+1\choose k}\sum_i\gamma_i+\lambda_k\alpha.
		\end{alignat*}
	\end{corollary}
	
	\section{Plumbings According to a Bipartite Graph}
	
	\label{S:PLUMBINGS}
	
	Let $G$ be a labeled bipartite graph, i.e.\ $G=(U,V,E,\pi,\delta)$, where $U$ and $V$ are the sets of vertices, which are assumed to be finite, $E\subseteq U\times V$ is the set of edges, $\delta\colon E\to\{\pm 1\}$ is a function, and $\pi$ assigns to each vertex an oriented linear disc bundle in the following way: For fixed $p,q\in\N$ we assign to each $u\in U$ an oriented linear disc bundle $D^p\hookrightarrow \overline{E}_u\xrightarrow{\pi_u}B_u^q$ and to each $v\in V$ an oriented linear disc bundle $D^q\hookrightarrow \overline{E}_v\xrightarrow{\pi_v}B_v^p$. We assume that all base spaces are connected and oriented. We call such a graph $G$ a \emph{geometric plumbing graph}. We refer to \cite[Section 2.1]{Re23} and the references therein for an introduction to plumbing and we use the notation established there.
	
	\begin{definition}\label{D:M_G}
		For a geometric plumbing graph $G$ we define $\overline{M}_G$ as the boundary connected sum of the manifolds obtained by plumbing the bundles $\pi_u$ and $\pi_v$ according to each connected component of the graph $G$, where each edge $e$ corresponds to plumbing with sign $\delta(e)$. We set $M_G=\partial\overline{M}_G$.
	\end{definition}
	
	We equip each total space $\overline{E}_u$ and $\overline{E}_v$ with the orientation induced from the base manifolds and the fiber orientations (cf.\ Subsection \ref{SS:LSB}). Then recall from \cite[Section 2.1]{Re23} that $\overline{M}_G$ can be oriented so that the orientation agrees with that of $\overline{E}_u$ and $\overline{E}_v$ if $p$ or $q$ is even and with that of $-\overline{E}_u$ and $\overline{E}_v$ if $p$ and $q$ are odd. Further, recall from Subsection \ref{SS:LSB} that the induced orientation of $E_u$ ($E_v$) as a boundary coincides with the induced orientation from the bundle structure if and only if $q$ is even ($p$ is even). Thus, the induced orientation on $M_G$ as a boundary coincides with the orientations of $E_u$ and $E_v$ induced from the bundle structures if and only if $p$ and $q$ are even, and if one of $p$ and $q$ is odd, then $M_G$ can be oriented so that the orientation agrees with the orientations of $-E_u$ and $E_v$. We will always choose these orientations in the following.

	\subsection{Modifications of Graphs}
	\label{SS:GRAPH_MOD}
	
	Let $G$ be a geometric plumbing graph. Our goal in this section is to simplify the graph data as much as possibly without changing the diffeomorphism type of $M_G$.
	
	\begin{proposition}[Sign of a Separating Edge]
		\label{P:GRAPH_RED_EDGE}
		Suppose $G$ is of the form
		\[
		\begin{tikzcd}[cells={nodes={circle, draw=black, anchor=center,minimum height=2em}}, column sep=small]
			|[draw=none]|G_1\arrow[dash]{rr}{e}&&|[draw=none]|G_2
		\end{tikzcd}
		\]
		with subgraphs $G_1$ and $G_2$ of $G$. Let $G^\prime$ be the graph obtained from $G$ by reversing the sign of $e$ and by reversing all base and fiber orientations of the bundles associated to vertices in $G_2$. Then $M_G$ is diffeomorphic to $M_{G^\prime}$.
	\end{proposition}
	
	\begin{proposition}[Vertex of Degree 1]
		\label{P:GRAPH_RED1}
		Suppose $G$ is of the form
	
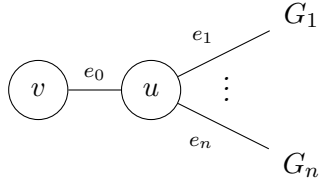
\begin{center}
		\begin{tikzpicture}
			\begin{scope}[every node/.style={circle,draw,minimum height=2em}]
				\node (V) at (-0.5,0) {$v$};
				\node (U) at (1,0) {$u$};
				\node[draw=none] (G1) at (3,1) {$G_1$};
				\node[draw=none] (Gn) at (3,-1) {$G_n$};
				\node[draw=none] (dots) at (2,0) {$\rvdots$};
			\end{scope}
			\path[-](V) edge["{$\scriptstyle e_0$}"] (U);
			\path[-](U) edge["{$\scriptstyle e_1$}"] (G1);
			\path[-](Gn) edge["{$\scriptstyle e_n$}"] (U);
		\end{tikzpicture}
	\end{center}
		with subgraphs $G_1,\dots,G_n$ of $G$. Let $G_0$ be the graph
		
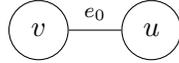
\begin{center}
			\begin{tikzpicture}
				\begin{scope}[every node/.style={circle,draw,minimum height=2em}]
					\node (V) at (-0.5,0) {$v$};
					\node (U) at (1,0) {$u$};
				\end{scope}
				\path[-](V) edge["{$\scriptstyle e_0$}"] (U);
			\end{tikzpicture}
		\end{center}
		with corresponding restrictions of $\delta$ and $\pi$. If $\pi_v$ is trivial with $B_v=S^p$, then $M_G$ is diffeomorphic to $M_{G_0}\# M_{G_1}\#\dots\# M_{G_n}$, i.e.\ we can replace $G$ by the disjoint union of the subgraphs $G_0,\dots,G_n$.
	\end{proposition}

	\begin{proposition}
		\label{P:GRAPH_REDG0}
		Let $G$ be the graph
		
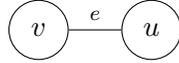
\begin{center}
			\begin{tikzpicture}
				\begin{scope}[every node/.style={circle,draw,minimum height=2em}]
					\node (V) at (-0.5,0) {$v$};
					\node (U) at (1,0) {$u$};
				\end{scope}
				\path[-](V) edge["{$\scriptstyle e$}"] (U);
			\end{tikzpicture}
		\end{center}
		where $\pi_v$ is trivial with $B_v=S^p$ and $B_u=S^q$. Then $M_G\cong S^{p+q-1}$.
	\end{proposition}

	\begin{proposition}[Vertex of Degree 2]
		\label{P:GRAPH_RED2}
		Suppose $G$ is of the following form:
		
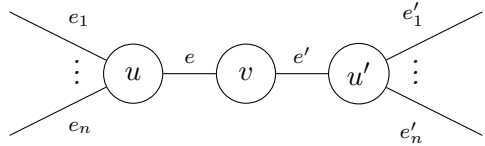
\begin{center}
		\begin{tikzpicture}
			\begin{scope}[every node/.style={circle,draw,minimum height=2em}]
				\node (V) at (0,0) {$v$};
				\node (U1) at (-1.5,0) {$u$};
				\node (U2) at (1.5,0) {$u^\prime$};
				\node[draw=none] (dots1) at (-2.25,0) {$\rvdots$};
				\node[draw=none] (dots2) at (2.25,0) {$\rvdots$};
				\node[draw=none] (G1) at (-3.5,1) {};
				\node[draw=none] (Gn) at (-3.5,-1) {};
				\node[draw=none] (G11) at (3.5,1) {};
				\node[draw=none] (G1n) at (3.5,-1) {};
			\end{scope}
			\path[-](U1) edge["{$\scriptstyle e$}"] (V);
			\path[-](V) edge["{$\scriptstyle e^\prime$}"] (U2);
			\path[-](G1) edge["{$\scriptstyle e_1$}"] (U1);
			\path[-](U1) edge["{$\scriptstyle e_n$}"] (Gn);
			\path[-](U2) edge["{$\scriptstyle e_1^\prime$}"] (G11);
			\path[-](G1n) edge["{$\scriptstyle e_n^\prime$}"] (U2);
		\end{tikzpicture}
	\end{center}
		Suppose $\pi_v$ is trivial with $B_v=S^p$ and $\delta(e)=1$, $\delta(e^\prime)=-1$. Define $\pi_{\hat{u}}$ as the fiber connected sum of $\pi_{u}$ and $\pi_{u'}$. Then we can replace $G$ by the following graph without changing the diffeomorphism type of $M_G$.
		\begin{center}
			\begin{tikzpicture}
				\begin{scope}[every node/.style={circle,draw,minimum height=2em}]
					\node (U) at (0,0) {$\hat{u}$};
					\node[draw=none] (dots1) at (-0.75,0) {$\rvdots$};
					\node[draw=none] (dots2) at (0.75,0) {$\rvdots$};
					\node[draw=none] (G1) at (-2,1) {};
					\node[draw=none] (Gn) at (-2,-1) {};
					\node[draw=none] (G11) at (2,1) {};
					\node[draw=none] (G1n) at (2,-1) {};
				\end{scope}
				\path[-](G1) edge["{$\scriptstyle e_1$}"] (U);
				\path[-](U) edge["{$\scriptstyle e_n$}"] (Gn);
				\path[-](U) edge["{$\scriptstyle e_1^\prime$}"] (G11);
				\path[-](G1n) edge["{$\scriptstyle e_n^\prime$}"] (U);
			\end{tikzpicture}
		\end{center}
	\end{proposition}
	
	\begin{remark}
		\begin{enumerate}
			\item Proposition \ref{P:GRAPH_RED1} generalizes \cite[Proposition 2.6]{CW17} and \cite[Satz 5.9 (A)]{Sc75}.
			\item By Proposition \ref{P:GRAPH_RED_EDGE}, we can, after possibly reversing base and fiber orientations in the corresponding subgraphs, apply Proposition \ref{P:GRAPH_RED2} without any condition on the signs of $e$ and $e'$.
		\end{enumerate}
	\end{remark}
	
	Proposition \ref{P:GRAPH_RED_EDGE} directly follows from the fact that, by reversing the orientations as assumed, the gluing maps for the plumbings do not change. Proposition \ref{P:GRAPH_RED2} follows from \cite[Lemma 2.10]{CW17}. For convenience we give the proof below.
	
	\begin{proof}[{{Proof of Proposition \ref{P:GRAPH_RED2}, cf.\ \cite[Lemma 2.10]{CW17}}}]
		Denote the graph
		\begin{center}
			\begin{tikzpicture}
				\begin{scope}[every node/.style={circle,draw,minimum height=2em}]
					\node (V) at (0,0) {$v$};
					\node (U1) at (-1.5,0) {$u$};
					\node (U2) at (1.5,0) {$u^\prime$};
				\end{scope}
				\path[-](U1) edge["{$\scriptstyle e$}"] (V);
				\path[-](V) edge["{$\scriptstyle e^\prime$}"] (U2);
			\end{tikzpicture}
		\end{center}
		by $G^\prime$. Then
		\begin{align*}
			M_{G^\prime} &\cong \pi_{u}^{-1}(B_{u}\setminus {D^q}^\circ)\cup_{S^{q-1}\times S^{p-1}} (S^p\setminus(D^p\sqcup D^p)^\circ)\times S^{q-1}\cup_{S^{q-1}\times S^{p-1}} \pi_{u^\prime}^{-1}(B_{u^\prime}\setminus {D^q}^\circ)\\
			&\cong \pi_{u}^{-1}(B_{u}\setminus {D^q}^\circ)\cup_{S^{q-1}\times S^{p-1}} (S^{p-1}\times I)\times S^{q-1}\cup_{S^{q-1}\times S^{p-1}} \pi_{u^\prime}^{-1}(B_{u^\prime}\setminus {D^q}^\circ)\\
			&\cong \pi_{u}^{-1}(B_{u}\setminus {D^q}^\circ)\cup_{S^{q-1}\times S^{p-1}} \pi_{u^\prime}^{-1}(B_{u^\prime}\setminus {D^q}^\circ).
		\end{align*}
		Here we identify the boundaries of $\pi_u^{-1}(B_u\setminus {D^q}^\circ ) $ and $\pi_{u'}^{-1}(B_{u'}\setminus {D^q}^\circ)$ with $S^{q-1}\times S^{p-1}$ via local trivializations $D^q\times S^{p-1}\hookrightarrow E_u,E_{u'}$ and consider gluings to different connected components of the boundaries of $(S^p\setminus (D^p\sqcup D^p)^\circ )\times S^{q-1}$ and $(S^{p-1}\times I)\times S^{q-1}$ which can both be identified canonically with the disjoint union of two copies of $S^{p-1}\times S^{q-1}\cong S^{q-1}\times S^{p-1}$.		
		
		Since the induced orientations on the boundary components of $S^{p-1}\times I$ are opposite to each other, and by the assumption on the signs, the gluing map reverses the orientation on the $S^{q-1}$-factor and preserves the orientation on the $S^{p-1}$-factor. Hence, the claim follows.
	\end{proof}
	\begin{proof}[Proof of Proposition \ref{P:GRAPH_REDG0}]
		By possibly reversing the orientations on the base and fibers of $\pi_v$, we can assume that $\delta(e)=1$. Let $T\colon S^{q-1}\to\SO(p)$ be the clutching function for $\pi_u$. Then we have
		\[M_G\cong (S^{q-1}\times D^p)\cup_{\tilde{T}}(D^q\times S^{p-1}),  \]
		where $\tilde{T}\colon S^{q-1}\times S^{p-1}\to S^{q-1}\times S^{p-1}$ is defined by
		\[\tilde{T}(x,y)=(x,T_x(y)). \]
		We extend this map to $S^{q-1}\times D^p$, i.e.\ we define the diffeomorphism $\bar{T}\colon S^{q-1}\times D^p\to S^{q-1}\times D^p$ by $\bar{T}(x,y)=(x,T_x(y))$, so by definition $\bar{T}|_{S^{q-1}\times S^{p-1}}=\tilde{T}$. Hence, 
		\[M_G\cong \bar{T}(S^{q-1}\times D^p)\cup_{\tilde{T}}(D^q\times S^{p-1})=(S^{q-1}\times D^p)\cup_{\textup{id}}(D^q\times S^{p-1})\cong S^{p+q-1}, \]
		where the last isomorphism follows by identifying $D^q\times D^p$ with $D^{p+q}$ by smoothing corners and considering the boundaries of each space.
	\end{proof}
	
	For the proof of Proposition \ref{P:GRAPH_RED1} we need the following notion.
	\begin{definition}
		An embedding $\varphi\colon D^q	\times S^{p-1}\hookrightarrow M^{p+q-1}$ into a manifold $M$ is called \emph{trivial}, if there is an embedding $\bar{\varphi}\colon  D^{q-1}\times D^p\hookrightarrow M$ so that $\varphi|_{D^{q-1}\times S^{p-1}}=\bar{\varphi}|_{D^{q-1}\times S^{p-1}}$.
	\end{definition}
	Here we consider $D^{q-1}\subseteq D^q$ according to the embedding $\R^{q-1}\cong \R^{q-1}\times \{0\}\subseteq \R^q$.
	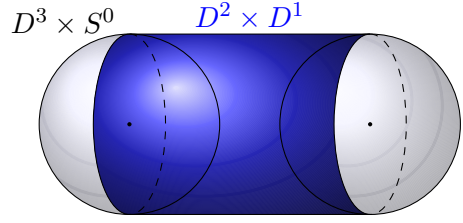
\begin{figure}
		\label{F:TRIV_EMB}
		\caption{A trivial embedding of $D^3\times S^0$}
		\centering
		\begin{tikzpicture}[scale=0.8]
			\shade[ball color=blue!10!white,opacity=0.2] (-2,0) circle (1.5cm);
			\shade[ball color=blue!10!white,opacity=0.2] (2,0) circle (1.5cm);
			\shade[ball color=blue!80,opacity=0.7] (-2,1.5) arc (90:270:0.6 and 1.5) arc (180:360:2 and 0) arc (270:90:0.6 and 1.5) arc (0:180:2 and 0);
			\fill[fill=black] (-2,0) circle (1pt);
			\fill[fill=black] (2,0) circle (1pt);
			\draw (-2,0) circle (1.5cm);
			\draw (2,0) circle (1.5cm);
			\draw (-2,1.5) arc (90:270:0.6 and 1.5);
			\draw[dashed] (-2,1.5) arc (90:-90:0.6 and 1.5);
			\draw (-2,-1.5) arc (180:360:2 and 0);
			\draw (2,-1.5) arc (270:90:0.6 and 1.5);
			\draw[dashed] (2,1.5) arc (90:-90:0.6 and 1.5);
			\draw (2,1.5) arc (0:180:2 and 0);
			\node[color=blue] at (0,1.8) {$D^2\times D^1$};
			\node at (-3.1,1.7) {$D^3\times S^0$};
		\end{tikzpicture}
	\end{figure}
	\begin{lemma}
		\label{L:TRIV_EMB_ISO}
		Let $M$ be connected. Then any two trivial embeddings $\varphi_1,\varphi_2\colon D^q\times S^{p-1}\hookrightarrow M$, that are orientation-preserving with respect to the same orientation of $M$ if $M$ is orientable, are isotopic.
	\end{lemma}
	\begin{proof}
		First note that we can assume that both $\bar{\varphi}_1$ and $\bar{\varphi}_2$ are orientation preserving if $M$ is orientable, since otherwise we can replace $\bar{\varphi}_i$ by the map
		\[(x_1,\dots,x_{q-1},y)\mapsto \bar{\varphi}_i(x_1,\dots,x_{q-2},-x_{q-1},y), \]
		where $(x_1,\dots,x_{q-1})\in D^{q-1}$ and $y\in D^p$, and $\varphi_i$ by the map
		\[(x_1,\dots,x_{q},y)\mapsto \varphi_i(x_1,\dots,x_{q-2},-x_{q-1},-x_q,y), \]
		where $(x_1,\dots,x_q)\in D^q$ and $y\in S^{p-1}$, which is clearly isotopic to $\varphi_i$. 
		
		By the disc theorem of Palais \cite[Theorem 5.5]{Pa59}, any two embeddings of $D^p$ into $M$ are ambient isotopic, where we require them to be orientation-preserving if $M$ is orientable and $p=\dim(M)$. It follows that, after applying an ambient isotopy, we can assume that $\bar{\varphi}_1|_{\{0\}\times D^p}=\bar{\varphi}_2|_{\{0\}\times D^p}$.
		
		If $M$ is non-orientable, we can introduce a local orientation by enlarging the image of one of $\bar{\varphi}_i$ to a ball $D^{p+q-1}$, so that $\varphi_i$, $\bar{\varphi}_i$ are all contained in this ball. If one of $\varphi_i$ is orientation-reversing with respect to the orientation of $D^{p+q-1}$, we can apply an isotopy of this ball that reverses the orientation, which exists by the disc theorem of Palais. If one of $\bar{\varphi}_i$ is orientation-reversing, we modify it as in the beginning of the proof. Hence, we can assume that $\varphi_i$, $\bar{\varphi}_i$ are all orientation-preserving.
		
		By the uniqueness of tubular neighborhoods, see e.g.\ \cite[Theorem 4.5.3]{Hi76}, after applying an isotopy to one of $\bar{\varphi}_i$, there is a smooth map $\phi\colon D^p\to\mathrm{GL}_+(q)$, so that $\bar{\varphi}_1(x,y)=\bar{\varphi}_2(\phi_y(x),y)$. Since $D^p$ is contractible, there exists a smooth homotopy of $\phi$ to the constant map $\phi\equiv\textrm{id}_{\R^q}$. This yields an isotopy of $\bar{\varphi}_2$ so that we can assume $\bar{\varphi}_1=\bar{\varphi}_2$. By the isotopy extension theorem, see e.g.\ \cite[Theorem 8.1.3]{Hi76}, this isotopy extends to a diffeotopy of $M$, in particular we can assume that after the isotopy the assumption $\varphi_i|_{D^{q-1}\times S^{p-1}}=\bar{\varphi}_i|_{D^{q-1}\times S^{p-1}}$ still holds.
		
		Again by the uniqueness of tubular neighborhoods there is $\phi\colon S^{p-1}\to\mathrm{GL}_+(q)$ so that, after applying an isotopy, $\varphi_1(x,y)=\varphi_2(\phi_y(x),y)$. This isotopy can be chosen so that the condition $\varphi_1|_{D^{q-1}\times S^{p-1}}=\varphi_2|_{D^{q-1}\times S^{p-1}}$ is preserved (cf.\ \cite[Proof of Theorem 4.5.3]{Hi76}), hence $\phi_y$ fixes $\R^{q-1}\subseteq\R^q$ pointwise for all $y\in S^{p-1}$. The subspace of $\mathrm{GL}_+(q)$ fixing $\R^{q-1}$ pointwise can be identified with $\R^{q-1}\times \R_+$ by considering the image of the last standard basis vector $(0,\dots 0,1)$. Since this space is convex, there is a smooth homotopy of $\phi$ to the constant map $\textrm{id}_{\R^q}$ within this space, so $\varphi_1$ is isotopic to $\varphi_2$.
	\end{proof}
	\begin{lemma}
		\label{L:TRIV_EMB}
		Let $\varphi_1\colon D^q\times S^{p-1}\hookrightarrow M_1^{p+q-1}$, $\varphi_2\colon D^p\times S^{q-1}\hookrightarrow M_2^{p+q-1}$ be embeddings into connected manifolds and let $M$ be the manifold
		\[M=M_1\setminus\im(\varphi_1)^\circ\cup_{S^{p-1}\times S^{q-1}} M_2\setminus\im(\varphi_2)^\circ, \]
		where for the gluing we identify the boundaries of $M_1\setminus\im(\varphi_1)^\circ$ and $M_2\setminus\im(\varphi_2)^\circ$ with $S^{p-1}\times S^{q-1}$ via $\varphi_1$ and $\varphi_2$, respectively.
		Suppose that one of $\varphi_1,\varphi_2$ is trivial and that, after choosing an orientation, both $\varphi_1$ and $\varphi_2$ are orientation-preserving if both $M_1$ and $M_2$ are orientable. Then $M$ is diffeomorphic to $M_1\# M_2$ if one of $M_1$ and $M_2$ is non-orientable and to $M_1\# (-1)^{pq-p-q}M_2$ otherwise.
	\end{lemma}
	\begin{proof}
		By possibly interchanging the roles of $\varphi_1$ and $\varphi_2$ we can assume that $\varphi_1$ is trivial. We decompose
		\begin{align}\label{EQ:CONN_SUM_DECOMP}
			M_1\cong M_1\# S^{p+q-1}\cong M_1\#((-1)^{q} (S^{q-1}\times D^p)\cup_\partial (D^q\times S^{p-1})),
		\end{align}
		the orientation on $S^{q-1}\times D^p$ is chosen so that the induced orientations on the boundary of $S^{q-1}\times D^p$ and $D^q\times S^{p-1}$ are opposite to each other, so that after gluing both pieces together we have a well-defined orientation.
		
		The embedding $\varphi\colon D^q\times S^{p-1}\hookrightarrow (S^{q-1}\times D^p)\cup_{S^{q-1}\times S^{p-1}} (D^q\times S^{p-1})$ of the second factor is trivial: The map $\bar{\varphi}$ is given by
		\[D^{q-1}\times D^p\cong (D^{q-1}\times D^p)\cup_{D^{q-1}\times S^{p-1}}(D^q_+\times S^{p-1})\hookrightarrow (S^{q-1}\times D^p)\cup_{S^{q-1}\times S^{p-1}} (D^q\times S^{p-1}),\]
		where $D^q_+\subseteq D^q$ denotes the upper half-ball and we embed $D^{q-1}\subseteq S^{q-1}$ as the upper half-sphere into the first factor (and note that in this way we obtain an embedding $D^{q-1}\subseteq S^{q-1}\subseteq D^q$ that differs from the standard embedding $D^{q-1}\subseteq D^q$).
		
		By Lemma \ref{L:TRIV_EMB_ISO}, the embeddings $\varphi$ and $\varphi_1$ are isotopic, where we consider both $\varphi$ and $\varphi_1$ as embeddings into $M_1$ via \eqref{EQ:CONN_SUM_DECOMP} and assume that the discs removed in the connected sum construction of $M_1$ and $S^{p+q-1}$ are disjoint from the images of $\varphi$, $\varphi_1$ and $\bar{\varphi}$. Hence,
		\begin{align*}
			M_1\setminus \im(\varphi_1)^\circ\cup_{S^{p-1}\times S^{q-1}} M_2\setminus\im(\varphi_2)^\circ&\cong M_1\#((-1)^{q} S^{q-1}\times D^p)\cup_{S^{q-1}\times S^{p-1}} M_2\setminus\im(\varphi_2)^\circ.
		\end{align*}
		The map $S^{p-1}\times S^{q-1}\to S^{q-1}\times S^{p-1}$, $(x,y)\mapsto(y,x)$, is orientation-preserving if and only if $(p-1)(q-1)$ is even. Hence, if $M_2$ is orientable, and if we equip it with the orientation $-(-1)^{(p-1)(q-1)}=(-1)^{pq-p-q}$, we have a well-defined orientation after gluing, so we obtain
		\begin{align*}
			M_1\setminus \im(\varphi_1)^\circ\cup_{S^{p-1}\times S^{q-1}} M_2\setminus\im(\varphi_2)^\circ&\cong M_1\#((-1)^q S^{q-1}\times D^p\cup_{S^{q-1}\times S^{p-1}}(-1)^{pq-p-q}M_2\setminus \im(\varphi_2)^\circ )\\
			&\cong M_1\# (-1)^{pq-p-q}M_2.
		\end{align*}
	\end{proof}
	\begin{proof}[Proof of Proposition \ref{P:GRAPH_RED1}.]
		First note that, by reversing the base and fiber orientations of the bundles associated to vertices in $G_i$ (resp.\ $v$) and setting $\delta(e_i)=1$ whenever $\delta(e_i)=-1$ for $i\in\{1,\dots,n\}$ (resp.\ $i=0$), we obtain a new graph (which we also denote by $G$) that satisfies $\delta(e_0)=\delta(e_1)=\dots=\delta(e_n)=1$ and leaves the diffeomorphism type of $M_G$ unchanged by Proposition \ref{P:GRAPH_RED_EDGE}. Then we have
		\begin{align}
			\notag M_{G_0}&\cong \pi_v^{-1}(B_v\setminus {D^p}^\circ)\cup_{S^{p-1}\times S^{q-1}} \pi_u^{-1}(B_u\setminus {D^q}^\circ)\\
			\notag &\cong D^p\times S^{q-1}\cup_{S^{p-1}\times S^{q-1}} \pi_u^{-1}(B_u\setminus {D^q}^\circ)\\
			\label{EQ:MG0_DECOMP}&\cong D^p\times S^{q-1}\cup_{S^{p-1}\times S^{q-1}}([0,1]\times S^{q-1}\times S^{p-1})\cup_{S^{q-1}\times S^{p-1}} \pi_u^{-1}(B_u\setminus {D^q}^\circ),
		\end{align}
		where we use similar conventions for the gluing as in the proof of Proposition \ref{P:GRAPH_RED2}.
		Since we can view the mid-part as part of the bundle $\pi_u$, the manifold $M_G$ is now obtained from $M_{G_0}$ by cutting out embeddings $\varphi_i\colon D^q\times S^{p-1}\hookrightarrow(0,1)\times S^{q-1}\times S^{p-1}$ of the form $\iota_i\times \textrm{id}_{S^{p-1}}$ for embeddings $\iota_i\colon D^q\hookrightarrow(0,1)\times S^{q-1}$, and gluing in the corresponding parts of $M_{G_i}$. We now show that all of these embeddings are trivial.
		
		For that we isotope $\iota_i$ so that $\iota_i|_{D^{q-1}}\colon D^{q-1}\hookrightarrow \{t_i\}\times D^{q-1}_i$ for an embedded disc $D^{q-1}_i\subseteq S^{q-1}$. We choose the discs $D^{q-1}_i$ so that they do not intersect. Then we define $\bar{\varphi}_i\colon D^{q-1}\times D^{p}\hookrightarrow M_{G_0}$ by identifying
		\[D^{q-1}\times D^{p}\cong D_i^{q-1}\times D^p\cup_{D_i^{q-1}\times S^{p-1}}([0,t_i]\times D_i^{q-1}\times S^{p-1})\]
		and mapping it to
		\[S^{q-1}\times D^p\cup_{S^{q-1}\times S^{p-1}}([0,1]\times S^{q-1}\times S^{p-1}) \]
		via the obvious inclusions on each part, cf.\ Figure \ref{F:TRIV_EMB2}.
		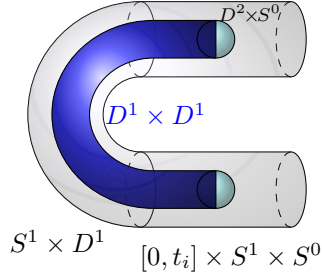
\begin{figure}
			\caption{The embeddings $\varphi_{i}$ are trivial}
			\label{F:TRIV_EMB2}
			\centering
			\begin{tikzpicture}
				\draw[dashed] (-2,1) arc (90:270:0.2 and 0.5);
				\draw (-2,0) arc (270:450:0.2 and 0.5);
				\draw[dashed] (-2,-1) arc (90:270:0.2 and 0.5);
				\draw (-2,-2) arc (270:450:0.2 and 0.5);
				\draw[dashed] (0,1) arc (90:270:0.2 and 0.5);
				\draw (0,0) arc (270:450:0.2 and 0.5);
				\draw[dashed] (0,-1) arc (90:270:0.2 and 0.5);
				\draw (0,-2) arc (270:450:0.2 and 0.5);
				\draw (-2,1) arc (180:360:1 and 0);
				\draw (-2,0) arc (180:360:1 and 0);
				\draw (-2,-1) arc (180:360:1 and 0);
				\draw (-2,-2) arc (180:360:1 and 0);
				\draw (-2,0) arc (90:270:0.5 and 0.5);
				\draw (-2,1) arc (90:270:1.5 and 1.5);
				
				\shade[ball color=blue!5!white,opacity=0.1] (0,1)--(-2,1)arc (90:270:1.5 and 1.5)--(0,-2) arc (270:450:0.2 and 0.5)--(-2,-1) arc (270:90:0.5 and 0.5)--(0,0) arc (270:450:0.2 and 0.5);
				\shade[ball color=cyan!40!white,opacity=0.2] (-1,0.5) circle (0.25);
				\shade[ball color=cyan!40!white,opacity=0.2] (-1,-1.5) circle (0.25);
				\shade[ball color=blue!80,opacity=0.65] (-1,0.75)--(-1.83,0.75)arc (90:270:1.35 and 1.25)--(-1,-1.75)--(-1,-1.25)--(-1.83,-1.25) arc (270:90:0.85 and 0.75)--(-1,0.25)--(-1,0.75);

				\draw (-1,0.5) circle (0.25);
				\draw (-1,0.75)--(-1.83,0.75);
				\draw (-1,0.25)--(-1.83,0.25);
				\draw (-1,0.75)--(-1,0.25);
				
				\draw (-1,-1.5) circle (0.25);
				\draw (-1,-1.75)--(-1.83,-1.75);
				\draw (-1,-1.25)--(-1.83,-1.25);
				\draw (-1,-1.75)--(-1,-1.25);
				
				\draw (-1.83,0.75) arc (90:270:1.35 and 1.25);
				\draw (-1.83,0.25) arc (90:270:0.85 and 0.75);
				
				\node[blue] at (-1.8,-0.5) {\small$D^1\times D^1$};
				\node at (-0.55,0.84) {\tiny$D^2\hspace{-1ex}\times \hspace{-1ex}S^0$};
				\node at (-0.8,-2.4) {\small$[0,t_i]\times S^1\times S^0$};
				\node at (-3.1,-2.2) {\small$S^1\times D^1$};
			\end{tikzpicture}
		\end{figure}
		Since the discs $D^{q-1}_i$ do not intersect each other, the maps $\bar{\varphi}_i$ have pairwise disjoint image. Hence, if we equip each $M_{G_i}$ with the orientation induced from each $E_v$, which equals the orientation induced from $(-1)^{pq-p-q}E_u$, it follows from Lemma \ref{L:TRIV_EMB} that
		\[M_G\cong M_{G_0}\# M_{G_1}\#\dots\# M_{G_n}. \]
	\end{proof}

	\subsection{Fundamental Group, Cohomology and Characteristic Classes}
	\label{SS:PLUMBING_TOP}
	
	The goal of this section is to determine the fundamental group, the cohomology ring and the characteristic classes of $M_G$. As in the previous section we fix a geometric plumbing graph $G$.
	
	\begin{lemma}
		\label{L:FUND_GROUP}
		Suppose $p,q>2$ and that every connected component of $G$ and every $B_u$ and $B_v$ are simply-connected. Then $M_G$ is simply-connected.
	\end{lemma}
	\begin{proof}
		If $G_1,\dots,G_n$ denote the connected components of $G$, then $M_G\cong M_{G_1}\#\dots\# M_{G_n}$ and hence
		\[\pi_1(M_G)\cong \pi_1(M_1)*\dots*\pi_1(M_n). \]
		Thus, we can restrict to the case where $G$ is connected.
		
		The manifold $M_G$ is obtained by gluing the spaces $\pi^{-1}_u(B_u\setminus(\bigsqcup_{\deg(u)}D^q)^\circ)$ and $\pi^{-1}_v(B_v\setminus(\bigsqcup_{\deg(v)}D^p)^\circ)$ according to $G$.  These spaces are fiber bundles with base spaces $B_u$ or $B_v$ with a finite number of discs removed. Since, $p,q>2$, it follows from the long exact sequence for fiber bundles that all these spaces are simply-connected.
		
		The graph is simply-connected, hence, after choosing a root, it is a tree. If $M_k$ is the manifold obtained by gluing according the subgraph of $G$ consisting of all vertices of distance at most $k$ from the root, then it follows inductively from van Kampen's theorem, that $M_k$ is simply-connected. For $k$ large enough we have $M_k=M_G$, and hence $M_G$ is simply-connected.
	\end{proof}
	\begin{remark}
		If one does not require that $G$ is simply-connected in Lemma \ref{L:FUND_GROUP}, then, if $G$ is connected, one can show, by using the groupoid version of van Kampen's theorem, that $\pi_1(M_G)\cong \pi_1(G)$.
	\end{remark}
	
	For the cohomology we fix a commutative unital ring $R$. For $u\in U$ recall from \eqref{EQ:BDL_COHOM_SPLIT} that if $e_R(\pi_u)=0$, then we have
	\begin{equation*}
		H^i(E_u)=\pi^*_u(H^i(B_u))\oplus \theta_{a_u}(H^{i-p+1}(B_u))
	\end{equation*}
	for an element $a_u\in H^{p-1}(E_u)$ with $\psi(a_u)=1$. We make the following assumption:
	\begin{equation}
		\label{EQ:GRAPH_CONDITION}
		G \text{ is simply-connected and either no vertex }v\in V\text{ is a leaf, or }G\text{ is of the form }\,\,\begin{tikzpicture}[scale=0.6,, transform shape]
			\begin{scope}[every node/.style={circle,draw,minimum height=2em}]
				\node (V) at (-0.5,0) {$v$};
				\node (U) at (1,0) {$u$};
			\end{scope}
			\path[-](V) edge["{$\scriptstyle e$}"] (U);
		\end{tikzpicture}.
	\end{equation}
	\begin{theorem}
		\label{T:MG_COHOM}
		Let $G$ be a geometric plumbing graph with $p>1$ satisfying \eqref{EQ:GRAPH_CONDITION}, so that $e_R(\pi_u)=0$ and $B_u$ is closed for all $u\in U$, and each bundle $\pi_v$ is trivial with $B_v\cong S^p$. Then the cohomology ring $H^*(M_G)$ is a quotient of a subring of $\bigoplus_{u\in U}H^*(E_u)$ as follows:
		\[
		H^i(M_G)\cong
		\begin{cases}
			\left(\sum\limits_{u\in U}1_{H^0(E_u)}\right)R\cong R,\quad &i=0;\\
			\bigoplus\limits_{u\in U} \pi_u^*H^{p-1}(B_u)\\
			\quad\oplus \left\{\sum\limits_{u\in U}\lambda_u\cdot a_u \;\middle|\;\lambda_u\in R\; \forall u\in U,\, \sum\limits_{(u,v)\in E}\delta(e)\lambda_u=0\;\forall v\in V \right\},\quad&i=p-1;\\
			\vbigslant{\bigoplus\limits_{u\in U}H^q(E_u)}{\bigoplus\limits_{v\in V}\left(\sum\limits_{e=(u,v)\in E}\delta(e)\pi_u^*[B_u]^*\right)R},\quad &i=q;\\
			\vbigslant{\bigoplus\limits_{u\in U}H^{p+q-1}(E_u)}{\left\{\sum\limits_{u\in U}\lambda_u[E_u]^*\;\middle|\; \sum\limits_{u\in U}\lambda_u=0 \right\}}\cong R,\quad &i=p+q-1;\\
			\bigoplus\limits_{u\in U} H^i(E_u),\quad &\text{else,}
		\end{cases}
		\]
		if $q\neq p-1$, and
		\[
		H^i(M_G)\cong
		\begin{cases}
			\left(\sum\limits_{u\in U}1_{H^0(E_u)}\right)R\cong R,\quad &i=0;\\
			\vbigslant{\bigoplus\limits_{u\in U} \pi_u^*H^{p-1}(B_u)}{\bigoplus\limits_{v\in V}\left(\sum\limits_{e=(u,v)\in E}\delta(e)\pi_u^*[B_u]^*\right)R}\\
			\quad\oplus\left\{\sum\limits_{u\in U}\lambda_u\cdot a_u\;\middle|\;\lambda_u\in R\;\forall u\in U,\, \sum\limits_{(u,v)\in E}\delta(e)\lambda_u=0\;\forall v\in V \right\} ,\quad &i=q;\\
			\vbigslant{\bigoplus\limits_{u\in U}H^{p+q-1}(E_u)}{\left\{\sum\limits_{u\in U}\lambda_u[E_u]^*\;\middle|\; \sum\limits_{u\in U}\lambda_u=0 \right\}}\cong R,\quad &i=2q;\\
			\bigoplus\limits_{u\in U} H^i(E_u),\quad &\text{else,}
		\end{cases}
		\]
		if $q=p-1$.	The cup product structure is induced from $\bigoplus_{u\in U}H^*(E_u)$.
	\end{theorem}
	\begin{remark}
		If, in the situation of Theorem \ref{T:MG_COHOM}, we merely require that $G$ has simply-connected connected components instead of \eqref{EQ:GRAPH_CONDITION}, we can apply Proposition \ref{P:GRAPH_RED1} to split it into a disjoint union of subgraphs that all either satisfy \eqref{EQ:GRAPH_CONDITION}, or consist of a single vertex in $V$. Thus, we can compute the cohomology ring of $M_G$ by applying Theorem \ref{T:MG_COHOM} to the corresponding components of $G$ (provided that the assumptions on the bundles as in Theorem \ref{T:MG_COHOM} are satisfied).
	\end{remark}

	\begin{theorem}
		\label{T:MG_CHAR_CLASSES}
		Let $c$ be a stable characteristic class. Let $G$ be a geometric plumbing graph satisfying the assumptions of Theorem \ref{T:MG_COHOM}. If $\xi_u$ denotes the vector bundle corresponding to the bundle $\pi_u$, then
		\[c(TM_G)=\sum\limits_{u\in U}\pi_u^*c(\xi_u\oplus TB_u), \]
		where we identified $H^*(M_G)$ with a quotient of a subring of $\bigoplus_{u\in U}H^*(E_u)$ according to Theorem \ref{T:MG_COHOM}.
	\end{theorem}
	
	Applying Theorems \ref{T:MG_COHOM} and \ref{T:MG_CHAR_CLASSES} in dimension $6k$ yields the following corollary.
	
	\begin{corollary}
		\label{C:MG_COHOM_6KDIM}
		In the setting of Theorem \ref{T:MG_COHOM} let $p=2k+1$, $q=4k$, i.e.\ $M_G$ has dimension $6k$. Let every $B_u$ be simply-connected with torsion-free homology. Then the following assertions hold:
		\begin{enumerate}
			\item If $G$ is the graph\,\, \begin{tikzpicture}[scale=0.5,, transform shape]
				\begin{scope}[every node/.style={circle,draw,minimum height=2em}]
					\node (V) at (-0.5,0) {$v$};
					\node (U) at (1,0) {$u$};
				\end{scope}
				\path[-](V) edge["{$\scriptstyle e$}"] (U);
			\end{tikzpicture}, then $M_G$ is a simply-connected $6k$-dimensional  manifold with torsion-free homology and invariants 
			\[(H^{2k}(M_G;\Z),\mu_{M_G},w_2(M_G)^k,p_k(M_G))=(H^{2k}(B_u;\Z),0,(w_2(B_u)+w_2(\pi_u))^k,0).\]
			If $k=1$, then $b_3(M_G)=0$.
			\item If $G$ has no vertex in $V$ that is a leaf, then $M_G$ is a simply-connected $6k$-dimensional manifold with torsion-free homology. Further, define $A=\bigoplus_{u\in U}H^{2k}(E_u;\Z)$, $\mu=\sum_{u\in U} \mu_{E_u}$ and $p=\sum_{u\in U}p_k(E_u)$ (viewed as a linear form on $A$). Then $H^{2k}(M_G;\Z)$ is given by
			\[\bigoplus\limits_{u\in U}\pi_u^*H^{2k}(B_u;\Z)\oplus\left\{\sum\limits_{u\in U}\lambda_u\cdot a_u \;\middle|\; \lambda_u\in \Z\;\forall u\in U,\, \sum\limits_{e=(u,v)\in E}\delta(e)\lambda_u=0\;\forall v\in V \right\}\subseteq A, \]
			and we have $w_2(M_G)^k=\sum_{u\in U}\pi_u^*(w_2(B_u)+w_2(\pi_u))^k$ and $\mu_{M_G}$ and $p_k(M_G)$ are the restrictions of $\mu$ and $p$ to $H^{2k}(M_G;\Z)$, respectively. If $k=1$, then $b_3(M_G)=0$.
		\end{enumerate}
	\end{corollary}
	
	The rest of this section consists of the proofs of Theorems \ref{T:MG_COHOM} and \ref{T:MG_CHAR_CLASSES}.
	
	\begin{lemma}
		\label{L:Snk}
		Let $S_k^n=S^n\setminus(\bigsqcup_k D^n)^\circ$ with $n\geq2$. Then the inclusion $\iota\colon\bigsqcup_k S^{n-1}\hookrightarrow S^n_k$ of the boundary induces an injective map
		\[
		\begin{tikzcd}
			\iota^*\colon H^{n-1}(S^n_k)\to H^{n-1}(\bigsqcup_k S^{n-1})\cong \bigoplus\limits_k H^{n-1}(S^{n-1})
		\end{tikzcd}
		\]
		with image generated by the elements $a_j-a_i$, where $a_i$ is a positively oriented generator of the $i$-th copy of $H^{n-1}(S^{n-1})$ with respect to the orientation induced by $S^n_k$.
	\end{lemma}
	\begin{proof}
		We use Lefschetz duality to obtain the following commutative diagram with exact rows and where the vertical arrows are isomorphisms:
		\[
		\begin{tikzcd}
			H^{n-1}(S^n_k,\bigsqcup_k S^{n-1})\arrow{d}{\cong}\arrow{r}&H^{n-1}(S^n_k)\arrow{r}{\iota^*}\arrow{d}{\cong}&H^{n-1}(\bigsqcup_k S^{n-1})\arrow{r}\arrow{d}{\cong}& H^n(S^n_k,\bigsqcup_k S^{n})\arrow{r}\arrow{d}{\cong}&0\\
			H_1(S^n_k)\arrow{r}&H_1(S^n_k,\bigsqcup_k S^{n-1})\arrow{r}&H_0(\bigsqcup_k S^{n-1})\arrow{r}&H_0(S^n_k)\arrow{r}&0
		\end{tikzcd}
		\]
		If $n\geq3$, then $S^n_k$ is simply-connected. If $n=2$, then the map $H_1(S^2_k)\to H_1(S^2_k,\bigsqcup_k S^{1})$ is trivial as any closed loop in $S^2_k$ is homologous to the sum of loops in $\bigsqcup_k S^{1}$ it encloses. Hence, the map $\iota^*$ is injective. To compute the image, it suffices to determine the kernel of the map $H_0(\bigsqcup_k S^{n-1})\to H_0(S^n_k)$. The claim now follows from the fact that we can join any two boundary components by a path and that the vertical map $H^{n-1}(\bigsqcup_k S^{n-1})\to H_0(\bigsqcup_k S^{n-1})$ is given by Poincar\'e duality.
	\end{proof}

	For each $e=(u,v)\in E$ we denote the embeddings along which the spaces $\overline{E}_u$ and $\overline{E}_v$ get glued by $\varphi_{(u,v)}\colon D^q\times D^p\hookrightarrow \overline{E}_u$ and $\varphi_{(v,u)}\colon D^p\times D^q\hookrightarrow \overline{E}_v$, respectively and by $I^\pm_{p,q}\colon D^p\times D^q\to D^q\times D^p$ the diffeomorphism
	\[I^\pm_{p,q}(x_1,\dots,x_p,y_1,\dots,y_q)=(\pm y_1,y_2,\dots,y_q,\pm x_1,x_2,\dots,x_p), \]
	so that $\overline{M}_G$ is the result of the following pushout:
	\begin{equation*}
		\begin{tikzcd}[row sep=huge, column sep = huge]
			\bigsqcup\limits_{e\in E} D^q\times D^p\arrow[hookrightarrow]{rr}{\bigsqcup\limits_{e=(u,v)\in E} \varphi_{(v,u)}\circ I^{\delta(e)}_{q,p} }\arrow[hookrightarrow]{d}{\bigsqcup\limits_{(u,v)\in E} \varphi_{(u,v)}} && \bigsqcup\limits_{v\in V} \overline{E}_v \arrow[hookrightarrow]{d}\\
			\bigsqcup\limits_{u\in U}\overline{E}_u\arrow[hookrightarrow]{rr}&& \overline{M}_G
		\end{tikzcd}
	\end{equation*}
	
	Now let \[X=\overline{M}_G\setminus \bigcup_{u\in U}\intr(\overline{E}_u)=\bigcup\limits_{v\in V} \overline{E}_v\setminus\left(\bigcup\limits_{(u,v)\in E}\varphi_{(v,u)}({D^p}^\circ\times D^q)\right)\cup \bigcup_{u\in U}E_u,\]
	where we considered the spaces $\overline{E}_v$ and $\overline{E}_u$ as subspaces of $\overline{M}_G$. Alternatively, the space $X$ is the result of the following pushout:
	\begin{equation}
		\label{EQ:PUSHOUT_X}
		\begin{tikzcd}[row sep=huge, column sep = huge]
			\bigsqcup\limits_{e\in E} D^q\times S^{p-1}\arrow[hookrightarrow]{rr}{\bigsqcup\limits_{e=(u,v)\in E} \varphi_{(v,u)}\circ I^{\delta(e)}_{q,p} }\arrow[hookrightarrow]{d}{\bigsqcup\limits_{(u,v)\in E} \varphi_{(u,v)}} && \bigsqcup\limits_{v\in V} \overline{E}_v\setminus\left(\bigcup\limits_{(u,v)\in E}\varphi_{(v,u)}({D^p}^\circ\times D^q)\right) \arrow[hookrightarrow]{d}\\
			\bigsqcup\limits_{u\in U}E_u\arrow[hookrightarrow]{rr}&& X
		\end{tikzcd}
	\end{equation}

	For every $v\in V$, by assumption, the manifold $\overline{E}_v\setminus(\bigcup_{(u,v)\in E}\varphi_{(v,u)}({D^p}^\circ\times D^q))$ can be identified with $S^p_{\deg(v)}\times D^q$.
	
\begin{lemma}
	\label{L:X_COHOM_INJ}
	Suppose that $G$ satisfies \eqref{EQ:GRAPH_CONDITION}. Then the inclusion $\bigsqcup_{u\in U} E_u\hookrightarrow X$ induces an injective map $H^i(X)\to \bigoplus_{u\in U}H^i(E_u)$ for every $i$ with image given by
	\begin{itemize}
		\item $\left(\sum\limits_{u\in U}1_{H^0(E_u)}\right)R\cong R$, \quad if $i=0$;
		\item $\bigoplus\limits_{u\in U} \pi_u^*H^{p-1}(B_u)\oplus \left\{\sum\limits_{u\in U}\lambda_u\cdot a_u \;\middle|\;\lambda_u\in R\;\forall u\in U, \sum\limits_{(u,v)\in E}\delta(e)\lambda_u=0\;\forall v\in V \right\}$,\quad if $i=p-1$;
		\item $\bigoplus\limits_{u\in U} H^i(E_u)$,\quad \text{else.}
	\end{itemize}
\end{lemma}
	We will henceforth identify the cohomology of $X$ with its image in $\bigoplus_{u\in U}H^*(E_u)$.
	\begin{proof}
		The Mayer-Vietoris sequence for the pushout \eqref{EQ:PUSHOUT_X} is given as follows:
		\[
		\dots\longrightarrow H^i(X)\longrightarrow\bigoplus\limits_{v\in V}H^i(S^p_{\deg(v)})\bigoplus\limits_{u\in U}H^i(E_u)\longrightarrow \bigoplus\limits_{e\in E}H^i(S^{p-1})\longrightarrow H^{i+1}(X)\longrightarrow\cdots
		\]
		To simplify notation, we will write $u$, $v$ or $e$ for a canonical generator of a group $H^i(M^n)$ (e.g.\ $[M]^*$ if $i=n$ or $1$ if $i=0$), where $M$ is related to $u\in U$, $v\in V$ or $e\in E$ (e.g.\ as the fiber, base or total space of $\pi_u$ or $\pi_v$).
		
		For $i=0$ the middle map of the sequence is given by
		\begin{align*}
			\bigoplus\limits_{v\in V}R v\bigoplus\limits_{u\in U}R u\longrightarrow \bigoplus\limits_{e\in E}R e,\quad v\mapsto \sum\limits_{e=(u,v)\in E} e,\quad u\mapsto -\sum\limits_{e=(u,v)\in E} e.
		\end{align*}
		This is the linear map associated to the incidence matrix $Q(G)$, when we view $G$ as a directed graph with edges originating from vertices in $V$ and ending in vertices in $U$. Since $G$ is simply-connected, it follows from Lemmas \ref{L:GRAPH_RANK_N-1} and \ref{L:GRAPH_SURJ} that this map is surjective with kernel generated by $\sum_{u\in U} u+\sum_{v\in V} v$. In particular, the image of $H^0(X)$ in $\bigoplus_{u\in U}H^0(E_u)$ is generated by
		\[\sum\limits_{u\in U} u=\sum\limits_{u\in U} 1_{H^0(E_u)}.\]
		
		For $1\leq i<p-1$ or $i\geq p$ the groups $H^i(S^p_k)$ and $H^i(S^{p-1})$ vanish. By exactness, the map $H^i(X)\to \bigoplus_{u\in U} H^i(E_u)$ is an isomorphism for $1\leq i<p-1$ and for $i>p$ (for $i=1$ this follows from the surjectivity of the map in degree $0$).
		
		Finally, we have to investigate the following part of the sequence:
		\begin{align*}
			0\longrightarrow H^{p-1}(X)\longrightarrow\bigoplus\limits_{v\in V}H^{p-1}(S^p_{\deg(v)})\bigoplus\limits_{u\in U}H^{p-1}(E_u)\longrightarrow &\bigoplus\limits_{e\in E}H^{p-1}(S^{p-1})\\
			&\longrightarrow H^{p}(X)\longrightarrow\bigoplus\limits_{u\in U}H^p(E_u)\longrightarrow 0.
		\end{align*}
		By \eqref{EQ:BDL_COHOM_SPLIT} and by writing $e=[S^{p-1}]*\in H^{p-1}(S^{p-1})$, viewed as the summand corresponding to $e\in E$, we can write the third map in this sequence as
		\[\bigoplus\limits_{v\in V}H^{p-1}(S^p_{\deg(v)})\bigoplus\limits_{u\in U}\pi^*_u(H^{p-1}(B_u))\oplus R a_u\longrightarrow \bigoplus\limits_{e\in E}R e. \]
		We denote this map by $\phi$. By Lemma \ref{L:Snk}, for every $v\in V$ there are generators $a^v_{u_1u_2}$, $u_1,u_2\in U$  with $(u_1,v),(u_2,v)\in E$, of $H^{p-1}(S^p_{\deg(v)})$ (note that they are not linearly independent), so that we have
		\begin{align*}
			\phi(a^v_{u_1u_2})&= \delta(u_1,v)(u_1,v)-\delta(u_2,v)(u_2,v),\quad&&\text{for }v\in V,\\
			\phi(x)&=0,\quad&&\text{for }x\in \pi^*_u(H^{p-1}(B_u)),u\in U,\\
			\phi(a_u)&= \sum\limits_{e=(u,v)\in E}e,\quad &&\text{for }u\in U.
		\end{align*}
		The last line follows from \eqref{EQ:a_FIBR_INCL}.
		
		If $G$ is the graph \begin{tikzpicture}[scale=0.5,, transform shape]
			\begin{scope}[every node/.style={circle,draw,minimum height=2em}]
				\node (V) at (-0.5,0) {$v$};
				\node (U) at (1,0) {$u$};
			\end{scope}
			\path[-](V) edge["{$\scriptstyle e$}"] (U);
		\end{tikzpicture}, then $S^p_{\deg(v)}$ is contractible, so $H^{p-1}(S^p_{\deg(v)})$ is trivial and $\phi(a_u)=e$. Hence, the map $\phi$ is surjective with kernel $\pi^*_u(H^{p-1}(B_u))$.
		
		Now suppose that no vertex in $V$ is a leaf. Then for any edge $e=(u,v)\in E$, that is connected to a leaf $u\in U$, we have $\phi(a_u)=e$. Hence, for any edge $e^\prime=(u^\prime,v)\in E$ connected to $v$, we have
		\[\phi(\delta(u^\prime,v)a^v_{u^\prime u}+\delta(u,v)a_u)=e^\prime.\]
		Hence, by induction over the distance to the root, we see that the map $\phi$ is surjective.
		
		We have that $\bigoplus_{u\in U}\pi^*_u(H^{p-1}(B_u))$ is contained in the kernel of $\phi$. Further, the image of the restriction $\bigoplus_{v\in V}H^{p-1}(S^p_{\deg(v)})\longrightarrow\bigoplus_{e\in E}R e$ of $\phi$ is given by 
		\[\left\{\sum\limits_{e\in E}\lambda_e e \;\middle|\;\lambda_e\in R\;\forall e\in E, \sum\limits_{e=(u,v)\in E}\delta(e)\lambda_e=0\;\forall v\in V \right\}. \]
		Thus, to determine the kernel of $\phi$, we need to determine all elements $\sum_{u\in U}\lambda_u a_u\in\bigoplus_{u\in U}R a_u$, that get mapped into this set via $\phi$. This is precisely the set
		\[\left\{\sum\limits_{u\in U}\lambda_u a_u \;\middle|\; \lambda_u\in R\;\forall u\in U,  \sum\limits_{e=(u,v)\in E}\delta(e)\lambda_u=0\;\forall v\in V \right\}, \]
		which is the projection of the kernel of $\phi$ to $\bigoplus_{u\in U}R a_u$. This finishes the proof.
	\end{proof}

\begin{lemma}
	\label{L:X_COHOM_SURJ}
	Suppose that $G$ satisfies \eqref{EQ:GRAPH_CONDITION}. Then the inclusion $M_G\hookrightarrow X$ induces a surjective map $H^i(X)\to H^i(M_G)$ with kernel given by
	\[
	\ker(H^i(X)\to H^i(M_G))=\begin{cases}
		\bigoplus\limits_{v\in V}\left(\sum\limits_{e=(u,v)\in E}\delta(e)\pi_u^*[B_u]^*\right)R,\quad &i=q;\\
		\left\{\sum\limits_{u\in U}\lambda_u[E_u]^* \;\middle|\;\lambda_u\in R\;\forall u\in U,  \sum\limits_{u\in U}\lambda_u=0 \right\},\quad &i=p+q-1;\\
		0,\quad &\text{else.}
	\end{cases}
	\]
\end{lemma}
\begin{proof}
	By excision, the cohomology of the pair $(X,M_G)$ is isomorphic to the cohomology of the pair
	\[\left(\bigsqcup\limits_{v\in V}D^q\times S^p_{\deg(v)},\bigsqcup\limits_{v\in V}S^{q-1}\times S^p_{\deg(v)} \right). \]
	By the long exact sequence, this pair only has non-vanishing cohomology groups in degrees $q$ and $p+q-1$. We first consider degree $q$. The inclusion $X\hookrightarrow \overline{M}_G$ then gives the following commutative square:
	\[
	\begin{tikzcd}
		H^q(\overline{M}_G,M_G)\arrow{r}\arrow{d}&H^q(\overline{M}_G)\arrow{d}\\
		H^q(X,M_G)\arrow{r}&H^q(X)
	\end{tikzcd}
	\]
	For this commutative diagram we will now prove the following claims:
	\begin{enumerate}
		\item $H^q(\overline{M}_G,M_G)\cong\bigoplus_{u\in U}H_p(B_u)\bigoplus_{v\in V}R v$ and $H^q(\overline{M}_G)\cong\bigoplus_{u\in U}R u\bigoplus_{v\in V} H^q(S^p)$.
		\item The map $H^q(\overline{M}_G,M_G)\to H^q(X,M_G)$ restricted to $\bigoplus_{v\in V}R v$ is an isomorphism.
		\item The map $H^q(\overline{M}_G,M_G)\to H^q(\overline{M}_G)$ maps $v\in V$ to $\sum_{e=(u,v)\in E}\delta(e)u$.
		\item The map $H^q(\overline{M}_G)\to H^q(X)$ maps $u\in U$ to $\pi^*_{u}[B_u]^*$.
	\end{enumerate}
	
	(1) By Lefschetz duality we have $H^q(\overline{M}_G,M_G)\cong H_p(\overline{M}_G)$, so both isomorphisms follow from the fact that $\overline{M}_G$ is homotopy equivalent to $\bigvee_{u\in U}B_u\bigvee_{v\in V}S^p$.
	
	(2) We have the following commutative diagram of inclusions:
	\begin{equation}
		\label{EQ:XMG_DIAG}
		\begin{tikzcd}
			\left(\bigsqcup\limits_{v\in V}S^p_{\deg(v)}\times D^q,\bigsqcup\limits_{v\in V} S^p_{\deg(v)}\times S^{q-1} \right)\arrow[hookrightarrow]{r}\arrow[hookrightarrow]{rd}&(X,M_G)\arrow[hookrightarrow]{d}\\
			&(\overline{M}_G,M_G)
		\end{tikzcd}
	\end{equation}
	and Lefschetz duality gives the following commutative square.
	\[
	\begin{tikzcd}
		H^q(\overline{M}_G,M_G)\arrow{r}\arrow{d}{\cong}& \bigoplus\limits_{v\in V} H^q(S^p_{\deg(v)}\times D^q,S^p_{\deg(v)}\times S^{q-1})\arrow{d}{\cong}\\
		H_p(\overline{M}_G)\arrow{r}&\bigoplus\limits_{v\in V} H_p(S^p_{\deg(v)}\times D^q,\bigsqcup_{\deg(v)} S^{p-1}\times D^q)
	\end{tikzcd}
	\]
	The elements $v\in H^q(\overline{M}_G,M_G)$ in (1) are represented by $B_v\cong S^p\subseteq \overline{M}_G$ in $H_p(\overline{M}_G)\cong H^q(\overline{M}_G,M_G)$. Each class $[B_v]$ also represents a generator of
	\begin{align*}
		H_p\left(S^p_{\deg(v)}\times D^q,\bigsqcup\nolimits_{\deg(v)} S^{p-1}\times D^q\right)\cong H_p\left(S^p,\bigsqcup\nolimits_{\deg(v)}D^p\right)\cong H_p(S^p).
	\end{align*}
	Hence, when restricted to $\bigoplus_{v\in V}R v$, the lower horizontal map is an isomorphism. The claim now follows from the commutativity of \eqref{EQ:XMG_DIAG} and that the map
	\[\left(\bigsqcup_{v\in V}S^p_{\deg(v)}\times D^q,\bigsqcup_{v\in V} S^p_{\deg(v)}\times S^{q-1} \right)\hookrightarrow(X,M_G)\]
	induces an isomorphism on cohomology by excision.
	
	(3) By Lefschetz duality we have the following commutative diagram.
	\[
	\begin{tikzcd}
		H^q(\overline{M}_G,M_G)\arrow{r}\arrow{d}{\cong}&H^q(\overline{M}_G)\arrow{d}{\cong}\\
		H_p(\overline{M}_G)\arrow{r}&H_p(\overline{M}_G,M_G)
	\end{tikzcd}
	\]
	By (1), for each $v\in V$, the element $v\in H^q(\overline{M}_G,M_G)$ maps to the homology class which is represented by $B_v\cong S^p\subseteq \overline{M}_G$. By isotoping the embedding of the zero-section $B_v\subseteq \overline{E}_v\subseteq \overline{M}_G$ to the boundary of the disc bundle (which is possible as the bundle $\pi_v$ is trivial), we see that the class of this embedding in $H_p(\overline{M}_G,M_G)$ is represented by the sum of embeddings of fibers of all $\overline{E}_u$ for which $(u,v)\in E$, each class multiplied by the sign $\delta(u,v)$. Each embedding of a fiber of $\overline{E}_u$ represents the dual to the class represented by the embedding of the zero-section $B_u\subseteq \overline{E}_u$. By commutativity of the diagram, it follows that $v\in H^q(\overline{M}_G,M_G)$ gets mapped to $\sum_{e=(u,v)\in E}\delta(e)u\in H^q(\overline{M}_G)$.
	
	(4) The diagram of maps
	\[
	\begin{tikzcd}
		\bigsqcup\limits_{u\in U}B_u\arrow[hookrightarrow]{r}&\overline{M}_G\\
		\bigsqcup\limits_{u\in U}E_u\arrow[hookrightarrow]{r}\arrow{u}{\sqcup_u\pi_u}&X\arrow[hookrightarrow]{u}
	\end{tikzcd}
	\]
	induces the following commutative diagram, where the lower horizontal map is injective by Lemma \ref{L:X_COHOM_INJ}.
	\[
	\begin{tikzcd}
		\bigoplus\limits_{u\in U}H^q(B_u)\arrow{d}{\oplus_u\pi_u^*}&H^q(\overline{M}_G)\arrow{d}\arrow{l}\\
		\bigoplus\limits_{u\in U} H^q(E_u)\arrow[hookleftarrow]{r}&H^q(X)
	\end{tikzcd}
	\]
	For each $u\in U$, the element $u\in H^q(\overline{M}_G)$ gets mapped to $[B_u]^*\in H^q(B_u)$, which in turn gets mapped to $\pi_u^*[B_u]^*\in H^q(E_u)$. This proves the claim.
	\leavevmode\\
	
	Combining claims (1)--(4) it follows that the map $H^q(X,M_G)\longrightarrow H^q(X)$ is given by
	\begin{align*}
		\bigoplus\limits_{v\in V}R v&\longrightarrow H^q(X)\subseteq \bigoplus\limits_{u\in U}H^q(E_u)\\
		v &\longmapsto \sum\limits_{e=(u,v)\in E}\delta(e)\pi_u^*[B_u]^*.
	\end{align*}
	This map is injective; this is clear if $G$ is of the form \begin{tikzpicture}[scale=0.5,, transform shape]
		\begin{scope}[every node/.style={circle,draw,minimum height=2em}]
			\node (V) at (-0.5,0) {$v$};
			\node (U) at (1,0) {$u$};
		\end{scope}
		\path[-](V) edge["{$\scriptstyle e$}"] (U);
	\end{tikzpicture}. If no vertex is a leaf, then this follows from Lemma \ref{L:BG_RANK}.
	
	To summarize, we showed for $i<p+q-1$ that the map $H^i(X)\to H^i(M_G)$ is surjective except possibly for $i=p+q-2$ and only has a non-trivial kernel when $i=q$, which is then generated by the elements $\sum_{e=(u,v)\in E}(\delta(e)\pi_u^*[B_u]^*)$ for $v\in V$.
	
	It remains to consider the case $i=p+q-1$. For each $u\in U$ we have the following commutative diagram of maps of pairs:
	\[	\begin{tikzcd}
		(M_G,\emptyset)\arrow[hookrightarrow]{r}\arrow[hookrightarrow]{d}&(X,\emptyset)\arrow[hookleftarrow]{d}\\
		\left(M_G,M_G\setminus \left(E_u\setminus \bigsqcup\limits_{(u,v)\in E}\varphi_{(u,v)}(D^q\times S^{p-1})\right)\right)\arrow[hookleftarrow]{d}&(E_u,\emptyset)\arrow[hookrightarrow]{d}\\
		\left(E_u\setminus\hspace{-2ex}\bigsqcup\limits_{(u,v)\in E}\hspace{-2ex}\varphi_{(u,v)}(D^q\times S^{p-1})^\circ,\hspace{-2ex}\bigsqcup\limits_{(u,v)\in E}\hspace{-2ex}\varphi_{(u,v)}(S^{q-1}\times S^{p-1})\right)\arrow[hookrightarrow]{r}&\left(E_u,\hspace{-2ex}\bigsqcup\limits_{(u,v)\in E}\hspace{-2ex}\varphi_{(u,v)}(D^q\times S^{p-1})\right)
	\end{tikzcd}
	\]
	The maps not involving $X$ are all orientation-preserving maps that induce isomorphisms on $H^{p+q-1}$ by excision and the long exact sequence. Hence, each $[E_u]^*\in H^{p+q-1}(X)$ gets mapped to the dual of the fundamental class of $M_G$ under the induced map $H^{p+q-1}(X)\longrightarrow H^{p+q-1}(M_G)$. In particular, the kernel of this map is given by
	\[\left\{\sum\limits_{u\in U}\lambda_u[E_u]^* \;\middle|\;\lambda_u\in U\;\forall u\in U, \sum\limits_{u\in U}\lambda_u=0 \right\}. \]
	Further, by the long exact sequence of the pair
	\[\left(\bigsqcup\limits_{v\in V}S^p_{\deg(v)}\times D^q,\bigsqcup\limits_{v\in V}S^p_{\deg(v)}\times S^{q-1} \right),\]
	the $R$-module $H^{p+q-1}(X,M_G)$ is free of rank $|E|-|V|$. Since the kernel of the map
	\[H^{p+q-1}(X)\longrightarrow H^{p+q-1}(M_G),\] 
	which is the image of the map $H^{p+q-1}(X,M_G)\to H^{p+q-1}(X)$, has rank $|U|-1$, and since in a tree we have $|U|+|V|=|E|+1$, i.e.\ $|E|-|V|=|U|-1$, it follows that $H^{p+q-1}(X,M_G)$ injects into $H^{p+q-1}(X)$ as any surjective $R$-module homomorphism $R^{|U|-1}\to R^{|U|-1}$ is necessarily injective. This shows that the boundary map $H^{p+q-2}(M_G)\to H^{p+q-1}(X,M_G)$ is trivial and hence the map $H^{p+q-2}(X)\to H^{p+q-2}(M_G)$ is surjective. This finishes the proof.
\end{proof}

We are now ready to prove Theorems \ref{T:MG_COHOM} and \ref{T:MG_CHAR_CLASSES}, and Corollary \ref{C:MG_COHOM_6KDIM}.

\begin{proof}[Proof of Theorem \ref{T:MG_COHOM}.]
	This directly follows from Lemmas \ref{L:X_COHOM_INJ} and \ref{L:X_COHOM_SURJ}.
\end{proof}
\begin{proof}[Proof of Theorem \ref{T:MG_CHAR_CLASSES}.]
	For $u\in U$ we have the inclusion $\overline{E}_u\hookrightarrow\overline{M}_G$, and, by naturality, the induced map on cohomology maps $c(T\overline{M}_G)$ to $c(T\overline{E}_u)$. If $\xi_u$ denotes the vector bundle corresponding to $\pi_u$, then the tangent bundle of $\overline{E}_u$ decomposes as
	\[T\overline{E}_u\cong \pi_u^*(\xi_u\oplus TB_u), \]
	cf.\ Proposition \ref{P:SPHERE_BDL_CHAR_CLASSES}. Hence,
	\[c(T\overline{E}_u)=\pi_u^*c(\xi_u\oplus TB_u). \]
	Since $\overline{M}_G\simeq\bigvee_{u\in U}B_u\bigvee_{v\in V}B_v$ and all bundles $\pi_v$ are trivial, it follows that
	\[c(T\overline{M}_G)=\sum\limits_{u\in U}c(T\overline{E}_u)=\sum\limits_{u\in U}\pi_u^*c(\xi_u\oplus TB_u). \]
	Now $M_G=\partial\overline{M}_G$, and we denote the inclusion $M_G\hookrightarrow\overline{M}_G$ by $\iota$. Then
	\[ \iota^*T\overline{M}_G\cong TM_G\oplus\underline{\R}_{M_G}, \]
	the trivial factor corresponding to the normal bundle of $M_G$ in $\overline{M}_G$. Hence, by the stability of $c$, we have
	\[c(TM_G)=\iota^*\left(\sum\limits_{u\in U}\pi_u^*c(\xi_u\oplus TB_u) \right). \]
	To determine this element in $H^*(M_G)$ consider the following commutative diagram:
	\[
	\begin{tikzcd}
		M_G\arrow[hookrightarrow]{r}\arrow[hookrightarrow]{d}{\iota}&X\arrow[hookleftarrow]{r}&\bigsqcup\limits_{u\in U} E_u\arrow{d}{\bigsqcup\limits_{u\in U}\pi_u}\\
		\overline{M}_G\arrow[hookleftarrow]{rr}&&\bigsqcup\limits_{u\in U}B_u
	\end{tikzcd}
	\]
	By Theorem \ref{T:MG_COHOM}, we need to determine the image of $c(TM_G)$ in the cohomology of $\bigsqcup_{u\in U} E_u$. This is given by
	\[\sum\limits_{u\in U}\pi_u^*c(\xi_u\oplus TB_u)\in \bigoplus\limits_{u\in U}H^*(E_u;R).\]
\end{proof}

\begin{proof}[Proof of Corollary \ref{C:MG_COHOM_6KDIM}.]
	In both cases $M_G$ is simply-connected by Lemma \ref{L:FUND_GROUP}. By \eqref{EQ:BDL_COHOM_SPLIT} we have that $\bigoplus_{u\in U}H^*(E_u;\Z)$ has torsion-free homology. Since, by Lemma \ref{L:BG_BASIS}, the subspace
	\[\bigoplus\limits_{v\in V}\left( \sum_{e=(u,v)\in E} \delta(e)\pi^*_u[B_u]^* \right)\Z\]
	is a direct summand in $\bigoplus_{u\in U}H^q(E_u;\Z)$, the same is true for $H^*(M_G;\Z)$ by Theorem \ref{T:MG_COHOM}. The remaining claims directly follow from Theorems \ref{T:MG_COHOM} and \ref{T:MG_CHAR_CLASSES}.
\end{proof}

	\section{Simply-Connected $6k$-Dimensional Manifolds}
	\label{S:6-MFDS}
	
	In this section we consider manifolds obtained by plumbing in dimension $6k$.
	
	\subsection{The Classification of Jupp}
	
	Recall from the introduction that we consider systems of invariants $(H^{2k}(M),\mu_M,w_2(M)^k,p_k(M))$ of a closed, simply-connected $6k$-dimensional manifold $M$ with torsion-free homology.
	
	Two systems of invariants $(H,\mu,w,p)$ and $(H^\prime,\mu^\prime,w^\prime,p^\prime)$ are \emph{equivalent} if there is an isomorphism $\phi\colon H\to H^\prime$ such that $\phi^*\mu^\prime=\mu$, $\phi(w)=w^\prime$ and $\phi^*p^\prime=p$. It is then clear that two manifolds that are orientation-preserving diffeomorphic have equivalent systems of invariants.
		
	A system of invariants is called \emph{admissible in dimension $6k$} if it is the system of invariants of a closed, simply-connected and oriented $6k$-dimensional manifold with torsion-free homology. With this terminology, for $k=1$ the classification of Jupp is then given as follows.
	\begin{theorem}[{\cite[Theorem 1]{Ju73}}]
		\label{T:JUPP}
		Two oriented closed simply-connected 6-manifolds $M$ and $M^\prime$ are orientation-preserving diffeomorphic if and only if $b_3(M)=b_3(M')$ and their systems of invariants are equivalent.
		A system of invariants $(H,\mu,w,p)$ is admissible in dimension $6$ if and only if
		\begin{equation}
		\label{EQ:JUPP_I}
			\mu(W)\equiv p(W)\mod 48
		\end{equation}
		holds for all $W\in H$ whose reduction to $H\otimes \Z/2$ is $w$.
	\end{theorem}
	For $k>1$, due to the much larger range of (co-)homology groups one needs to consider and since there can exist exotic spheres, there is no such classification possible.
	
	Theorem \ref{T:JUPP} is commonly referred to as the classification of closed, simply-connected 6-manifolds with torsion-free homology. It should be noted that the term \textquotedblleft classification\textquotedblright\ can be misleading in this context: in fact, Jupp's result shows that the classification of closed, simply-connected 6-manifolds with torsion-free homology is \emph{equivalent} to the classification of equivalence classes of systems of invariant that satisfy \eqref{EQ:JUPP_I} and there is no classification known of the latter, except if $\textup{rk}(H)=1$ (this is obvious) or if $\textup{rk}(H)=2$, in which case there exists a partial classification by Schmitt \cite{Sc97}.
	
	We also note that Jupp's classification extends the classification of Wall \cite{Wa66}, which requires the manifolds to be spin, and is a special case of the classification of Zhubr \cite{Zh00}, which does not require the manifolds to have torsion-free homology.
	
	\subsection{Algebraic Plumbing Graphs}
	\label{SS:ALG_PLUMBING}
	
	In this Section we introduce the notion of algebraic plumbing graphs and give the proof of Theorem \ref{T:ALG_PL_GR_CORE1}.
	
	Let $G=(U,V,E,(\alpha,k^+,k^-))$ be a bipartite graph, which has a labeling $(\alpha,k^+,k^-)\colon U\to \Z\times \N_0^2$ for vertices in $U$. We call such a graph an \emph{algebraic plumbing graph}. We will draw vertices $u\in U$ as
	\[
	\begin{tikzpicture}
		\node[circle,draw,label=below:${\scriptstyle k^{-}(u)}$,label=above:${\scriptstyle k^{+}(u)}$] (A) at (0,0) {$\alpha(u)$};
	\end{tikzpicture}
	\]
	If one of $k^+(u)$ and $k^-(u)$ vanishes, then we will omit it.
	Vertices in $V$ will be drawn as dots (as they do not have any labeling). An example for such a graph is the graph given in the introduction.

	For every $u\in U$ we introduce the symbols $u^{-k^-(u)},\dots,u^{k^+(u)}$ and define the free abelian group
	\[A=\bigoplus\limits_{u\in U}\bigoplus\limits_{i=-k^-(u)}^{k^+(u)}\Z u^i. \]
	For $k\in\N$ we define the symmetric trilinear form $\mu^k\colon A^3\to\Z$ by defining it for each $u\in U$ on $\bigoplus_{i=-k^-(u)}^{k^+(u)}\Z u^i$ by
	\begin{align*}
		\mu^k(u^0,u^0,u^0)&=\frac{\lambda_k}{4}\alpha(u),\\
		\mu^k(u^0,u^0,u^l)&=0,\\
		\mu^k(u^0,u^j,u^l)&=\begin{cases}
			\textup{sgn}(j),\quad &j=l,\\
			0,\quad &else,
		\end{cases}\\
		\mu^k(u^i,u^j,u^l)&=0
	\end{align*}
	for $i,j,l\in\{-k^-(u),\dots,k^+(u) \}\setminus\{0\}$, where $\lambda_k\in\N$ is the constant from Lemma \ref{L:S2k_BUNDLES_S4k} (note that it is a multiple of 4 by definition). Then extend $\mu^k$ to $A$ by setting
	\[\mu^k(u^i_m,u^j_n,u^k_r)=0 \]
	whenever any two of $u_m,u_n,u_r$ are not equal.
	
	Next we define a linear form $p^k\colon A\to\Z$ by
	\[p^k(u^j)=\begin{cases}
		\lambda_k\alpha(u)+{2k+1\choose k}(k^+(u)-k^-(u)),\quad&j=0,\\
		0,\quad &\text{else,}
	\end{cases} \]
	and we define $w_G\in A\otimes\Z/2$ by
	\[w_G=\sum\limits_{\substack{u\in U,\\i\neq0}}u^i\mod 2. \]
	
	Finally, we set
	\[H_G=\bigoplus\limits_{\substack{u\in U\\i\neq0}}\Z u^i\oplus\left\{\sum\limits_{u\in U}\lambda_u\cdot u^0 \;\middle|\; \lambda_u\in\Z\;\forall u\in U,\, \sum\limits_{e=(u,v)\in E}\lambda_u=0\;\forall v\in V\right\}\subseteq A\]
	and denote the restrictions of $\mu^k$ and $p^k$ to $H_G$ by $\mu_G^k$ and $p_G^k$, respectively (and note, that, by definition, we have $w_G\in H_G\otimes\Z/2$). Further, we set $\mu_G=\mu_G^1$ and $p_G=p_G^1$.
	\begin{definition}\label{D:INV}
		We call $H_G$, $\mu^k_G$, $w_G$ and $p^k_G$ the \emph{invariants of $G$} and we define the \index{rank of a graph}\emph{rank of $G$} by $\mathrm{rank}(H_G)$. We say that $G$ is \index{spin (graph)}\emph{spin} if $k^+=k^-\equiv0$. Two algebraic plumbing graphs $G$ and $G^\prime$ are \index{equivalence of graphs}\emph{$k$-equivalent}, denoted $G\sim_k G^\prime$, if the systems of invariants $(H_G,\mu_G^k,w_G,p^k_G)$ and $(H_{G^\prime},\mu^k_{G^\prime},w_{G^\prime},p_{G^\prime}^k)$ are equivalent.
	\end{definition}
	\begin{remark}
		In the spin case, if two algebraic plumbing graphs are $k$-equivalent for one $k$, then they are $k$-equivalent for all $k$, since then $\mu_G^k=\frac{\lambda_k}{4}\mu_G$ and $p_G^k=\frac{\lambda_k}{4}p_G$ (recall that $\lambda_1=4)$. However, in the non-spin case it is not clear if $k$-equivalence for one $k$ implies $k$-equivalence for other, or all, $k$.
	\end{remark}
	
	\begin{definition}\label{D:GEOM_PL_GRAPH}
		We define a geometric plumbing graph $\overline{G}^k=(U,V,E,\pi,\delta)$ with the same set of vertices and edges as $G$ as follows. For $u\in U$ set 
		\[B_u=\#_{k^+(u)}\C P^{2k}\#_{k^-(u)}(-\C P^{2k})\]
		(note that the empty connected sum is defined as $S^{4k}$) and define $\pi_u$ as the disc bundle of the bundle $\pi_{\alpha(u),B_u}$ of Definition \ref{D:pi_alpha,B}. For $v\in V$ set $B_v=S^{2k+1}$ and define $\pi_v$ as the trivial $D^{4k}$-bundle over $S^{2k+1}$, i.e.\ $\overline{E}_v=S^{2k+1}\times D^{4k}$ and $\pi_v$ is given by the projection onto the first factor. Finally, we define $\delta(e)=1$ for all $e\in E$. We set $\overline{G}=\overline{G}^1$.
	\end{definition}
	\begin{remark}\label{R:preimage_choice}
		For $k>2$ the definition of each bundle $\pi_{\alpha(u),B_u}$, and therefore the definition of $\overline{G}^k$, depends on a choice of preimage in Definition \ref{D:pi_alpha,B}. As noted after Definition \ref{D:pi_alpha,B}, the results of this article hold for any choice of such preimage.
	\end{remark}	
	
	\begin{lemma}
		\label{L:ALG_PL_GR_INV}
		Let $G=(U,V,E,(\alpha,k^+,k^-))$ be an algebraic plumbing graph for which all connected components are simply-connected and set $M=M_{\overline{G}^k}$. Then
		\begin{enumerate}
			\item $M$ is a closed, simply-connected $6k$-dimensional manifold with torsion-free homology and the systems of invariants $(H^{2k}(M),\mu_M,w_2^k(M),p_k(M))$ and $(H_G,\mu_G^k,w_G,p_G^k)$ are equivalent. In particular, the system $(H_G,\mu_G^k,w_G,p_G^k)$ is admissible in dimension $6k$.
			\item There exists a $k$-equivalent subgraph $G^\prime$ of $G$ so that for $M^\prime=M_{\overline{G^\prime}^k}$ the same as in (1) holds and additionally the odd Betti numbers of $M^\prime$ vanish and $M=M^\prime\#_r(S^{2k+1}\times S^{4k-1})$ for some $r\in\N_0$.
			\item $b_{2k}(M^\prime)=b_{2k}(M)=|U^\prime|-|V^\prime|+\sum_{u\in U^\prime}k^+(u)+k^-(u)$.
			\item $M$ is spin if and only if $G$ is spin, and the same holds for $M'$ and $G'$ (and note that $M$ is spin if and only if $M'$ is spin by (2)).
		\end{enumerate}
	\end{lemma}
	\begin{proof}
		We use Proposition \ref{P:GRAPH_RED1} to split $\overline{G}^k$ into connected components that either satisfy the hypotheses of Theorem \ref{T:MG_COHOM} or consist of a single vertex in $V$ as follows: For any $u\in U$ that is connected to a leaf, we remove all edges connected to $u$ except one that connects $u$ to a leaf, this is precisely the modification in Proposition \ref{P:GRAPH_RED1}. The corresponding modification of $G$ does not change its invariants, since for such $u$ we always have that $u^0$ has coefficient zero in every element of $H_G$. We repeat this process until all connected components of $\overline{G}^k$ satisfy the hypotheses of Theorem \ref{T:MG_COHOM} and we denote the graph we obtain from $G$ in this way after additionally removing all isolated vertices in $V$ by $G^\prime$. Then $G$ and $G^\prime$ have the same invariants, since isolated vertices in $V$ do not make any contribution, and the manifolds $M$ and $M^\prime$ then only differ by connected sums of copies of $S^{2k+1}\times S^{4k-1}$.
		
		By Corollary \ref{C:MG_COHOM_6KDIM}, the cohomology group $H^{2k}(M)$ is given by
		\[H^{2k}(M)=\bigoplus\limits_{u\in U}\pi_u^*H^{2k}(B_u)\oplus\left\{\sum\limits_{u\in U}\lambda_u\cdot a_u \;\middle|\; \lambda_u\in \Z \; \forall u\in U,\, \sum\limits_{e=(u,v)\in E^\prime}\delta(e)\lambda_u=0 \;\forall v\in V^\prime \right\} \]
		and we define the isomorphism $\phi\colon H^{2k}(M)\to H_{G'}$ by mapping a generator of the $i$-th summand in the first component of the right hand side of
		\[H^{2k}(B_u)=\bigoplus\limits_{k^+(u)}H^{2k}(\C P^{2k})\bigoplus\limits_{k^-(u)}H^{2k}(-\C P^{2k})  \]
		to $u^i$ and a generator of the $i$-th summand in the second component of the right hand side to $u^{-i}$. Further, we define
		\[\phi\left(\sum\limits_{u\in U}\lambda_u\cdot a_u\right)=\sum\limits_{u\in U}\lambda_u u^0. \]
		It now follows from Theorem \ref{T:MG_COHOM} and Corollary \ref{C:2k_SPHERE_BDL_COHOM}, that this isomorphism preserves the remaining invariants. Moreover, it follows from \eqref{EQ:BDL_COHOM_SPLIT} that $\bigoplus_{u\in U}H^*(E_u)$ is trivial in odd degrees and by Theorem \ref{T:MG_COHOM} the same holds for $H^*(M')$.
		
		For (3) we need to determine the rank of
		\[\left\{\sum\limits_{u\in U}\lambda_u\cdot u^0 \;\middle|\; \lambda_u\in\Z\;\forall u\in U,  \sum\limits_{e=(u,v)\in E^\prime}\lambda_u=0\;\forall v\in V^\prime\right\}\subseteq H_{G^\prime} . \]
		The condition for the coefficients $\lambda_u$ is equivalent to $(\lambda_u)_{u\in U}$ being an element of the kernel of $B(G^\prime)^\top$ (see Appendix \ref{A:GRAPHS}). By Lemma \ref{L:BG_RANK}, the matrix $B(G^\prime)$ has rank $|V^\prime|$ (as $|U^\prime|\geq |V^\prime|$), hence its kernel has rank $|U^\prime|-|V^\prime|$.
		
		Finally, for (4), note that by (1) $w_G$ (and $w_{G^\prime}$) vanishes if and only if $w_2(M)^k$ (and $w_2(M^\prime)^k$) vanishes, which is the case if and only if $w_2(M)$ (and $w_2(M^\prime)$) vanishes by the cohomology ring structure of each $E_u$.
	\end{proof}
	
	We can now prove Theorem \ref{T:ALG_PL_GR_CORE1}.
	
	\begin{proof}[Proof of Theorem \ref{T:ALG_PL_GR_CORE1}]
		First consider an arbitrary algebraic plumbing graph $G$. By construction, the group $H_G$ is a subgroup of $A$ and the invariants $\mu_G$ and $p_G$ are the restrictions of the invariants $\mu^1$ and $p^1$. Let $G^0$ be the graph obtained from $G$ be removing all edges. Then the invariants of $G^0$ are precisely $A$, $\mu^1$, $w_G$ and $p^1$ and the system of invariants $(A,\mu^1,w_G,p^1)$ is realized by $M_{\overline{G^0}}$ by (1) of Lemma \ref{L:ALG_PL_GR_INV}. It follows that \eqref{EQ:JUPP_I} holds for all $W\in A$, and hence also for all $W\in H_G$, that restrict to $w_G$. Thus, the system $(H_G,\mu_G,w_G,p_G)$ is admissible in dimension 6.		
		
		Now assume that every connected component of $G$ is simply-connected and let $M=M_{\overline{G}^k}$. We apply Lemma \ref{L:ALG_PL_GR_INV} to obtain an equivalent subgraph $G^\prime$ of $G$ and a manifold $M^\prime=M_{\overline{G^\prime}^k}$ with vanishing odd cohomology. Since each connected component of $G^\prime$ is simply-connected, we can apply Theorem \ref{T:PLUMBING} to obtain a core metric on each summand of $M^\prime$, hence $M^\prime$ also admits a core metric, but note that if $k=1$, then the restrictions on the dimensions in this theorem require that every connected component contains a vertex in $V$. This can always be achieved by introducing a new vertex according to Proposition \ref{P:GRAPH_RED2}. Since $M=M^\prime\#_r(S^{2k+1}\times S^{4k-1})$ for some $r\in\N_0$, and $S^{2k+1}\times S^{4k-1}$ admits a core metric (see the list in Subsection \ref{SS:CORE_METRICS}), $M$ admits a core metric.
		
		Finally, if $k=1$ and $N$ is a simply-connected 6-manifold with torsion-free homology, whose invariants are equivalent to $(H_G,\mu_G,w_G,p_G)$, then, by Theorem \ref{T:JUPP}, $N$ is diffeomorphic to $M^\prime\#_r(S^3\times S^3)$ for some $r\in\N_0$, so $N$ admits a core metric.
	\end{proof}

	\subsection{Reduced Graphs}
	\label{SS:RED_GRAPHS}
	
	A fixed system of invariants $(H,\mu,w,p)$ can potentially be realized by many different algebraic plumbing graphs. To analyze this, we consider modifications of graphs that do not change the invariants.
	\begin{lemma}
		\label{L:GRAPH_RED}
		We can modify algebraic plumbing graphs in the following ways without changing their $k$-equivalence classes.
		\begin{enumerate}
			\item[$(1)$] \begin{minipage}[c]{\linewidth}
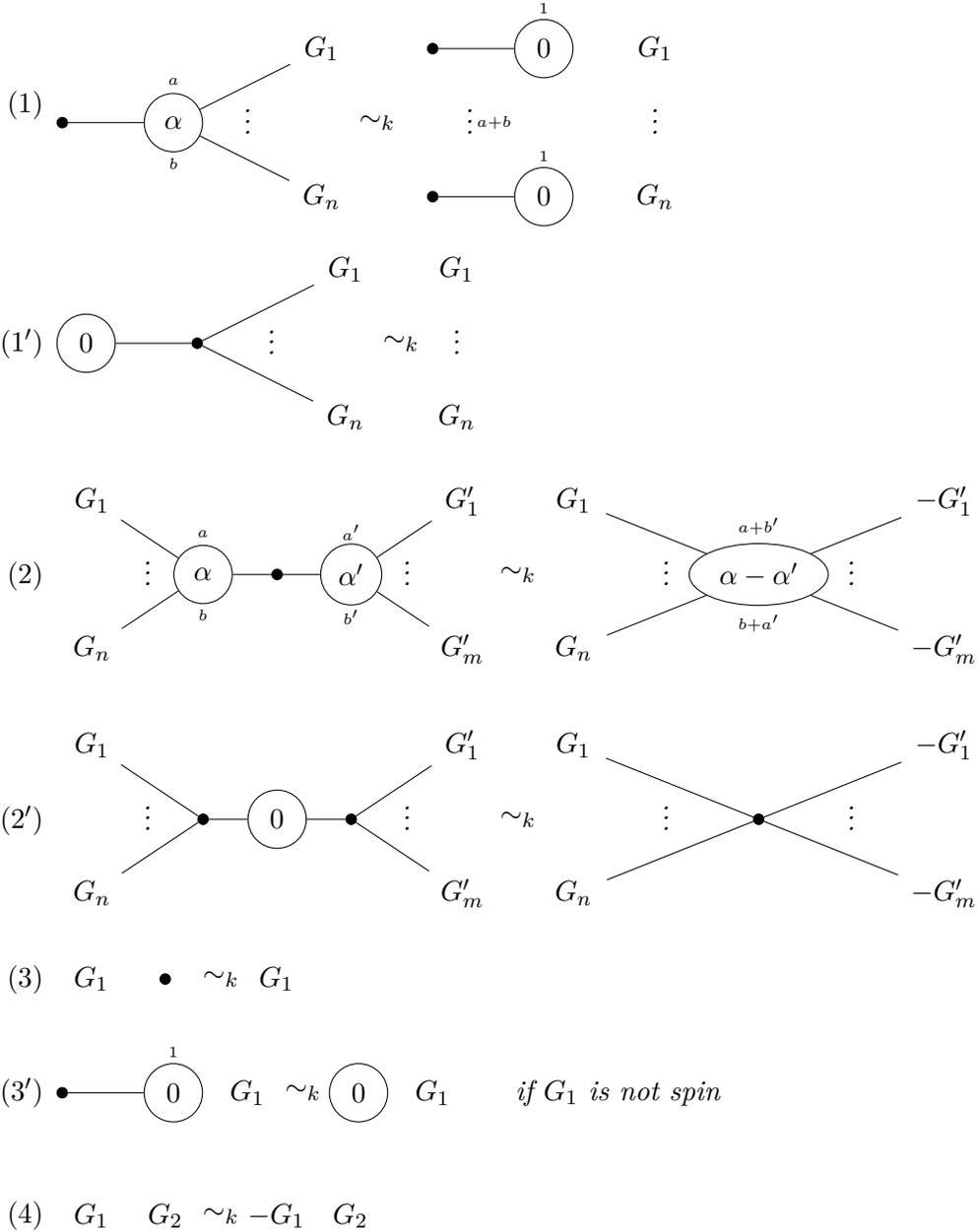
\begin{tikzpicture}
					\begin{scope}[every node/.style={circle,draw,minimum height=2em}]
						\node[Bullet] (V) at (-0.5,0) {};
						\node[label={[label distance=-0.25cm]below:${\scriptscriptstyle b}$},label={[label distance=-0.25cm]above:${\scriptscriptstyle a}$}] (U) at (1,0) {$\alpha$};
						\node[draw=none] (G1) at (3,1) {$G_1$};
						\node[draw=none] (Gn) at (3,-1) {$G_n$};
						\node[draw=none] (dots) at (2,0) {$\rvdots$};
					\end{scope}
					\path[-](V) edge (U);
					\path[-](G1) edge (U);
					\path[-](Gn) edge (U);
					\begin{scope}[every node/.style={circle,draw,minimum height=2em}]
						\node[Bullet] (V1) at (4.5,1) {};
						\node[label={[label 	distance=-0.25cm]above:${\scriptscriptstyle 1}$}] (U1) at (6,1) {$0$};
						\node[Bullet] (V2) at (4.5,-1) {};
						\node[label={[label distance=-0.25cm]above:${\scriptscriptstyle 1}$}] (U2) at (6,-1) {$0$};
						\node[draw=none] (G11) at (7.5,1) {$G_1$};
						\node[draw=none] (Gn1) at (7.5,-1) {$G_n$};
						\node[draw=none] (dots1) at (7.5,0) {$\rvdots$};
						\node[draw=none,label={[label distance=-0.5cm]right:${\scriptscriptstyle a+b}$}] (dots2) at (5,0) {$\rvdots$};
						\node[draw=none] (sim) at (3.75,0) {$\sim_k$};
					\end{scope}
					\path[-](V1) edge (U1);
					\path[-](V2) edge (U2);
				\end{tikzpicture}
			\end{minipage}
			\item[$(1^\prime)$] \begin{minipage}[c]{\linewidth}\begin{tikzpicture}
					\begin{scope}[every node/.style={circle,draw,minimum height=2em}]
						\node (U) at (-0.5,0) {$0$};
						\node[Bullet] (V) at (1,0) {};
						\node[draw=none] (G1) at (3,1) {$G_1$};
						\node[draw=none] (Gn) at (3,-1) {$G_n$};
						\node[draw=none] (dots) at (2,0) {$\rvdots$};
					\end{scope}
					\path[-](U) edge (V);
					\path[-](G1) edge (V);
					\path[-](Gn) edge (V);
					\begin{scope}[every node/.style={circle,draw,minimum height=2em}]
						\node[draw=none] (G11) at (4.5,1) {$G_1$};
						\node[draw=none] (Gn1) at (4.5,-1) {$G_n$};
						\node[draw=none] (dots1) at (4.5,0) {$\rvdots$};
						\node[draw=none] (sim) at (3.75,0) {$\sim_k$};
					\end{scope}
				\end{tikzpicture}
			\end{minipage}
			\item[$(2)$] \begin{minipage}[c]{\linewidth}\begin{tikzpicture}
					\begin{scope}[every node/.style={circle,draw,minimum height=2em}]
						\node[Bullet] (V) at (0,0) {};
						\node[label={[label distance=-0.25cm]below:${\scriptscriptstyle b}$},label={[label distance=-0.25cm]above:${\scriptscriptstyle a}$}] (U1) at (-1,0) {$\alpha$};
						\node[label={[label distance=-0.25cm]below:${\scriptscriptstyle b^\prime}$},label={[label distance=-0.25cm]above:${\scriptscriptstyle a^\prime}$}] (U2) at (1,0) {$\alpha^\prime$};
						\node[draw=none] (G1) at (-2.5,1) {$G_1$};
						\node[draw=none] (Gn) at (-2.5,-1) {$G_n$};
						\node[draw=none] (G1p) at (2.5,1) {$G_1^\prime$};
						\node[draw=none] (Gmp) at (2.5,-1) {$G_m^\prime$};
						\node[draw=none] (dots1) at (-1.75,0) {$\rvdots$};
						\node[draw=none] (dots2) at (1.75,0) {$\rvdots$};
					\end{scope}
					\path[-](V) edge (U1);
					\path[-](V) edge (U2);
					\path[-](G1) edge (U1);
					\path[-](Gn) edge (U1);
					\path[-](G1p) edge (U2);
					\path[-](Gmp) edge (U2);
					\begin{scope}[every node/.style={circle,draw,minimum height=2em}]
						\node[ellipse, label={[label distance=-0.25cm]below:${\scriptscriptstyle b+a^\prime}$},label={[label distance=-0.25cm]above:${\scriptscriptstyle a+b^\prime}$}] (U11) at (6.5,0) {$\alpha-\alpha^\prime$};
						\node[draw=none] (G11) at (4,1) {$G_1$};
						\node[draw=none] (Gn1) at (4,-1) {$G_n$};
						\node[draw=none] (G1p1) at (9,1) {$-G_1^\prime$};
						\node[draw=none] (Gmp1) at (9,-1) {$-G_m^\prime$};
						\node[draw=none] (dots3) at (5.25,0) {$\rvdots$};
						\node[draw=none] (dots4) at (7.75,0) {$\rvdots$};
						\node[draw=none] (sim) at (3.25,0) {$\sim_k$};
					\end{scope}
					\path[-](G11) edge (U11);
					\path[-](Gn1) edge (U11);
					\path[-](G1p1) edge (U11);
					\path[-](Gmp1) edge (U11);
				\end{tikzpicture}
			\end{minipage}
			\item[$(2^\prime)$]\begin{minipage}[c]{\linewidth}\begin{tikzpicture}
					\begin{scope}[every node/.style={circle,draw,minimum height=2em}]
						\node (U) at (0,0) {$0$};
						\node[Bullet] (V1) at (-1,0) {};
						\node[Bullet] (V2) at (1,0) {};
						\node[draw=none] (G1) at (-2.5,1) {$G_1$};
						\node[draw=none] (Gn) at (-2.5,-1) {$G_n$};
						\node[draw=none] (G1p) at (2.5,1) {$G_1^\prime$};
						\node[draw=none] (Gmp) at (2.5,-1) {$G_m^\prime$};
						\node[draw=none] (dots1) at (-1.75,0) {$\rvdots$};
						\node[draw=none] (dots2) at (1.75,0) {$\rvdots$};
					\end{scope}
					\path[-](U) edge (V1);
					\path[-](U) edge (V2);
					\path[-](G1) edge (V1);
					\path[-](Gn) edge (V1);
					\path[-](G1p) edge (V2);
					\path[-](Gmp) edge (V2);
					\begin{scope}[every node/.style={circle,draw,minimum height=2em}]
						\node[Bullet] (V11) at (6.5,0) {};
						\node[draw=none] (G11) at (4,1) {$G_1$};
						\node[draw=none] (Gn1) at (4,-1) {$G_n$};
						\node[draw=none] (G1p1) at (9,1) {$-G_1^\prime$};
						\node[draw=none] (Gmp1) at (9,-1) {$-G_m^\prime$};
						\node[draw=none] (dots3) at (5.25,0) {$\rvdots$};
						\node[draw=none] (dots4) at (7.75,0) {$\rvdots$};
						\node[draw=none] (sim) at (3.25,0) {$\sim_k$};
					\end{scope}
					\path[-](G11) edge (V11);
					\path[-](Gn1) edge (V11);
					\path[-](G1p1) edge (V11);
					\path[-](Gmp1) edge (V11);
				\end{tikzpicture}
			\end{minipage}
			\item[$(3)$] \begin{minipage}[c]{\linewidth}\begin{tikzpicture}
					\begin{scope}[every node/.style={circle,draw,minimum height=2em}]
						\node[draw=none] (G1) at (0,0) {$G_1$};
						\node[Bullet] (V) at (1,0) {};
						\node[draw=none] (G11) at (2.5,0) {$G_1$};
						\node[draw=none] (sim) at (1.75,0) {$\sim_k$};
					\end{scope}
				\end{tikzpicture}
			\end{minipage}\vspace{-0.5cm}
			\item[$(3^\prime)$] \begin{minipage}[c]{\linewidth}\begin{tikzpicture}
					\begin{scope}[every node/.style={circle,draw,minimum height=2em}]
						\node[draw=none] (G1) at (1,0) {$G_1$};
						\node[draw=none] (G11) at (3.5,0) {$G_1$};
						\node[draw=none] (sim) at (1.75,0) {$\sim_k$};
						\node[Bullet] (V1) at (-1.5,0) {};
						\node[label={[label 	distance=-0.25cm]above:${\scriptscriptstyle 1}$}] (U1) at (0,0) {$0$};
						\node (U4) at (2.5,0) {$0$};
						\node[draw=none](txt) at (6,0) {if $G_1$ is not spin};
					\end{scope}
					\path[-](U1) edge (V1);
				\end{tikzpicture}
			\end{minipage}\vspace{-0.5cm}
			\item[$(4)$] \begin{minipage}[c]{\linewidth}\begin{tikzpicture}
					\begin{scope}[every node/.style={circle,draw,minimum height=2em}]
						\node[draw=none] (G1) at (0,0) {$G_1$};
						\node[draw=none] (G2) at (1,0) {$G_2$};
						\node[draw=none] (G11) at (2.5,0) {$-G_1$};
						\node[draw=none] (G21) at (3.5,0) {$G_2$};
						\node[draw=none] (sim) at (1.75,0) {$\sim_k$};
					\end{scope}
				\end{tikzpicture}
			\end{minipage}
		\end{enumerate}
		Here $G_i$, $G_i^\prime$ are (pairwise disjoint, and possibly empty) subgraphs, and $-G$ denotes $G$ with $\alpha$ replaced by $-\alpha$ and $(k^+,k^-)$ replaced by $(k^-,k^+)$.
	\end{lemma}
	\begin{proof}
		First note that $\overline{-G}^k$ is obtained from $\overline{G}^k$ by reversing the orientations of all bases and fibers (but not of the total spaces). Then the equivalences $(3)$ and $(4)$ are clear and the remaining equivalences follow from Propositions \ref{P:GRAPH_RED_EDGE}, \ref{P:GRAPH_RED1}, \ref{P:GRAPH_RED2} and Lemma \ref{L:ALG_PL_GR_INV}, except $(1)$ and $(3^\prime)$. For $(1)$ we additionally need to show that
		
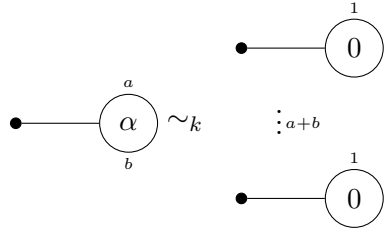
\begin{center}
			\begin{tikzpicture}
				\begin{scope}[every node/.style={circle,draw,minimum height=2em}]
					\node[Bullet] (V) at (-0.5,0) {};
					\node[label={[label distance=-0.25cm]below:${\scriptscriptstyle b}$},label={[label distance=-0.25cm]above:${\scriptscriptstyle a}$}] (U) at (1,0) {$\alpha$};
				\end{scope}
				\path[-](V) edge (U);
				\begin{scope}[every node/.style={circle,draw,minimum height=2em}]
					\node[Bullet] (V1) at (2.5,1) {};
					\node[label={[label 	distance=-0.25cm]above:${\scriptscriptstyle 1}$}] (U1) at (4,1) {$0$};
					\node[Bullet] (V2) at (2.5,-1) {};
					\node[label={[label distance=-0.25cm]above:${\scriptscriptstyle 1}$}] (U2) at (4,-1) {$0$};
					\node[draw=none,label={[label distance=-0.5cm]right:${\scriptscriptstyle a+b}$}] (dots2) at (3,0) {$\rvdots$};
					\node[draw=none] (sim) at (1.75,0) {$\sim_k$};
				\end{scope}
				\path[-](V1) edge (U1);
				\path[-](V2) edge (U2);
			\end{tikzpicture}
		\end{center}
		holds. For that, denote the graph on the left-hand side by $G$ and the graph on the right-hand side by $G^\prime$. Denote the single element of $U$ by $u$ and the elements of $U^\prime$ by $u_1,\dots,u_{a+b}$. Then, by definition, we have
		\[H_G=\bigoplus\limits_{\substack{i=-b\\i\neq 0}}^a \Z u^i\quad \text{ and }\quad H_{G^\prime}=\bigoplus\limits_{i=1}^{a+b}\Z u_i^1, \]
		and an isomorphism $H_G\to H_{G^\prime}$ is given by mapping $u^i$ to $u_i^1$ for $i>0$ and $u^i$ to $u_{a-i}^1$ for $i<0$. It is now easily verified that $\mu_G^k$, $p_G^k$, $\mu_{G^\prime}^k$ and $p_{G^\prime}^k$ all vanish, and that
		\[w_G=\sum\limits_{\substack{u\in U,\\i\neq0}}u^i\mod 2,\quad w_{G^\prime}=\sum\limits_{i=1}^{a+b}u_i^1\mod 2,\]
		hence this isomorphism preserves all invariants.
		
		For $(3^\prime)$ denote the graph on the left-hand side by $G$ and the one on the right-hand side by $G^\prime$. Let $x_0\in H_{G_1}$ be a primitive element that restricts to $w_{G_1}$ (which exists since $w_{G_1}$ is non-trivial) and extend it to a basis $(x_0,\dots,x_n)$ of $H_{G_1}$. Let $u_1$ be the additional vertex in $U$ and $u_2$ the additional vertex in $U^\prime$. Then
		\[w_G=u_1^1+x_0\mod 2,\quad w_{G^\prime}=x_0\mod 2. \]
		Hence, by mapping $u_1^1$ to $u_2^0$, $x_0$ to $x_0-u_2^0$ and $x_i$ to $x_i$ for $i>0$, we obtain an isomorphism $H_G\to H_{G^\prime}$ that maps $w_G$ to $w_{G^\prime}$. Since the linear and trilinear forms are only non-trivial on elements of $H_{G_1}$, they are also preserved under this isomorphism.
	\end{proof}
	
	These modifications will be used to bring a given graph into a reduced form.
	\begin{definition}
		Let $G=(U,V,E,(\alpha,k^+,k^-))$ be an algebraic plumbing graph. We call $G$ \index{reduced graph}\emph{reduced}, if it satisfies the following conditions:
		\begin{itemize}
			\item Every connected component of $G$ is simply-connected.
			\item The graph \begin{tikzpicture}[scale=0.6,, transform shape]
				\begin{scope}[every node/.style={circle,draw}]
					\node[Bullet] (V) at (0,0) {};
					\node[label={[label distance=-0.2cm]above:${\scriptscriptstyle 1}$}] (U) at (1.5,0) {$0$};
				\end{scope}
				\path[-](V) edge (U);
			\end{tikzpicture} only appears as a connected component in $G$ if it is the only non-spin connected component. Every $v\in V$ not contained in this connected component has degree at least $3$, and
			\item Every $u\in U$ with $\alpha(u)=k^+(u)=k^-(u)=0$ has degree $0$ or at least $3$.
		\end{itemize}
		Given a reduced graph $G$ we define its \emph{reduced class} as the set of isomorphism classes of reduced graphs that are obtained from $G$ by multiplying each connected component by $(\pm1)$. By (4) of Lemma \ref{L:GRAPH_RED}, two reduced graphs with the same reduced class are $k$-equivalent, so the notion of $k$-equivalence descends to an equivalence relation between reduced classes.
	\end{definition}
	
	The following result is now a direct consequence of Lemmas \ref{L:ALG_PL_GR_INV} and \ref{L:GRAPH_RED}.
	\begin{corollary}
		\label{C:RED_GRAPHS}
		Let $G=(U,V,E,(\alpha,k^+,k^-))$ be an algebraic plumbing graph. If every connected component of $G$ is simply-connected, then $G$ is $k$-equivalent to a reduced graph $G^\prime$. Further, we have $\textup{rank}(H_G)=|U^\prime|-|V^\prime|+\sum_{u\in U^\prime}k^+(u)+k^-(u)$.
	\end{corollary}
	It is now a natural question, how \textquotedblleft good\textquotedblright\ this notion of reduced form is, i.e.\ if any two $k$-equivalent reduced graphs are contained in the same reduced class.
	\begin{question}
		\label{Q:RED_EQUIV=ISOM?}
		Let $G_1$, $G_2$ be reduced graphs that are $k$-equivalent for some $k$. Are their reduced classes the same?
	\end{question}
	This question is open, but we will answer it affirmatively in some special cases. For that we first list the reduced graphs of low rank. Clearly the only reduced graph of rank $0$ is the empty graph, which defines $S^6$.
	
	\begin{proposition}
		\label{P:RED_FORMS1}
		Let $G$ be a reduced graph of rank 1. Then the following assertions hold:
		\begin{itemize}
			\item If $G$ is spin, then it is of the form \begin{tikzpicture}[scale=0.6,, transform shape]
				\begin{scope}[every node/.style={circle,draw}]
					\node (U) at (0,0) {$\alpha$};
				\end{scope}
			\end{tikzpicture}. The manifold $M_{\overline{G}^k}$ is the total space of a linear $S^{2k}$-bundle over $S^{4k}$. Two graphs \begin{tikzpicture}[scale=0.5,, transform shape]
				\begin{scope}[every node/.style={circle,draw}]
					\node (U) at (0,0) {$\alpha_1$};
				\end{scope}
			\end{tikzpicture} and \begin{tikzpicture}[scale=0.5,, transform shape]
				\begin{scope}[every node/.style={circle,draw}]
					\node (U) at (0,0) {$\alpha_2$};
				\end{scope}
			\end{tikzpicture} are $k$-equivalent if and only if $\alpha_1=\pm\alpha_2$.
			\item If $G$ is not spin, then it is given by \begin{tikzpicture}[scale=0.6,, transform shape]
				\begin{scope}[every node/.style={circle,draw}]
					\node[Bullet] (V) at (0,0) {};
					\node[label={[label distance=-0.1cm]above:${\scriptscriptstyle 1}$}] (U) at (1.5,0) {$0$};
				\end{scope}
				\path[-](V) edge (U);
			\end{tikzpicture}. The graph $G$ has trivial trilinear form $\mu_G^k$ and trivial linear form $p_G^k$. If $k=1$, then the manifold $M_{\overline{G}}$ is the unique non-trivial linear $S^4$-bundle over $S^2$.
		\end{itemize}
	\end{proposition}
	\begin{proposition}
		\label{P:RED_FORMS2}
		Let $G$ be a reduced graph of rank 2. Then the following assertions hold:
		\begin{itemize}
			\item If $G$ is spin, then it is of the form
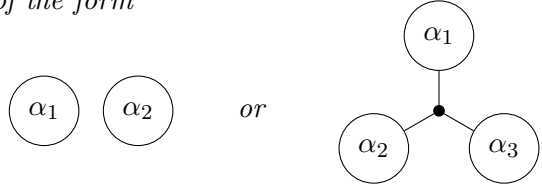
\vspace{-0.3cm}
			\begin{center}
				\begin{tikzpicture}
					\begin{scope}[every node/.style={circle,draw,minimum height=2.35em}]
						\node(U1) at (-1.25,0) {$\alpha_1$};
						\node(U2) at (0,0) {$\alpha_2$};
						\node[draw=none](or) at (1.5,0){or};
						\node(U11) at (4,1){$\alpha_1$};
						\node(U21) at (3.133975,-0.5){$\alpha_2$};
						\node(U31) at (4.866025,-0.5){$\alpha_3$};
						\node[Bullet](V) at (4,0){};
					\end{scope}
					\path[-](V) edge (U11);
					\path[-](V) edge (U21);
					\path[-](V) edge (U31);
				\end{tikzpicture}
			\end{center}
			with $\alpha_i\neq0$ in the second case. For every such graph $G$ there is at most one reduced class that is $k$-equivalent but not equal to that of $G$.
			\item If $G$ is not spin, its reduced class is equal to that of a graph of the form
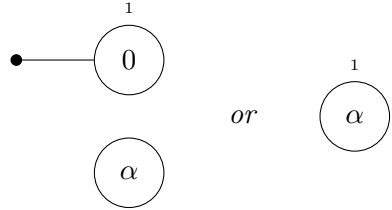
\vspace{-0.3cm}
			\begin{center}
				\begin{tikzpicture}
					\begin{scope}[every node/.style={circle,draw,minimum height=2.35em}]
						\node[Bullet] (V1) at (-0.5,0.75) {};
						\node[label={[label distance=-0.25cm]above:${\scriptscriptstyle 1}$}] (U1) at (1,0.75) {$0$};
						\node (U2) at (1,-0.75) {$\alpha$};
						\node[draw=none](or) at (2.5,0){or};
					\end{scope}
					\path[-](V1) edge (U1);
					\begin{scope}[every node/.style={circle,draw,minimum height=2.35em}]
						\node[label={[label 	distance=-0.25cm]above:${\scriptscriptstyle 1}$}] (U) at (4,0) {$\alpha$};
					\end{scope}
				\end{tikzpicture}
			\end{center}
			For every such graph $G$ every reduced class that is $k$-equivalent to that of $G$ is equal to that of $G$.
		\end{itemize}
	\end{proposition}
	We refer to \cite[Proposition 6.4.3]{Re22} for a list of reduced graphs of rank 3.
	
	The possibilities for the reduced graphs are a simple cobinatorial consequence of the following lemma.
	\begin{lemma}
		\label{L:RED_GRAPH_INEQ}
		Let $G^\prime$ be a non-empty connected component of a reduced graph that is not of the form \begin{tikzpicture}[scale=0.6,, transform shape]
			\begin{scope}[every node/.style={circle,draw}]
				\node[Bullet] (V) at (0,0) {};
				\node[label={[label distance=-0.15cm]above:${\scriptscriptstyle 1}$}] (U)	at (1.5,0) {$0$};
			\end{scope}
			\path[-](V) edge (U);
		\end{tikzpicture}. Then
		\[2\cdot|V^\prime|+1\leq |U^\prime|.\]
	\end{lemma}
	\begin{proof}
		If $V^\prime$ is empty, then the inequality holds trivially. If $V^\prime$ is non-empty, then, since $G^\prime$ is simply-connected, we can choose a root $v_0\in V^\prime$ and consider it as a tree. Then, the inequality follows from the fact that, by definition, $v_0$ has at least $3$ descending vertices in $U^\prime$, while every other $v\in V^\prime$ has at least 2.
	\end{proof}

	It follows from Lemma \ref{L:RED_GRAPH_INEQ} and Corollary \ref{C:RED_GRAPHS} that for a connected component $G^\prime$ of a reduced graph, that is not of the form \begin{tikzpicture}[scale=0.6,, transform shape]
		\begin{scope}[every node/.style={circle,draw}]
			\node[Bullet] (V) at (0,0) {};
			\node[label={[label distance=-0.15cm]above:${\scriptscriptstyle 1}$}] (U)	at (1.5,0) {$0$};
		\end{scope}
		\path[-](V) edge (U);
	\end{tikzpicture}, we have
	\[\textrm{rank}(H_{G^\prime})=|U^\prime|-|V^\prime|+\sum\limits_{u\in U^\prime}k^+(u)+k^-(u)\geq 1+|V^\prime|+\sum\limits_{u\in U^\prime}k^+(u)+k^-(u). \]
	Thus, $\textrm{rank}(H_{G^\prime})=1$ implies that $V^\prime$ is empty and $k^+=k^-\equiv0$. In a similar way, by going through all possibilities, we obtain the reduced forms in Proposition \ref{P:RED_FORMS2}. It remains to prove the statements about $k$-equivalence in Propositions \ref{P:RED_FORMS1} and \ref{P:RED_FORMS2}.
		\begin{proof}[Proof of Proposition \ref{P:RED_FORMS1}]
		If $G$ is not of the form \begin{tikzpicture}[scale=0.6,, transform shape]
			\begin{scope}[every node/.style={circle,draw}]
				\node[Bullet] (V) at (0,0) {};
				\node[label={[label distance=-0.15cm]above:${\scriptscriptstyle 1}$}] (U)	at (1.5,0) {$0$};
			\end{scope}
			\path[-](V) edge (U);
		\end{tikzpicture}, then $G$ is given by \begin{tikzpicture}[scale=0.6,, transform shape]
			\begin{scope}[every node/.style={circle,draw}]
				\node (U) at (0,0) {$\alpha$};
			\end{scope}
		\end{tikzpicture}, which, by construction, yields the linear $S^{2k}$-bundle over $S^{4k}$ corresponding to $\alpha$.
		
		For such a graph $G$ let $u\in U$ be the unique element. Then $H_G=\Z u^0$ and
		\[\mu(u^0,u^0,u^0)=\frac{\lambda_k}{4}\alpha=-\mu(-u^0,-u^0,-u^0).\]
		Since $u^0$ and $-u^0$ are the only generators of $H_G$, this shows that for different absolute values of $\alpha$ we obtain non-equivalent trilinear forms.
		
		In the non-spin case the manifold $M_{\overline{G}}$ is diffeomorphic to the unique non-trivial $S^4$-bundle over $S^2$ by \cite[Lemma 4.2]{Re23}.
	\end{proof}
	
	\begin{proof}[Proof of Proposition \ref{P:RED_FORMS2}]
		In the non-spin case the linear form $p_G^k$ of the first graph is given by $p_G^k(u_1^1)=0$, $p_G^k(u_2^0)=\lambda_k\alpha$ (where $u_1$ denotes the upper vertex and $u_2$ the lower one), while the linear form on the second graph is given by $p_G^k(u^0)=\lambda_k\alpha+{2k+1\choose k}$, $p_G^k(u^1)=0$, in particular this is always non-zero. Since $\lambda_k>{2k+1\choose k}$, see Remark \ref{R:LAMBDA_K}, this shows that for different (absolute) values of $\alpha$ we obtain non-equivalent graphs.
		
		In the spin case we first consider the case $k=1$. The first graph is $k$-equivalent to a graph of the second form with $\alpha_3=0$, therefore it is enough to consider graphs of the second form, allowing $\alpha_i=0$. Let $G$ be such a graph. Then $e_1=u^0_1-u^0_3$ and $e_2=u^0_2-u^0_3$ form a basis of $H_G$ (where $u_i$ denotes the vertex labeled by $\alpha_i$) and we have
		\begin{align*}
			\mu_G(e_1,e_1,e_1)&=\alpha_1-\alpha_3,\\
			\mu_G(e_1,e_1,e_2)&=-\alpha_3,\\
			\mu_G(e_1,e_2,e_2)&=-\alpha_3,\\
			\mu_G(e_2,e_2,e_2)&=\alpha_2-\alpha_3,\\
			p_G(e_1)&=4(\alpha_1-\alpha_3),\\
			p_G(e_2)&=4(\alpha_2-\alpha_3).
		\end{align*}
		We define homogeneous polynomials $f$ and $p$ by setting $f(x_1,x_2)=\mu_G(x,x,x)$, and $p(x_1,x_2)=\frac{1}{4}p_G(x)$, $x=x_1 e_1+x_2 e_2$ and $x_1,x_2\in\C$. We obtain
		\begin{align*}
			f(x_1,x_2)&=(\alpha_1-\alpha_3)x_1^3-3\alpha_3 x_1^2 x_2-3\alpha_3 x_1 x_2^2+(\alpha_2-\alpha_3)x_2^3,\\
			p(x_1,x_2)&=(\alpha_1-\alpha_3)x_1+(\alpha_2-\alpha_3)x_2.
		\end{align*}
	
		Given two systems of invariants $(H_1,\mu_1,0,p_1)$ and $(H_2,\mu_2,0,p_2)$ with $\mathrm{rank}(H_1)=\mathrm{rank}(H_2)=n$ for which there is an isomorphism $\phi\colon H_1\to H_2$ with $\phi^*\mu_{2}=\mu_{1}$ and $\phi^* p_{2}=p_{1}$ the associated pairs of homogeneous polynomials $(f_1,p_1)$ and $(f_2,p_2)$ for fixed bases $E_1$ and $E_2$ of $H_{1}$ and $H_{2}$, respectively, satisfy
		\[f_1(v)=f_2(A^\top\cdot v)\quad\text{and}\quad p_1(v)=p_2(A^\top\cdot v) \]
		for all $v\in \C^n$, $n=\mathrm{rank}(H_{G_1})$. Here $A\in\C^{n\times n}$ is the matrix that transforms the basis $E_2$ into $\phi_*E_1$. In particular, $A\in \mathrm{GL}(n,\Z)\subseteq\widetilde{\mathrm{SL}}(n,\C)$, where
		\[\widetilde{\mathrm{SL}}(n,\C)=\{A\in\mathrm{GL}(n,\C)\mid \det(A)=\pm 1 \}. \]
		It follows that if two graphs are $1$-equivalent, then the associated pairs of homogeneous polynomials lie in the same orbit of the action of $\widetilde{\mathrm{SL}}(n,\C)$.
		
		For a given pair $(f,p)$ of binary homogeneous polynomials, where $f$ has degree 3 and $p$ has degree 1, by \cite[Section 3.1]{Sc97} there are invariants, that is, homogeneous polynomials in the coefficients of $f$ and $p$ that are invariant under the action of $\widetilde{\mathrm{SL}}(2,\C)$, given by $D(f)$, $R(f,p)^2$, $I(f\cdot p)$ and $J(f\cdot p)$, where
		\begin{itemize}
			\item $D(f)$ is the discriminant of $f$,
			\item $R(f,p)$ is the resultant of the pair $(f,p)$, and
			\item $I(f\cdot p)$ and $J(f\cdot p)$ are binary quartic polynomials.
		\end{itemize}
		For the precise definition of these invariants we refer to \cite[Section 3.1]{Sc97}. In our case these are given as follows:
		\begin{align*}
			D(f)=&\alpha_1^2\alpha_2^2+\alpha_1^2\alpha_3^2+\alpha_2^2\alpha_3^2-2\alpha_1^2\alpha_2\alpha_3-2\alpha_1\alpha_2^2\alpha_3-2\alpha_1\alpha_2\alpha_3^2,\\
			R(f,p)=&(\alpha_1-\alpha_3)(\alpha_2-\alpha_3)(\alpha_2-\alpha_1)(\alpha_1+\alpha_2+\alpha_3),\\
			I(f\cdot p)=&\alpha_1^2\alpha_2^2+\alpha_1^2\alpha_3^2+\alpha_2^2\alpha_3^2-\alpha_1^2\alpha_2\alpha_3-\alpha_1\alpha_2^2\alpha_3-\alpha_1\alpha_2\alpha_3^2,\\
			J(f\cdot p)=&-\alpha_1^4\alpha_2^2-\alpha_1^2\alpha_2^4-\alpha_1^4\alpha_3^2-\alpha_1^2\alpha_3^4-\alpha_2^4\alpha_3^2-\alpha_2^2\alpha_3^4+2\alpha_1^4\alpha_2\alpha_3+2\alpha_1\alpha_2^4\alpha_3+2\alpha_1\alpha_2\alpha_3^4\\
			&+\alpha_1^3\alpha_2^2\alpha_3+\alpha_1^3\alpha_2\alpha_3^2+\alpha_1^2\alpha_2^3\alpha_3+\alpha_1\alpha_2^3\alpha_3^2+\alpha_1^2\alpha_2\alpha_3^3+\alpha_1\alpha_2^2\alpha_3^3-6\alpha_1^2\alpha_2^2\alpha_3^2.
		\end{align*}
	
		Note that $D(f)$, $R(f,p)^2$, $I(f\cdot p)$ and $J(f\cdot p)$ are symmetric when viewed as polynomials in $(\alpha_1,\alpha_2,\alpha_3)$, hence we can express them in terms of the elementary symmetric polynomials
		\begin{align*}
			\sigma_1&=\alpha_1+\alpha_2+\alpha_3,\\
			\sigma_2&=\alpha_1\alpha_2+\alpha_1\alpha_3+\alpha_2\alpha_3,\\
			\sigma_3&=\alpha_1\alpha_2\alpha_3,
		\end{align*}
		and we obtain
		\begin{align*}
			D(f)&=\sigma_2^2-4\sigma_1\sigma_3,\\
			R(f,p)^2&=\sigma_1^2(\sigma_1^2\sigma_2^2-4\sigma_2^3-4\sigma_1^3\sigma_3+18\sigma_1\sigma_2\sigma_3-27\sigma_3^2),\\
			I(f\cdot p)&=\sigma_2^2-3\sigma_1\sigma_3,\\
			J(f\cdot p)&=-\sigma_1^2\sigma_2^2+2\sigma_2^3+4\sigma_1^3\sigma_3-9\sigma_1\sigma_2\sigma_3.
		\end{align*}
		Hence, the values of $\sigma_2^2$, in particular the value of $\sigma_2$ up to sign, and $\sigma_1\sigma_3$ are determined by $D(f)$ and $I(f\cdot p)$. We distinguish two cases.
		
		\begin{enumerate}
			\item[\textbf{Case 1.}] Assume that at least one of $D(f)$ and $I(f\cdot p)$ is non-zero. We now show that there are at most two triples $(\alpha_1,\alpha_2,\alpha_3)$ (up to permutation and simultaneous multiplication by $(-1)$) with invariants given by $(D(f),R(f,p)^2,I(f\cdot p),J(f\cdot p))$.
				
			Indeed, from the expression for $J(f\cdot p)$ we obtain
			\[(4\sigma_1\sigma_3-\sigma_2^2)\sigma_1^2+\sigma_2(2\sigma_2^2-9\sigma_1\sigma_3)-J(f\cdot p)=0. \]
			Since the values of $\sigma_2^2$ and $\sigma_1\sigma_3$ are determined by $D(f)$ and $I(f\cdot p)$, we consider these values as fixed and obtain a quadratic equation in $\sigma_1$ with no linear term. Hence, if $D(f)\neq0$, then for every choice of sign for $\sigma_2$, we obtain, up to sign, at most one solution for $\sigma_1$. If $D(f)=0$ and $I(f\cdot p)\neq 0$, and thus $\sigma_2^2,\sigma_1\sigma_3\neq 0$, then from the expression for $R(f,p)^2$ we obtain
			\[(2\sigma_1\sigma_2\sigma_3)\sigma_1^2-(27\sigma_1^2\sigma_3^2+R(f,p)^2)=0 \]
			and as before, every choice of sign for $\sigma_2$ uniquely determines $\sigma_1$ up to sign. The values for $\alpha_1,\alpha_2,\alpha_3$ are then obtained as the solutions of the equation
			\[y^3-\sigma_1 y^2+\sigma_2y-\sigma_3=0 \]
			and choosing a different sign for $\sigma_1$ (and thus also for $\sigma_3$) results in a simultaneous change of sign for the $\alpha_i$.
			
			\item[\textbf{Case 2.}] Assume that $D(f)$ and $I(f\cdot p)$ both vanish. We now show that in this case at least two of the $\alpha_i$ vanish and the third one can be determined up to sign from the $\widetilde{\mathrm{SL}}(2,\C)$-orbit of $(f,p)$.
				
			Indeed, we have $\sigma_2^2=\sigma_1\sigma_3=0$, which implies $\sigma_1=0$ or $\sigma_3=0$. If $\sigma_1=0$, then $\alpha_3=-\alpha_1-\alpha_2$, so
			\[0=\sigma_2=-\alpha_1^2-\alpha_1\alpha_2-\alpha_2^2=-\frac{1}{2}(\alpha_1^2+\alpha_2^2+(\alpha_1+\alpha_2)^2), \]
			which implies $\alpha_1=\alpha_2=0$ and hence also $\alpha_3=0$. If $\sigma_3=0$, then one $\alpha_i$, say $\alpha_3$, vanishes. Then $\sigma_2=\alpha_1\alpha_2$, hence also one of $\alpha_1,\alpha_2$, say $\alpha_2$, vanishes.
			
			To determine the value of $\alpha_1$ under the assumption $\alpha_2=\alpha_3=0$, first note that $\alpha_1$ vanishes if and only if $f$ is trivial. If $f$ is non-trivial, we have that $f$ divides $p^3$ with quotient $\alpha_1^2$. In particular, the value of $\alpha_1$ can be determined up to sign from the $\widetilde{\mathrm{SL}}(2,\C)$-orbit of $(f,p)$.
		\end{enumerate}
		Thus, we have shown that for any pair of binary homogeneous polynomials $(f,p)$, where $f$ has degree $3$ and $p$ has degree $1$, there are at most two triples $(\alpha_1,\alpha_2,\alpha_3)$ (up to permutation and simultaneous multiplication by $(-1)$) whose associated pairs of polynomials are in the $\widetilde{\mathrm{SL}}(2,\C)$-orbit of $(f,p)$.
		
		Since in the spin case $(\mu_G^k,p_G^k)=\frac{\lambda_k}{4}(\mu_G,p_G)$ holds, this result carries over to $k$-equivalence for all $k$.
	\end{proof}
	
	\begin{remark}
		\label{R:RED_FORMS2}
		We conjecture that Proposition \ref{P:RED_FORMS2} can be improved to give a positive answer to Question \ref{Q:RED_EQUIV=ISOM?} for reduced graphs of rank 2. In fact, there are many special cases where for given values $(D,R^2,I,J)\in\Z^4\setminus\{0\}$ in the proof of Proposition \ref{P:RED_FORMS2} there exists only one triple $(\alpha_1,\alpha_2,\alpha_3)\in\Z^3$ with $D(f)=D$, $R(f,p)^2=R^2$, $I(f\cdot p)=I$ and $J(f\cdot p)=J$, see \cite[Remark 6.4.5]{Re22}. As a result, we obtain an affirmative answer to Question \ref{Q:RED_EQUIV=ISOM?} for graphs defined by these values $(\alpha_1,\alpha_2,\alpha_3)$. These cases include the following:
		\begin{enumerate}
			\item $\alpha_1+\alpha_2+\alpha_3=0$,
			\item $\alpha_1\alpha_2+\alpha_1\alpha_3+\alpha_2\alpha_3=0$,
			\item One $\alpha_i$ vanishes,
			\item Two of the $\alpha_i$ equal $-1$.
			\item $\alpha_1,\alpha_2,\alpha_3\in[-1000,1000]$.
		\end{enumerate}
		
		On the other hand, the strategy of the proof of Proposition \ref{P:RED_FORMS2} cannot be extended to provide a proof of Question \ref{Q:RED_EQUIV=ISOM?} for reduced graphs of rank 2. In fact, there are triples $(\alpha_1,\alpha_2,\alpha_3)$ and $(\alpha_1^\prime,\alpha_2^\prime,\alpha_3^\prime)$, which are not related via a permutation or simultaneous multiplication by $(-1)$, but whose invariants $D$, $R$, $I$ and $J$ are all the same, see \cite[Remark 6.4.5]{Re22}. This occurs for example for the triples
		\[(4,15,30)\quad\text{ and }(-6,-5,60). \]
		By bringing the associated homogeneous polynomials in a canonical form as in \cite[Section 3.4]{Sc97} one sees that the corresponding systems of invariants are not equivalent, thus, they do not provide counterexamples to Question \ref{Q:RED_EQUIV=ISOM?}.
	\end{remark}
	\begin{remark}
		The proof of Proposition \ref{P:GRAPH_RED2} in fact provides an algorithm to decide whether a given closed, simply-connected spin 6-manifold $M$ with torsion-free homology and $b_2(M)=2$ can be constructed via an algebraic plumbing graph, see \cite[Algorithm E.2]{Re22}.
	\end{remark}

	\subsection{Manifolds Obtained from Algebraic Plumbing Graphs}
	\label{SS:ALG_PLUMBING_APPLICATIONS}
	
	In this section we give examples of manifolds that can or cannot be obtained from an algebraic plumbing graph.
	
	Given $p,q\in\N$ and a set $\mathcal{B}_q$ containing oriented diffeomorphism classes of manifolds of dimension $q$, we define $\mathcal{T}_{p+q-1}(\mathcal{B}_q)$ as the set containing all diffeomorphism classes of manifolds $M_G$, where $G=(U,V,E,\pi,\delta)$ is a geometric plumbing graph with simply-connected connected components, so that $B_v\cong S^p$ and $\pi_v$ is trivial for all $v\in V$, and $B_u\in \mathcal{B}_q$ for all $u\in U$.
	
	We will consider the case where $p=3$ and $q=4$, and where $\mathcal{B}_4$ consists of all known examples of 4-manifolds that admit a core metric, i.e.\
	\[\mathcal{B}_4=\left\{\#_{i=1}^k\varepsilon_i\C P^2 \;\middle|\; k\in\N_0, \varepsilon_i\in\{\pm 1\} \right\}. \]
	Then $\mathcal{T}_6(\mathcal{B}_4)$ consists of all manifolds on which we can construct a core metric using Theorem~\ref{T:PLUMBING} with all known examples of 4-manifolds that admit a core metric. By construction we have $M_{\overline{G}}\in\mathcal{T}_6(\mathcal{B}_4)$ for any algebraic plumbing graph $G$ with simply-connected connected components. We will now show that also the converse holds.
	\begin{proposition}
		\label{P:T6B4_RED}
		Let $M\in\mathcal{T}_6(\mathcal{B}_4)$. Then there is an algebraic plumbing graph $G$ so that $M\cong M_{\overline{G}}$.
	\end{proposition}
	For the proof we need the following lemma.
	\begin{lemma}
		\label{L:S2_BUNDLES_CP2}
		For $\alpha\in\Z$ let $G$ be the reduced graph below.
		
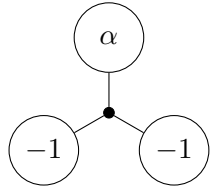
\begin{center}
			\begin{tikzpicture}
				\begin{scope}[every node/.style={circle,draw,minimum height=2.35em}]
					\node(U11) at (4,1){$\alpha$};
					\node(U21) at (3.133975,-0.5){$-1$};
					\node(U31) at (4.866025,-0.5){$-1$};
					\node[Bullet](V) at (4,0){};
				\end{scope}
				\path[-](V) edge (U11);
				\path[-](V) edge (U21);
				\path[-](V) edge (U31);
			\end{tikzpicture}
		\end{center}
		Then $M_{\overline{G}}$ is diffeomorphic to the total space of the linear $S^2$-bundle over $\C P^2$ corresponding to $(\alpha,1)$ in Remark \ref{R:S2_B}.
	\end{lemma}
	\begin{proof}
		The group $H_G$ is generated by the elements $e_1=u_1^0-u_2^0$ and $e_2=u_3^0-u_2^0$, where we denote by $u_1$ the vertex labeled by $\alpha$. Then
		\begin{align*}
			\mu_G(e_1,e_1,e_1)&=\alpha+1,\\
			\mu_G(e_1,e_1,e_2)&=1,\\
			\mu_G(e_1,e_2,e_2)&=1,\\
			\mu_G(e_2,e_2,e_2)&=0,\\
			p_G(e_1)&=4(\alpha+1),\\
			p_G(e_2)&=0.
		\end{align*}
		Now the claim follows from \cite[Corollary 5.8]{Re23} and Theorem \ref{T:JUPP}.
	\end{proof}
	\begin{proof}[Proof of Proposition \ref{P:T6B4_RED}.]
		Let $G^\prime$ be a geometric plumbing graph so that $M\cong M_{G^\prime}$. We have that any linear sphere bundle over a connected sum of $4$-manifolds is obtained by the fiber connected sum of bundles over each summand as the gluing area of the connected sum, which is diffeomorphic to $S^3$, does not admit non-trivial vector bundles. Hence, we can apply Proposition \ref{P:GRAPH_RED2} in reverse iteratively to split all base manifolds that are connected sums into their summands. In this way we modify $G^\prime$ so that all $B_u$ are given by $S^4$ or $\pm\C P^2$.
		
		Further, by applying Proposition \ref{P:GRAPH_RED_EDGE}, we can achieve that $\delta(e)=1$ for all $e\in E$. Now we turn $G^\prime$ into an algebraic plumbing graph $G$ by reversing the construction in Subsection \ref{SS:ALG_PLUMBING}, except for vertices $u$ for which $\pi_u$ is an $S^2$-bundle over $\pm\C P^2$ with spin total space, i.e.\ those corresponding to $(\alpha,1)\in\Z\times \{0,1\}$ in Remark \ref{R:S2_B}, as these are the only linear $S^2$-bundles over $S^4$ and $\pm\C P^2$ that do not arise in the construction of Subsection \ref{SS:ALG_PLUMBING}. Every such vertex $u$ gets replaced by a piece of the form as in Lemma \ref{L:S2_BUNDLES_CP2}, multiplied by $\pm1$, where all edges connected to $u$ get connected to the vertex labeled by $\alpha$. By using Corollary \ref{C:MG_COHOM_6KDIM} we now obtain that $M_{\overline{G}}$ has the same invariants as $M$, hence they are diffeomorphic.
	\end{proof}
	\begin{remark}
		\label{R:NEW_EX}
		The bundles in Lemma \ref{L:S2_BUNDLES_CP2}, together with connected sums of $S^2$-bundles over $S^4$, are the only known infinite families of closed, simply-connected spin 6-manifolds $M$ with $b_2(M)=2$ that admit a metric of positive Ricci curvature, see \cite[Theorem C and Subsection 5.1]{Re23}. Hence, Proposition \ref{P:RED_FORMS2} and items (3) and (4) of Remark \ref{R:RED_FORMS2} yield an infinite number of new examples of such manifolds and therefore an infinite number of new examples of 6-manifolds with a metric of positive Ricci curvature and $b_2=2$. To the best of our knowledge, the corresponding manifolds in dimension $6k$, which are $(2k-1)$-connected, are also new examples of manifolds with a metric of positive Ricci curvature.
	\end{remark}
	\begin{remark}
		It is worth noting that not all closed, simply-connected $6$-manifolds with torsion-free homology can be constructed via an algebraic plumbing graph. For graphs of rank 1 this directly follows from Proposition \ref{P:GRAPH_RED1} together with Theorem \ref{T:JUPP}. For graphs of rank 2 this can be seen by considering the normal forms of \cite[Section 3.4]{Sc97}. This has been analyzed in detail in \cite[Proposition 6.5.4]{Re22}.
				
		Further, one can show that any manifold that is diffeomorphic to the total space of a linear $S^{2k}$-bundle over $S^{2k}\times S^{2k}$ cannot be constructed via an algebraic plumbing graph, see \cite[Proposition 6.5.7]{Re22}. These manifolds are of special interest as it is known that they admit metrics of positive Ricci curvature for all $k$ and core metrics if $k\geq2$. This shows that not all manifolds that admit a metric of positive Ricci curvature can be constructed via algebraic plumbing graphs.
	\end{remark}
	
	Next we consider homotopy $\C P^3$'s. A closed manifold is a \emph{homotopy $\C P^3$}, if it is homotopy equivalent to $\C P^3$. In terms of invariants, a closed manifold $M$ is a homotopy $\C P^3$ if and only if
	\begin{itemize}
		\item $M$ is a simply-connected 6-manifold with torsion-free homology,
		\item $b_3(M)=0$ and $b_2(M)=1$,
		\item $\mu_M(x,x,x)=1$ for a generator $x$ of $H^2(M)$, and
		\item $w_2(M)=0$.
	\end{itemize}
	This follows for example from the homotopy classification for simply-connected 6-manifolds by Zhubr \cite{Zh00}. We note that the homotopy classification given by Wall \cite{Wa66} and Jupp \cite{Ju73} is erroneous, cf.\ \cite[5.14]{Zh00} or \cite[Remark 2]{OV95}. It can also be seen directly as follows: Let $M$ be a simply-connected manifold whose cohomology ring is isomorphic to that of $\C P^n$ and let $x\in H^2(M)$ be a generator. The element $x\in H^2(M)$ defines a map
	\[f_x\colon M\to K(\Z,2)=\C P^\infty \]
	with the property that the induced map on cohomology maps a generator of $H^*(\C P^\infty)$ to $x$. By the cellular approximation theorem we can assume that $f_x$ is cellular, so the image is contained in $\C P^n\subseteq \C P^\infty$. Then $f_x\colon M\to \C P^n$ induces an isomorphism on cohomology and, since $M$ is assumed to be simply-connected, it is therefore a homotopy equivalence by the Hurewicz theorem and Whitehead's theorem. Note that we did not use the assumption on $w_2$. In fact, since Stiefel-Whitney classes are preserved under homotopy equivalences, it follows that $w_2(M)$ vanishes if and only if $n$ is odd (alternatively, this can also be seen by using Wu classes).
	
	From Theorem \ref{T:JUPP} it follows that there is an infinite family of homotopy $\C P^3$'s whose diffeomorphism types are distinguished by $p_1$, which can take any value congruent to $4\mod 24$ on the generator $x$. By Proposition \ref{P:RED_FORMS1}, the only homotopy $\C P^3$ that can be constructed via an algebraic plumbing graph is the standard $\C P^3$ as only the graph \begin{tikzpicture}[scale=0.5,, transform shape]
			\begin{scope}[every node/.style={circle,draw}]
				\node (U) at (0,0) {$\pm 1$};
			\end{scope}
	\end{tikzpicture}, which defines the standard $\C P^3$, can produce the cohomology ring of $\C P^3$. However, we have the following result:
	
	\begin{proposition}
		\label{P:HOM_CP3}
		There is an infinite family $M_i$, $i\in\N_0$, of pairwise non-diffeomorphic homotopy $\C P^3$'s, such that for each $i\in\N_0$ there is a closed, simply-connected 6-manifold $N_i$ so that $M_i\#N_i\in\mathcal{T}_6(\mathcal{B}_4)$. In particular, $M_i\# N_i$ admits a core metric.
	\end{proposition}
	\begin{proof}
		Let $G$ be the graph
		
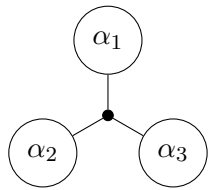
\begin{center}
			\begin{tikzpicture}
				\begin{scope}[every node/.style={circle,draw,minimum height=2.35em}]
					\node(U11) at (4,1){$\alpha_1$};
					\node(U21) at (3.133975,-0.5){$\alpha_2$};
					\node(U31) at (4.866025,-0.5){$\alpha_3$};
					\node[Bullet](V) at (4,0){};
				\end{scope}
				\path[-](V) edge (U11);
				\path[-](V) edge (U21);
				\path[-](V) edge (U31);
			\end{tikzpicture}
		\end{center}
		with
		\[\alpha_1=(2i+1)(i+1),\quad \alpha_2=(2i+1)i,\quad \alpha_3=\frac{i(i+1)}{2}. \]
		Denote the vertices labeled by $\alpha_i$ by $u_i$. Then a basis for $H_G$ is given by $e_1=u_1^0+u_2^0-2u_3^0$, $e_2=i u_1^0+(i+1)u_2^0-(2i+1)u_3^0$, and we have
		\begin{align*}
			\mu_G(x_1,x_1,x_1)&=1,\\
			\mu_G(x_1,x_1,x_2)&=0,\\
			\mu_G(x_1,x_2,x_2)&=0,\\
			\mu_G(x_2,x_2,x_2)&=\frac{i(i+1)(2i+1)}{2},\\
			p_G(x_1)&=4+24\frac{i(i+1)}{2},\\
			p_G(x_2)&=4\frac{i(i+1)(2i+1)}{2}+24\frac{i(i+1)(2i+1)}{6}.
		\end{align*}
		Hence, $M_{\overline{G}}$ is diffeomorphic to $M_i\# N_i$ if we define $M_i$ and $N_i$ as the unique closed, oriented simply-connected spin 6-manifolds with
		\begin{itemize}
			\item $b_3(M_i)=b_3(N_i)=0$,
			\item $b_2(M_i)=b_2(N_i)=1$,
			\item $\mu_{M_i}(x,x,x)=1$ for a generator $x$ of $H^2(M_i)$,
			\item $p_1(M_i)(x)=4+24\frac{i(i+1)}{2}$,
			\item $\mu_{N_i}(y,y,y)=\frac{i(i+1)(2i+1)}{2}$ for a generator $y$ of $H^2(N_i)$, and
			\item $p_1(N_i)(y)=4\frac{i(i+1)(2i+1)}{2}+24\frac{i(i+1)(2i+1)}{6}$.
		\end{itemize}
		The manifold $M_i$ is a homotopy $\C P^3$, and for different values of $i$ we obtain different values for the divisibility of the first Pontryagin class, hence all $M_i$ are pairwise non-diffeomorphic.
	\end{proof}

	Finally, we consider the question, if a given manifold $M$ can be decomposed into a connected sum $M=M_1\# M_2$, where $M_1,M_2\not\cong \Sigma^{6k}$ for any homotopy sphere $\Sigma^{6k}$. To analyze this on the level of cohomology let $H$ be a finitely generated free abelian group with a symmetric trilinear form $\mu\colon H\times H\times H\to\Z$. Given a subspace $Y\subseteq H$, we say that $Y$ is a direct summand in $(H,\mu)$, if there is another subspace $Z\subseteq H$ such that $H=Y\oplus Z$ and
	\[\mu(y_1,z_1,z_2)=\mu(z_1,y_1,y_2)=0 \]
	for all $y_1,y_2\in Y$, $z_1,z_2\in Z$.
	\begin{lemma}
		\label{L:DIR_SUM_RANK}
		Let $m$ be the rank of $Y$ and let $(y_1,\dots, y_m)$ be a basis of $Y$. If $Y$ is a direct summand in $(H,\mu)$, then for any basis $(x_1,\dots,x_n)$ of $H$ the matrix
		\[
		\begin{pmatrix}
			\mu(x_1,x_1,y_1) &\cdots&\mu(x_1,x_n,y_1) &\cdots&\cdots& \mu(x_1,x_1,y_m) &\cdots&\mu(x_1,x_n,y_m)\\
			\vdots & \ddots &\vdots &\cdots&\cdots& \vdots & \ddots &\vdots\\
			\mu(x_n,x_1,y_1) & \cdots & \mu(x_n,x_n,y_1) & \cdots&\cdots&\mu(x_n,x_1,y_m) & \cdots & \mu(x_n,x_n,y_m) \\
		\end{pmatrix}
		\]
		has rank at most $m$.
	\end{lemma}
	\begin{proof}
		For each $y_i$ we have a symmetric bilinear form $\mu(\cdot,\cdot,y_i)$. Let $A_i$ denote its matrix in the basis $(x_1,\dots,x_n)$, i.e.\ the matrix we consider is given by
		\[A=\begin{pmatrix}
			A_1|\dots|A_m
		\end{pmatrix}. \]
		Since $Y$ is a direct summand, there exists a subspace $Z\subseteq H$ so that $H= Y\oplus Z$ and products between elements of $Y$ and $Z$ vanish. Let $(z_1,\dots,z_{n-m})$ be a basis of $Z$. Then a basis for $H$ is given by $(y_1,\dots,y_m,z_1,\dots,z_{n-m})$. Let $S$ denote the matrix that maps this basis to $(x_1,\dots,x_n)$. Then define
		\[A^\prime=S^\top
		\begin{pmatrix}
			A_1|\dots|A_m
		\end{pmatrix}
		\begin{pmatrix}
			S && 0\\
			&\ddots&\\
			0&&S
		\end{pmatrix}
		=\begin{pmatrix}
			A_1^\prime|\dots| A_m^\prime
		\end{pmatrix},
		\]
		where $A_i^\prime$ is the matrix of $\mu(\cdot,\cdot,y_i)$ in the basis $(y_1,\dots,y_m,z_1,\dots,z_{n-m})$. The matrix $A^\prime$ has rank at most $m$, since it has $(n-m)$ lines that vanish. Since $A^\prime$ and $A$ are obtained from each other via invertible matrices, also $A$ has rank at most $m$.
	\end{proof}
	\begin{proposition}
		\label{P:MG_IRRED}
		Let $G$ be the graph
		
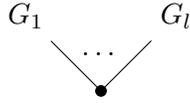
\begin{center}
			\begin{tikzpicture}
				\begin{scope}[every node/.style={circle,draw,minimum height=2.35em}]
					\node[draw=none](U1) at (-1,1){$G_1$};
					\node[draw=none](U2) at (1,1){$G_l$};
					\node[draw=none](dots) at (0,0.5){$\dots$};
					\node[Bullet](V) at (0,0){};
				\end{scope}
				\path[-](V) edge (U1);
				\path[-](V) edge (U2);
			\end{tikzpicture}
		\end{center}
		where each $G_i$ is one of 
		
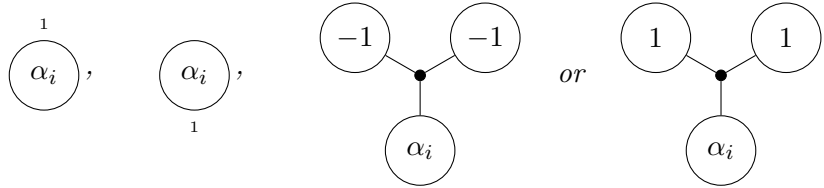
\begin{center}
			\begin{tikzpicture}
				\begin{scope}[every node/.style={circle,draw,minimum height=2.35em}]
					\node[label={[label distance=-0.25cm]above:${\scriptscriptstyle 1}$}](U1) at (0,0){$\alpha_i$};
					\node[label={[label distance=-0.25cm]below:${\scriptscriptstyle 1}$}](U2) at (2,0){$\alpha_i$};
				\end{scope}
				\begin{scope}[every node/.style={circle,draw,minimum height=2.35em}]
					\node(U11) at (5,-1){$\alpha_i$};
					\node(U21) at (4.133975,0.5){$-1$};
					\node(U31) at (5.866025,0.5){$-1$};
					\node[Bullet](V1) at (5,0){};
				\end{scope}
				\path[-](V1) edge (U11);
				\path[-](V1) edge (U21);
				\path[-](V1) edge (U31);
				\begin{scope}[every node/.style={circle,draw,minimum height=2.35em}]
					\node(U12) at (9,-1){$\alpha_i$};
					\node(U22) at (8.133975,0.5){$1$};
					\node(U32) at (9.866025,0.5){$1$};
					\node[Bullet](V2) at (9,0){};
				\end{scope}
				\path[-](V2) edge (U12);
				\path[-](V2) edge (U22);
				\path[-](V2) edge (U32);
				
				\node(comma1) at (0.6,0){,};
				\node(comma2) at (2.6,0){,};
				\node(comma3) at (7,0){or};
			\end{tikzpicture}
		\end{center}
		and the edge is connected to the vertex labeled by $\alpha_i$. If $l\geq2$, then $M_{\overline{G}^k}\not\cong M_1\# M_2$ for any $6k$-dimensional manifolds $M_1,M_2$ that are not homotopy spheres.
	\end{proposition}
	\begin{proof}
		Set $M=M_{\overline{G}^k}$ and let $M_1$, $M_2$ be $6k$-dimensional  manifolds so that $M_1\# M_2\cong M$. Then both $H^{2k}(M_1)$ and $H^{2k}(M_2)$ are direct summands in $(H^{2k}(M),\mu_M)$. By choosing the subspace of smaller rank, we obtain a direct summand $Y$ in $(H^{2k}(M),\mu_M)\cong (H_G,\mu_G^k)$ of rank $m\leq\lfloor b_{2k}(M)/2\rfloor=\lfloor(2l-1)/2\rfloor=l-1$.
		
		Fix $i$ and let $u_1$ be the vertex labeled by $\alpha_i$. If $G_i$ consists of more than one vertex, denote the others by $u_2$ and $u_3$. Then either $x_i=u_1^0$, $x_i^\prime=u_1^{\pm1}$, or $x_i=u_1^0-u_2^0$, $x_i^\prime=u_3^0-u_2^0$ is a basis of $H_{G_i}$. We define $\gamma_i=\mu^k_{G_i}(x_i,x_i^\prime,x_i^\prime)$ and $\beta_i=\mu^k_{G_i}(x_i,x_i,x_i^\prime)/\gamma_i$, i.e.\ we have the following possibilities:
		
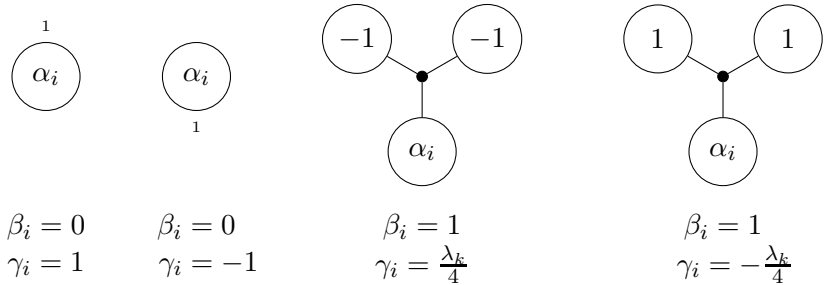
\begin{center}
			\begin{tikzpicture}
				\begin{scope}[every node/.style={circle,draw,minimum height=2.35em}]
					\node[label={[label distance=-0.25cm]above:${\scriptscriptstyle 1}$}](U1) at (0,0){$\alpha_i$};
					\node[label={[label distance=-0.25cm]below:${\scriptscriptstyle 1}$}](U2) at (2,0){$\alpha_i$};
				\end{scope}
				\begin{scope}[every node/.style={circle,draw,minimum height=2.35em}]
					\node(U11) at (5,-1){$\alpha_i$};
					\node(U21) at (4.133975,0.5){$-1$};
					\node(U31) at (5.866025,0.5){$-1$};
					\node[Bullet](V1) at (5,0){};
				\end{scope}
				\path[-](V1) edge (U11);
				\path[-](V1) edge (U21);
				\path[-](V1) edge (U31);
				\begin{scope}[every node/.style={circle,draw,minimum height=2.35em}]
					\node(U12) at (9,-1){$\alpha_i$};
					\node(U22) at (8.133975,0.5){$1$};
					\node(U32) at (9.866025,0.5){$1$};
					\node[Bullet](V2) at (9,0){};
				\end{scope}
				\path[-](V2) edge (U12);
				\path[-](V2) edge (U22);
				\path[-](V2) edge (U32);
				
				\node (C11) at (0,-2) {$\beta_i=0$};
				\node (C12) at (0,-2.5) {$\gamma_i=1$};
				\node (C11) at (2,-2) {$\beta_i=0$};
				\node (C12) at (2,-2.5) {$\,\,\,\,\,\gamma_i=-1$};
				\node (C11) at (5,-2) {$\beta_i=1$};
				\node (C12) at (5,-2.5) {$\gamma_i=\frac{\lambda_k}{4}$};
				\node (C11) at (9,-2) {$\beta_i=1$};
				\node (C12) at (9,-2.5) {$\,\,\,\,\,\gamma_i=-\frac{\lambda_k}{4}$};
			\end{tikzpicture}
		\end{center}
		
		A basis for $H_G$ is now given by $(x_1-x_l,\dots,x_{l-1}-x_l,x_1^\prime,\dots,x_l^\prime)$. Let $(y_1,\dots,y_m)$ be a basis of $Y$. Then there exist $\lambda_{i,j}$ and $\mu_{i,j}$, so that
		\[y_i=\sum\limits_{j=1}^{l-1}\lambda_{i,j}(x_j-x_l)+\sum\limits_{j=1}^{l}\mu_{i,j}x_j^{\prime}. \]
		Let $\lambda_{i,l}=-\sum_{j=1}^{l-1}\lambda_{i,j}$. Then the matrix $A_i$ in Lemma \ref{L:DIR_SUM_RANK} is given as follows:
		\[
		\begin{tikzpicture}[baseline=(current bounding box.center)]
			\matrix (m) [matrix of math nodes,nodes in empty cells,right delimiter={)},left delimiter={(} ]{
				* & &  & * & \hspace{-8ex}{\scriptstyle\gamma_1(\mu_{i,1}+\lambda_{i,1}\beta_1)} & 0 & & 0 & {\scriptstyle -\gamma_l(\mu_{i,l}+\beta_i\lambda_{i,l})} \\
				& & & & 0 & & & & \\
				& & & & & & & 0 & \\
				* & & & * & 0 & & 0 & {\scriptstyle\gamma_{l-1}(\mu_{i,l-1}+\lambda_{i,l-1}\beta_{l-1} )} & {\scriptstyle -\gamma_l(\mu_{i,l}+\beta_i\lambda_{i,l})}\\
				{\scriptstyle\gamma_1(\mu_{i,1}+\lambda_{i,1}\beta_1)} & 0 & & 0 & \gamma_1 \lambda_{i,1} & 0 & & &0\\
				0 & & & &0 & & & &\\
				& & & 0 & & & & &\\
				0 & & 0 & {\scriptstyle\gamma_{l-1}(\mu_{i,l-1}+\lambda_{i,l-1}\beta_{l-1} )}& & & & & 0\\
				{\scriptstyle-\gamma_l(\mu_{i,l}+\beta_i\lambda_{i,l})} & & & {\scriptstyle-\gamma_l(\mu_{i,l}+\beta_i\lambda_{i,l})}& 0 & & & 0 & \gamma_l \lambda_{i,l} \\
			} ;
			\draw[loosely dotted] (m-1-1)-- (m-1-4);
			\draw[loosely dotted] (m-1-1)-- (m-4-4);
			\draw[loosely dotted] (m-1-1)-- (m-4-1);
			\draw[loosely dotted] (m-1-4)-- (m-4-4);
			\draw[loosely dotted] (m-4-1)-- (m-4-4);
			\draw[loosely dotted] (m-5-1)-- (m-8-4);
			\draw[loosely dotted] (m-5-2)-- (m-5-4);
			\draw[loosely dotted] (m-5-2)-- (m-7-4);
			\draw[loosely dotted] (m-5-4)-- (m-7-4);
			\draw[loosely dotted] (m-6-1)-- (m-8-1);
			\draw[loosely dotted] (m-6-1)-- (m-8-3);
			\draw[loosely dotted] (m-8-1)-- (m-8-3);
			\draw[loosely dotted] (m-9-1)-- (m-9-4);
			\draw[loosely dotted] (m-1-5)-- (m-4-8);
			\draw[loosely dotted] (m-2-5)-- (m-4-5);
			\draw[loosely dotted] (m-2-5)-- (m-4-7);
			\draw[loosely dotted] (m-4-5)-- (m-4-7);
			\draw[loosely dotted] (m-1-6)-- (m-1-8);
			\draw[loosely dotted] (m-1-6)-- (m-3-8);
			\draw[loosely dotted] (m-1-8)-- (m-3-8);
			\draw[loosely dotted] (m-1-9)-- (m-4-9);
			\draw[loosely dotted] (m-5-5)-- (m-9-9);
			\draw[loosely dotted] (m-6-5)-- (m-9-5);
			\draw[loosely dotted] (m-6-5)-- (m-9-8);
			\draw[loosely dotted] (m-9-5)-- (m-9-8);
			\draw[loosely dotted] (m-5-6)-- (m-5-9);
			\draw[loosely dotted] (m-5-6)-- (m-8-9);
			\draw[loosely dotted] (m-5-9)-- (m-8-9);
		\end{tikzpicture}
		\]
		By Lemma \ref{L:DIR_SUM_RANK}, the rank of $A$ is at most $m\leq l-1$. Let $A_i^j$ denote the $j$-th column of $A_i$. By considering $A_i^l,\dots,A_i^{2l-1}$, it follows that either all $\lambda_{i,j}$ vanish, or there is $j_0\in\{l,\dots,2l-1\}$, so that $A_i^{j_0}=0$ for all $i$. We first show that the second case implies the first one.
		
		By symmetry reasons we can assume that $j_0=2l-1$, i.e.\ $\lambda_{i,l}=\mu_{i,l}=0$ for all $i$. In particular, $\sum_{j=1}^{l-1}\lambda_{i,j}=0$. Now consider the matrix
		\[
		\begin{pmatrix}
			A_1^l|\dots|A_1^{2l-2}|\dots|A_m^l|\dots|A_m^{2l-2}
		\end{pmatrix}
		\]
		which then also has rank at most $m$. By elementary row and column operations we bring it into the following form:
		\[
		A^\prime=\begin{tikzpicture}[baseline=(current bounding box.center)]
			\matrix (m) [matrix of math nodes,nodes in empty cells,right delimiter={)},left delimiter={(} ]{
				\mu_{1,1} & 0 & & 0 &&\mu_{2,1} & 0 & & 0 &&\quad\quad&& \mu_{m,1} & 0 & & 0\\
				0 & & & && 0 & & & &&&& 0 & & &\\
				& & & 0 && & & & 0 &&&& & & & 0\\
				0 & & 0 & \mu_{1,l-1} && 0 & & 0 & \mu_{2,l-1} &&&& 0 & & 0 & \mu_{m,l-1}\\
				\lambda_{1,1} & 0 & & 0 &&  \lambda_{2,1} & 0 & & 0 &&&&  \lambda_{m,1} & 0 & & 0\\
				0 & & & && 0 & & & &&&& 0 & & &\\
				& & & 0 && & & & 0 &&&& & & & 0\\
				& & &\lambda_{1,l-1}&& & & &\lambda_{2,l-1} &&&& & & &\lambda_{m,l-1} \\
				0 & & & 0 && 0 & & & 0  &&&& 0 & & & 0\\
			} ;
			\draw[loosely dotted] (m-2-1)-- (m-4-1);
			\draw[loosely dotted] (m-6-1)-- (m-9-1);
			\draw[loosely dotted] (m-2-1)-- (m-4-3);
			\draw[loosely dotted] (m-4-1)-- (m-4-3);
			\draw[loosely dotted] (m-6-1)-- (m-9-4);
			\draw[loosely dotted] (m-9-1)-- (m-9-4);
			\draw[loosely dotted] (m-1-1)-- (m-4-4);
			\draw[loosely dotted] (m-5-1)-- (m-8-4);
			\draw[loosely dotted] (m-1-2)-- (m-3-4);
			\draw[loosely dotted] (m-1-2)-- (m-1-4);
			\draw[loosely dotted] (m-1-4)-- (m-3-4);
			\draw[loosely dotted] (m-5-2)-- (m-5-4);
			\draw[loosely dotted] (m-5-2)-- (m-7-4);
			\draw[loosely dotted] (m-5-4)-- (m-7-4);
			
			\draw[loosely dotted] (m-2-6)-- (m-4-6);
			\draw[loosely dotted] (m-6-6)-- (m-9-6);
			\draw[loosely dotted] (m-2-6)-- (m-4-8);
			\draw[loosely dotted] (m-4-6)-- (m-4-8);
			\draw[loosely dotted] (m-6-6)-- (m-9-9);
			\draw[loosely dotted] (m-9-6)-- (m-9-9);
			\draw[loosely dotted] (m-1-6)-- (m-4-9);
			\draw[loosely dotted] (m-5-6)-- (m-8-9);
			\draw[loosely dotted] (m-1-7)-- (m-3-9);
			\draw[loosely dotted] (m-1-7)-- (m-1-9);
			\draw[loosely dotted] (m-1-9)-- (m-3-9);
			\draw[loosely dotted] (m-5-7)-- (m-5-9);
			\draw[loosely dotted] (m-5-7)-- (m-7-9);
			\draw[loosely dotted] (m-5-9)-- (m-7-9);
			
			\draw[loosely dotted] (m-2-13)-- (m-4-13);
			\draw[loosely dotted] (m-6-13)-- (m-9-13);
			\draw[loosely dotted] (m-2-13)-- (m-4-15);
			\draw[loosely dotted] (m-4-13)-- (m-4-15);
			\draw[loosely dotted] (m-6-13)-- (m-9-16);
			\draw[loosely dotted] (m-9-13)-- (m-9-16);
			\draw[loosely dotted] (m-1-13)-- (m-4-16);
			\draw[loosely dotted] (m-5-13)-- (m-8-16);
			\draw[loosely dotted] (m-1-14)-- (m-3-16);
			\draw[loosely dotted] (m-1-14)-- (m-1-16);
			\draw[loosely dotted] (m-1-16)-- (m-3-16);
			\draw[loosely dotted] (m-5-14)-- (m-5-16);
			\draw[loosely dotted] (m-5-14)-- (m-7-16);
			\draw[loosely dotted] (m-5-16)-- (m-7-16);
			
			\draw (m-1-5)-- (m-9-5);
			\draw (m-1-10)-- (m-9-10);
			\draw (m-1-12)-- (m-9-12);
			\draw[loosely dotted] (m-5-10)-- (m-5-12);
		\end{tikzpicture}
		\]
		Then the vectors
		\[w_i=\begin{pmatrix}
			\mu_{i,1}\\\vdots\\\mu_{i,l-1}\\\lambda_{i,1}\\\vdots\\ \lambda_{i,l-1}\\0
		\end{pmatrix} \]
		lie in the space generated by the columns of $A^\prime$. Since $\mu_{i,l}=\lambda_{i,l}=0$, and since $(y_1,\dots,y_m)$ is a basis of $Y$, it follows that $(w_1,\dots,w_m)$ is linearly independent. Hence, since $A^\prime$ has rank at most $m$, $(w_1,\dots,w_m)$ is a basis (over $\Q$) of the space generated by the columns of $A^\prime$. Thus, since $\sum_{j=1}^{l-1}\lambda_{i,j}=0$, every $v=(v_1,\dots,v_{2l-1})$ in the space generated by the columns of $A^\prime$ satisfies
		\[\sum\limits_{i=l}^{2l-2}v_i=0. \]
		In particular, $\lambda_{i,j}=0$ for all $i,j$.
		
		Hence we can assume that all $\lambda_{i,j}$ vanish. Then define
		\[w_i^\prime=\sum\limits_{j=1}^{l-1}\frac{1}{\gamma_j } A_i^j=\begin{pmatrix}
			*\\\vdots\\ *\\\mu_{i,1}\\\vdots\\\mu_{i,l-1}\\-(l-1)\mu_{i,l}
		\end{pmatrix}. \]
		Since $(y_1,\dots,y_m)$ is a basis of $Y$, the elements $w^\prime_1,\dots,w^\prime_m$ are linearly independent and hence a basis over $\Q$ of the space generated by the columns of $A$. But they also form a linearly independent set together with the $w_i$ (which also lie in the space generated by the columns of $A$ since $\lambda_{i,j}=0$), which is a contradiction unless $m=0$.
		
		Thus, one of the $M_i$, say $M_1$, has vanishing cohomology in degree $2k$, so $H^{2k}(M)=H^{2k}(M_2)$. By Corollary \ref{C:2k_SPHERE_BDL_COHOM} and Theorem \ref{T:MG_COHOM}, any $x\in H^2(M)$ is non-trivial if and only if $x^k\in H^{2k}(M)$ is non-trivial and all other elements in $H^*(M)$ are obtained from powers of elements in $H^2(M)$ and multiplication by $a\in H^{2k}(M)$. Thus, $H^*(M_2)=H^*(M)$ and $M_1$ has non-trivial cohomology groups only in degrees $0$ and $6k$. Hence, $M_1$ is a homology sphere and since it is simply-connected it is therefore a homotopy sphere.
	\end{proof}
	\begin{remark}
		Since there do not exist exotic spheres in dimensions $6$ and $12$, we actually obtain for $k=1,2$ that $M_{\overline{G}^k}$ does not split as a connected sum $M_1\# M_2$ for any $M_1,M_2$ which are not standard spheres.
	\end{remark}

	For the proof of Theorem \ref{T:NEW_EX_B2_LARGE} we need the following lemma to conclude that we obtain new examples.
	\begin{lemma}
		\label{L:GROMOV_BETTI}
		For any $n\in\N$ there exists a constant $C(n)$ so that any closed $n$-dimensional manifold $M$ that has the structure of a homogeneous space, a biquotient, a cohomogeneity one manifold or a Fano variety satisfies Gromov's Betti number bound
		\[\sum\limits_{i=0}^{n}b_i(M)\leq C(n). \]
	\end{lemma}
	\begin{proof}
		Every homogeneous space and, more general, every biquotient admits a metric of non-negative sectional curvature. This is well-known and follows from the fact that for a compact Lie group $G$ any biinvariant metric has non-negative sectional curvature and, by O'Neill's formulas for Riemannian submersions, this property descends to any biquotient $G\bq H$. Further, by \cite[Theorem A]{ST01}, any closed cohomogeneity one manifold admits a metric of almost non-negative sectional curvature. Thus, for these spaces the claim follows from Gromov's Betti number theorem \cite{Gr81}. For Fano varieties the claim follows from the fact that, by \cite[Theorem 2.1]{FH12}, there exist only finitely many diffeomorphism types in each dimension.
	\end{proof}
	
	\begin{proof}[Proof of Theorem \ref{T:NEW_EX_B2_LARGE}]
		Let $M=M_{\overline{G}^k}$ with $G$ as in Proposition \ref{P:MG_IRRED}, where we choose each subgraph $G_i$ to be spin, i.e.\ one of the third or fourth option. Then, by Proposition \ref{P:MG_IRRED}, $M$ does not split non-trivially as a connected sum and we have that $M$ is $(2k-1)$-connected with $b_{2k}(M)=2l-1$. Further, $M$ admits a core metric by Theorem \ref{T:ALG_PL_GR_CORE1} and for $l$ sufficiently large it is not diffeomorphic to a homogeneous space, a biquotient, a cohomogeneity one manifold or a Fano variety by Proposition \ref{L:GROMOV_BETTI}.
		
		To show that it is not diffeomorphic to the total space of a linear sphere bundle, note that, for every $\lambda_1,\lambda_2\in\Z$ we have $\lambda_1 x_1+\lambda_2 x_2-(\lambda_1+\lambda_2)x_3\in H_G$ (where we use the notation $x_i,x_i^\prime$ as in the proof of Proposition \ref{P:MG_IRRED}). Hence, there exist $\lambda_1,\lambda_2\in\Z$, with one of $\lambda_i\neq0$, so that
		\[p_G^k(\lambda_1 x_1+\lambda_2 x_2-(\lambda_1+\lambda_2)x_3)=0. \]
		Then, for $i$ with $\lambda_i\neq0$, we have
		\[\mu_G^k(\lambda_1 x_1+\lambda_2 x_2-(\lambda_1+\lambda_2)x_3,x_i^{\prime},x_i^{\prime})=\lambda_i\gamma_i\neq0. \]
		Hence, if $p_G^k\neq0$, then $M$ is not diffeomorphic to the total space of a linear sphere bundle by Lemma \ref{L:SPHERE_BDL_OBSTR}.
		
		To obtain a non-spin manifold, we can alternatively choose some of the $G_i$, say $m$, to be non-spin, i.e.\ one of the first and second options. Then $M$ is merely simply-connected, apart from that the same conclusions hold as in the spin case, except that the fact that $\mu_G^k$ is non-trivial on $\ker p_G^k$ does not necessarily imply, that $M$ is not diffeomorphic to the total space of a linear sphere bundles, as in this case this follows only for linear $S^{p-1}$-bundles with $p\geq 2k+1$. Hence, suppose that $M$ is diffeomorphic to the total space $E$ of a linear $S^{p-1}$-bundle with $1<p\leq 2k$. We consider the Euler characteristic, which, by Corollary \ref{C:S2k_BDLS_S4k} and Theorem \ref{T:MG_COHOM}, is given by
		\[\chi(M)=2+2m(k-1)+(l-1+m)(k+1). \]
		Since $\chi(M)>0$, we have that $p$ is odd by Lemma \ref{L:SPHERE_BDL_OBSTR}. If $k=1$, this is a contradiction. If $k$ is even, then $\chi(M)$ is odd for $m\equiv l\mod 2$, which is a contradiction by Lemma \ref{L:SPHERE_BDL_OBSTR}.
		
		Finally, we need to determine when $p_G^k$ vanishes. We have
		\begin{align*}
			p_G^k(x_i-x_l)&=\lambda_k(\alpha_i-\alpha_l)+{2k+1\choose k}((1-\beta_i)\gamma_i-(1-\beta_l)\gamma_l)+4(\beta_i\gamma_i-\beta_l\gamma_l)\\
			&\equiv{2k+1\choose k}(\gamma_i(1-\beta_i)-\gamma_l(1-\beta_l))\mod \lambda_k
		\end{align*}
		and $p^k_G(x_i^\prime)=0$. Hence, $p^k_G\equiv0\mod\lambda_k$ if and only if $\beta_i=0$ and $\gamma_i=\gamma_l$ for all $i$, or $\beta_i=1$ for all $i$ (note that $2{2k+1\choose k}<\lambda_k$ for $k\geq 3$ and $\lambda_1=4$ and $12\mid \lambda_2$ by Remark \ref{R:LAMBDA_K}). It follows that $p_G^k=0$ if and only if $\beta_i=0$, $\gamma_i=\gamma_l$ and $\alpha_i=\alpha_l$ for all $i$, or $\beta_i=1$ and $\alpha_i+\gamma_i=\alpha_l+\gamma_l$ for all $i$.
		
		Thus, we can, for example, set $\alpha_i=0$ for all $i>1$ and $\alpha_1=j$ to define the manifold $M_j$ for all $j>1$.
	\end{proof}

	\appendix

	\section{Adjacency and Incidence Matrix of a Directed Graph}
	\label{A:GRAPHS}
	
	In this appendix we recall some well-known facts about the adjacency and incidence matrices of a directed graph. For convenience we also include self-contained proofs.
	
	Let $R$ be a commutative ring and let $G=(V,E)$ be a directed graph, where $V=\{v_1,\dots,v_n\}$ is the set of vertices and $E=\{e_1,\dots,e_m\}\subseteq V\times V$ is the set of edges. The \emph{incidence matrix} of $G$, denoted by $Q(G)$, is the $n\times m$-matrix with entries in $R$ defined by
	\[
	Q(G)_{ij}=\begin{cases}
		1,\quad&\text{ if there is }k\text{ so that }e_j=(v_i,v_k),\\
		-1,\quad&\text{ if there is }k\text{ so that }e_j=(v_k,v_i),\\
		0,\quad&\text{else.}
	\end{cases}
	\]
	\begin{lemma}
		\label{L:GRAPH_RANK_N-1}
		Suppose $G$ is connected. Then the matrix $Q(G)^\top$ has kernel generated by $(1,\dots,1)^\top$. In particular, $Q(G)$ has rank $n-1$.
	\end{lemma}
	\begin{proof}
		Let $x\in R^n$ such that $x^\top Q(G)=0$. Then, for every $e=(v_i,v_j)\in E$, we have $x_i-x_j=0$, i.e.\ $x_i=x_j$. Since $G$ is connected, for any $v_i,v_j\in V$ there is a path between $v_i$ and $v_j$, so $x_i=x_j$ for all $1\leq i,j\leq n$.
	\end{proof}
	\begin{lemma}
		\label{L:GRAPH_SURJ}
		Suppose the underlying undirected graph of $G$ is simply-connected, i.e.\ a tree. Then the matrix $Q(G)^\top$, when considered as a linear map $R^n\to R^m$, is surjective.
	\end{lemma}
	\begin{proof}
		Let $e_i\in E$. Since $G$ is simply-connected, the edge $e_i$ divides $G$ into two subtrees. Let $G^\prime$ be the subtree from which $e_i$ originates with root the vertex connected to $e_i$. Then define $x\in R^n$ by $x_j=1$ whenever $v_j$ is contained in $G^\prime$ and $x_j=0$ otherwise. Then it follows that $Q(G)^\top x$ is the $i$-th standard basis vector of $R^m$ and hence $Q(G)^\top$ is surjective.
	\end{proof}
	
	Now let $G=(U,V,E)$, be a bipartite graph, where $U=\{u_1,\dots, u_r\}$ and $V=\{v_1,\dots,v_s\}$ are the sets of vertices and $E\subseteq U\times V$ is the set of edges. Let $\delta\colon E\to\{\pm1\}$ be a labeling of the edges. The \emph{biadjacency matrix} of $G$, denoted by $B(G)$, is the $r\times s$-matrix with entries in $R$ defined by
	\[B(G)_{ij}=\begin{cases}
		\delta(e),\quad & e=(u_i,v_j)\in E,\\
		0,\quad &\text{else.}
	\end{cases}  \]
	\begin{lemma}
		\label{L:BG_RANK}
		Suppose $G$ is a tree so that no $v\in V$ is a leaf. Then $B(G)$ has full rank.
	\end{lemma}
	\begin{proof}
		We show that $B(G)$, when considered as a linear map $R^s\to R^r$, is injective. Let $x\in R^s$ with $B(G)x=0$ and let $u_i\in U$ be a leaf. Then there is a unique $v_j\in V$ so that $(u_i,v_j)\in E$, hence $x_j=0$.
		
		Now we remove all the leaves, which are vertices in $U$, and all $v\in V$ that turn into leaves by this procedure and get a subgraph $G^\prime$ that is again a tree and has no vertex in $V$ as a leaf. For all those $v_j\in V$ that get removed we already have $x_j=0$. Hence, by removing these entries from $x$, we obtain a vector $x^\prime$ with $B(G^\prime)x^\prime=0$. Hence, by induction, it follows that $x=0$.
	\end{proof}
	
	\begin{lemma}
		\label{L:BG_BASIS}
		Suppose that $G$ is a tree so that no $v\in V$ is a leaf. Then the image of $B(G)$, when considered as a map $R^s\to R^r$, is a direct summand in $R^r$.
	\end{lemma}
	\begin{proof}
		We define a vertex in $V$, say $v_1$, to be the root of $G$. For every vertex $v_j\in V$ we then pick a vertex $u_i\in U$ connected to $v_j$ that has greater distance to $v_1$ than $v_j$. We set $f(j)=i$ and obtain a function $f\colon\{1,\dots,s \}\to\{1,\dots,r\}$. Since no vertex in $V$ is a leaf, the function $f$ is well-defined and since $G$ contains no cycle, the function $f$ is injective. Denote by $e_i$ the $i$-th standard basis vector of $R^r$ and the union of the columns of $B(G)$ and the elements $e_i$ with $i\not\in\im(f)$ by $C$. We will show that $C$ is a basis of $R^r$.
		
		Let $u_\ell\in U$. If $\ell\not\in\im(f)$, then $e_\ell$ is already contained in $C$. Otherwise let $j\in\{1,\dots,s\}$ such that $f(j)=\ell$. If $j=1$, then the first column of $B(G)$ is a sum of elements in $(e_{i})_{i\not\in\im(f)}$ and $e_\ell$ (after multiplying each element by $\pm1$), in particular, $e_\ell$ is a linear combination of elements in $C$. If $j>1$, then the $j$-th column of $B(G)$ is a sum of elements in $(e_{i})_{i\not\in\im(f)}$, $e_\ell$ and an element $e_{k}$ where $u_k$ is the preceding vertex to $v_j$ (and again each element is multiplied by $\pm1$). In particular, $u_\ell$ is a linear combination of elements in $C$ and $e_k$. By induction, this shows that any $e_\ell$ is a linear combination of elements in $C$, so $C$ is a generating set for $R^r$.
		
		Since $C$ has $r=\mathrm{rank}(R^r)$ elements, it follows that $C$ is a basis of $R^r$ and $R^r$ is the direct sum of the subspace generated by the columns of $B(G)$ and the subspaces generated by the elements $e_i$ with $i\not\in\im(f)$.
	\end{proof}

	\bibliographystyle{abbrvurl}
	\bibliography{References}

\end{document}